\documentclass[12pt,a4paper]{amsart}
\usepackage[T1]{fontenc}       
\usepackage{amsfonts}          
\usepackage{amsmath}
\usepackage{amssymb}
\usepackage[colorlinks]{hyperref}

\newcommand{\R}{\mathbb{R}}    
\newcommand{\C}{\mathbb{C}}
\newcommand{\Q}{\mathbb{Q}}

\newcommand{\N}{\mathbb{N}}

\newcommand{\Z}{\mathbb{Z}}

\let\swap=\phi
\let\phi=\varphi
\let\varphi=\swap

\let\swap=\epsilon
\let\epsilon=\varepsilon
\let\varepsilon=\swap

\def\dist{\qopname\relax o{Dist}}

\def\dist{\qopname\relax o{dist}}

\def\viiva #1 #2 {\mathop{\Big/}\limits_{\!\!\!{#1}}^{\>\,{#2}}}

\usepackage{amsthm}

\swapnumbers
\theoremstyle{plain}
\newtheorem{lause}[equation]{Theorem}
\newtheorem{lem}[equation]{Lemma}
\newtheorem{prop}[equation]{Proposition}

\theoremstyle{definition}

\theoremstyle{remark}
\newtheorem{huom}[equation]{Remark}

\numberwithin{equation}{section}

\makeatletter
\@namedef{subjclassname@2010}{%
  \textup{2010} Mathematics Subject Classification}
\makeatother

\begin{document}
\author{Tuomas P. Hyt\"onen \qquad Antti V. V\"ah\"akangas}
\title[The local non-homogeneous $Tb$ theorem]{The local non-homogeneous $Tb$ theorem for vector-valued functions}
\address{Department of Mathematics and Statistics, 
Gustaf H\"allstr\"omin katu 2b, FI-00014 University of Helsinki, Finland}
\email{tuomas.hytonen@helsinki.fi}
\email{antti.vahakangas@helsinki.fi}
\thanks{
T.P.H. was supported by the Academy of Finland, grants 1130166, 1133264 and 1218148, and by the European Union through the ERC Starting grant `Analytic--probabilistic methods for borderline singular integrals'.
A.V.V. was supported by the Academy of Finland, grants 75166001 and 1134757.
}

\maketitle

\begin{abstract}
We extend the local non-homogeneous $Tb$ theorem of Nazarov, Treil and Volberg to the setting of singular integrals with operator-valued kernel that act on vector-valued functions. Here, `vector-valued' means `taking values in a function lattice with the UMD (unconditional martingale differences) property'. A similar extension (but for general UMD spaces rather than UMD lattices) of Nazarov--Treil--Volberg's global non-homogeneous $Tb$ theorem was achieved earlier by the first author, and it has found applications in the work of Mayboroda and Volberg on square-functions and rectifiability. Our local version requires several elaborations of the previous techniques, and raises new questions about the limits of the vector-valued theory.
\end{abstract}

\setcounter{tocdepth}{1}
\tableofcontents





\section{Introduction}

\subsection*{Background and motivation}
This paper is a continuation of \cite{hytonen1}, where the first author extended the `global' non-homogeneous $Tb$ theorem of Nazarov, Treil and Volberg \cite{NTV:Tb} to $L^p$ spaces of vector-valued functions. The goal of the paper at hand is to obtain a similar extension for the `local' version of Nazarov, Treil and Volberg's result \cite{NTV}.

By `local' we understand that the $Tb$ conditions involve a family  (an `accretive system') of testing functions $b_Q$, one for each cube $Q$, where $b_Q$ is only required to satisfy a non-degeneracy condition on its `own' $Q$; this contrasts with the `global' $Tb$ conditions, where a single testing function $b$ should be appropriately non-degenerate over all positions and length-scales. While the two types of $Tb$ theorems are not strictly comparable, the verification of the local conditions has turned out more approachable in a number of applications.

By `vector-valued' we understand functions taking values in a possibly infinite-dimensional Banach space $X$. It is well known that the most general class of Banach spaces in which extensions of deeper results in harmonic analysis can be hoped for consists of the spaces with the $\mathrm{UMD}$ property (unconditionality of martingale differences); see \cite{Bourgain:83,Burkholder}. The quest for vector-valued extensions of theorems in classical analysis has three types of motivation:

First, by revisiting a proof in a more general framework we can often develop new insight into the original argument; in particular, the tools available in an abstract $\mathrm{UMD}$ space often lead us into discovering new martingale structure behind the classical scene. In the present case, for example, we are led to study the $L^p$ estimates for martingale difference expansions adapted to an accretive system of functions, where mainly the Hilbert space $L^2$ theory for such expansions existed so far. While the $L^p$ theory for the `globally' adapted martingale differences was developed in \cite{hytonen1}, the local setting brings several new complications, most prominently the fact that the expansion is no longer with respect to a basis of adapted Haar functions but rather with respect to an overdetermined frame.

Second, new connections between different properties of Banach spaces are revealed, when looking for the minimal conditions under which we can run a given classical analysis argument. In particular, for vector-valued functions, there appears a subtle difference between the square function estimates for the adapted martingale difference operators and for their adjoints, and we are only able to handle the latter case under the additional assumption that our Banach space is a function lattice. Whether this assumption could be eliminated from certain key inequalities raises interesting questions for further investigation.

 Finally, the extended scope of the theorem allows for wider applications. The applications of vector-valued singular integrals in general are widespread; for the vector-valued non-homogeneous $Tb$ theorem \cite{hytonen1} in particular, we mention the work of Mayboroda and Volberg \cite{MayVol} on square functions and rectifiability, see \cite[p.~1056]{MayVol}.

Now, let us turn to a more detailed discussion of the objects of this paper.

\subsection*{Calder\'on--Zygmund operators} Let $\mu$ be a compactly supported Borel measure
on $\R^N$ which satisfies the upper bound
\begin{equation}\label{mesas}
\mu(B(x,r)) \le r^d,\qquad d\in (0,N],
\end{equation}
for any ball $B(x,r)$ of centre $x\in \R^N$ and radius $r>0$. A {\em $d$-dimensional
Calder\'on--Zygmund kernel} is a complex-valued function $K(x,y)$ of variables
$x,y\in \R^N$, $x\not=y$, such that
\begin{equation}\label{size}
|K(x,y)|\le \frac{1}{|x-y|^d}
\end{equation}
and, if $2|x-x'|\le |x-y|$, then
\begin{equation}\label{smooth}
|K(x,y)-K(x',y)|+|K(y,x)-K(y,x')|\le \frac{|x-x'|^\alpha}{|x-y|^{d+\alpha}}
\end{equation}
for some $\alpha>0$.
An operator $T$ acting on some functions is called a {\em Calder\'on--Zygmund operator} with kernel $K$ if
\begin{equation}\label{action}
Tf(x)=\int_{\R^N} K(x,y)f(y)\,d\mu(y),\qquad x\not\in \mathrm{supp}\,f.
\end{equation}

\subsection*{Testing functions}
Following \cite{NTV} we say
that a collection of functions $\{b_Q\}$ is an {\em $L^\infty$-accretive
system (supported on cubes)} if for every cube $Q$ in $\R^N$ there exists a function $b_Q$ from the system such
that
\begin{equation}\label{accre}
\mathrm{supp}\,b_Q \subset Q,\qquad \|b_Q\|_{L^\infty(\mu)}\le 1,\qquad \bigg|\int_Q b_Q\,d\mu\bigg|\ge \delta\,\mu(Q).
\end{equation}
Here the constant $\delta\in (0,1)$ is not allowed to depend on $Q$.

We say that $\{b_Q\}$ is  an $L^\infty$-accretive system for 
a Calder\'on--Zygmund operator $T$ if,
for every cube $Q$ in $\R^N$, there is a function $b_Q$ from the system such that
the conditions \eqref{accre} hold true and 
\begin{equation}\label{accreass}
\|Tb_Q\|_{L^\infty(\mu)}\le B,
\end{equation}
where $B>0$ is a constant.


\subsection*{Banach spaces}
We want to study the action of $T$ as in \eqref{action} on the Bochner space $L^p(\R^N,\mu;X)$ of functions with values in the Banach space $X$. As is well-known, even for the simplest non-trivial case where $T$ is the Hilbert transform with $d=N=1$ and $\mu$ is the Lebesgue measure, a necessary condition for the boundedness on the Bochner space is that $X$ be a $\mathrm{UMD}$ space \cite{Bourgain:83}. For the more complicated operators as described, we will need to assume some further conditions.

We will make the more restrictive assumption that $X$ is a $\mathrm{UMD}$ function lattice, i.e., $X$ is a $\mathrm{UMD}$ space whose elements are represented by functions on some measure space, and the norm of $X$ is compatible with the pointwise comparison of functions in that $|f|\leq |g|$ pointwise implies that $\|f\|_X\leq\|g\|_X$. See \cite{RdF} for more information on function lattices with the $\mathrm{UMD}$ property. We will make use of this assumption both directly, 
via Theorem \ref{haa},
and through the following consequence established by Hyt\"onen, McIntosh and Portal \cite{HMP}: such spaces satisfy the so-called RMF property, also introduced in \cite{HMP}, which means the boundedness of the so-called Rademacher maximal function from $L^p(\R^N,\mu;X)$ to $L^p(\R^N,\mu)$. A detailed study of this property can be found in Kemppainen \cite{Kemppainen}. The RMF property is used to estimate the so-called paraproducts arising in the proof of the $Tb$ theorem; for the same purpose, RMF was also assumed in an earlier version of the global non-homogeneous $Tb$ theorem \cite{hytonen1}, but it was subsequently circumvented there. In addition, we make explicit use of the lattice structure at one specific point of the proof to obtain a certain auxiliary square function estimate. We do not know about the necessity of this assumption, so it seems interesting to single out the place where we use it for possible further investigation. Note that many of the concrete $\mathrm{UMD}$ spaces appearing in harmonic analysis, like the $L^p$ and Lorentz spaces, are all lattices; others, like Sobolev spaces, can still be identified with closed subspaces of such lattices, e.g. for $\Omega\subset\C^N$, we have $W^{1,p}(\Omega)\simeq\{(f,g)\in L^p(\Omega)\times L^p(\Omega)^N:g=\nabla f\}\subset L^p(\Omega)^{N+1}\eqsim L^p(\bigcup_{i=0}^N\Omega_i)$, where the $\Omega_i$'s are disjoint copies of $\Omega$; but some other examples of $\mathrm{UMD}$ spaces like the Schatten ideals $C_p$ fall outside this class of spaces.

We are ready to formulate our main result.

\begin{lause}\label{mainth}
Suppose that $X$ is a $\mathrm{UMD}$ function lattice.
Assume that $T$ is a Calder\'on--Zygmund operator, and that there exists
two $L^\infty$-accretive systems, $b^1=\{b_Q^1\}$ for $T$ and $b^2=\{b_Q^2\}$ for $T^*$.
Then, under the qualitative a priori assumption that
$T\in\mathcal{L}(L^p(\R^N,\mu;X))$ for some $p\in (1,\infty)$,
we have the quantitative bound
\[
\|T\|_{\mathcal{L}(L^p(\R^N,\mu;X))}\le C,
\]
where the constant $C=(N,d,\alpha,\delta,B,p,X)>0$ is independent of $T$.
\end{lause}

Having stated this, we should admit two things. First, this result remains valid for general $\mathrm{UMD}$ spaces. Second, it follows relatively easily, even in the just mentioned more general form, from a combination of the results of \cite{hytonen1}, \cite{NTV} and \cite{NTV:Tb}. Namely, Nazarov, Treil and Volberg's local $Tb$ theorem \cite{NTV} states that under the mentioned assumptions we have the scalar valued bound $\|T\|_{\mathcal{L}(L^2(\R^N,\mu))}\le C$. Then, the converse direction of Nazarov, Treil and Volberg's global $Tb$ theorem \cite{NTV:Tb} tells that $T$ satisfies the global $Tb$ (or even $T1$ conditions) $\|T1\|_{\mathrm{BMO}(\R^N,\mu)}+\|T^*1\|_{\mathrm{BMO}(\R^N,\mu)}\leq C$, where $\mathrm{BMO}(\R^N,\mu)$ is an appropriate bounded mean oscillation space adapted to the non-homogeneous situation. Finally, the vector-valued global $Tb$ theorem of \cite{hytonen1} completes the argument, as we have just checked that its assumptions are satisfied.

What, then, is the point of struggling for a weaker statement, when a stronger one is available for free? Sure, we can still develop some new insight into the proof technique of \cite{NTV}, but there is also a more substantial reason on the level of actual results. Namely, the proof of Theorem~\ref{mainth} that we give immediately yields a further generalization to the case of operator-valued kernels $K$, i.e., for 
 kernels $K(x,y)\in\mathcal{L}(X)$.
 Then the associated $\mathcal{L}(X)$-valued operator
  $T$ in \eqref{action} is
genuinely an object of the vector-valued realm, and the above shortcut via the scalar-valued theory is no longer available. 

\subsection*{Rademacher--Calder\'on--Zygmund operators}
Let us consider an operator $T$ given by the same formula \eqref{action} as before, but with $K(x,y)\in\mathcal{L}(X)$. The kernel bounds above will have to be replaced by certain operator-theoretic analogues involving the notion of $\mathcal{R}$-boundedness (see definition in \eqref{r_bounded}), and we refer the reader to Section~\ref{oper_kernels} for a precise statement. We then say that $K$ is a $d$-dimensional Rademacher--Calder\'on--Zygmund kernel, and that $T$ is an $\mathcal{L}(X)$-valued Rademacher--Calderon--Zygmund operator.
For further details, we refer to Section \ref{oper_kernels}.

The testing functions are now as follows.
We  say that $\{b_Q^1\}$ is an
$L^\infty$-accretive system for
an $\mathcal{L}(X)$-valued Rademacher--Calder\'on--Zygmund operator $T$ if, for every cube $Q$ in $\R^N$,
there is a function $b_Q^1$ from the system such that the conditions \eqref{accre} 
hold true for $b_Q=b_Q^1$ and $Tb_Q^1:\R^N\to Y$ satisfies
$||Tb_Q^1||_{L^\infty(\R^N,\mu;Y)}\le B$,
where $Y\subset\mathcal{L}(X)$ is an
$\mathrm{UMD}$ function lattice 
which has cotype $2$ (see definition in \eqref{cotype_s}), and 
whose unit ball $\bar B_{Y}$ is an $\mathcal{R}$-bounded subset of $\mathcal{L}(X)$.
In a similar manner, we say that $\{b_Q^2\}$ is an $L^\infty$-accretive system is for  $T^*$ 
if, for every cube $Q$ in $\R^N$,
there is a function $b_Q^2$ from the system such that the conditions \eqref{accre} 
hold true for $b_Q=b_Q^2$ and $T^*b_Q^2:\R^N\to Z$ satisfies
$||T^*b_Q^2||_{L^\infty(\R^N,\mu;Z)}\le B$,
where $Z\subset \mathcal{L}(X^*)$ is an
$\mathrm{UMD}$ function lattice 
which has cotype $2$ and 
whose unit ball $\bar B_{Z}$ is an $\mathcal{R}$-bounded subset of $\mathcal{L}(X^*)$.

The following local $Tb$ theorem for operator-valued kernels is obtained by employing the entire power of the proof of Theorem~\ref{mainth}, with minor necessary adjustments. Unlike Theorem~\ref{mainth}, it cannot be obtained by a shortcut from the scalar-valued $Tb$ theorem of Nazarov, Treil and Volberg \cite{NTV}.

\begin{lause}\label{mainth_operator}
Suppose that $X$ is a $\mathrm{UMD}$ function lattice.
Assume that  $T$ is an $\mathcal{L}(X)$-valued Rademacher--Calder\'on--Zygmund operator, and that there exists two $L^\infty$-accretive systems,
 $b^1=\{b_Q^1\}$ for $T$ and
$b^2=\{b_Q^2\}$ for  $T^*$.
Then, under the qualitative a priori assumption that
$T\in\mathcal{L}(L^p(\R^N,\mu;X))$ for some $p\in (1,\infty)$,
we have the quantitative bound
\[
\|T\|_{\mathcal{L}(L^p(\R^N,\mu;X))}\le C,
\]
where the constant $C=C(N,d,\alpha,\delta,B,p,X,Y,Z)>0$ is independent of $T$.\end{lause}

Concerning the interest and potential applicability of such a result over the simpler Theorem~\ref{mainth}, we make the following remarks. First, in applying the global vector-valued $Tb$ theorem from \cite{hytonen1}, Mayboroda and Volberg \cite[p.~1056]{MayVol} specifically use the operator-kernel version \cite[$Tb$ theorem 4]{hytonen1}. Second, in the mentioned application, all the Banach spaces are function lattices, so that this assumption is not too restrictive for such purposes. We also recall, although this is not directly connected to the non-homogeneous issues at hand, that the theory of singular integrals with operator-kernel has been a necessary strengthening of the vector-valued scalar-kernel theory in applications like the maximal regularity question for partial differential equations; see in particular the influential paper \cite{Weis}.

\subsection*{Organization of the paper}
In order to keep the notation somewhat lighter, we will concentrate in the main body of the paper on the proof of Theorem~\ref{mainth} about scalar-valued kernels. Mostly, however, this argument goes through without trouble for the operator-kernel version of Theorem~\ref{mainth_operator} as well, and we only explain a few necessary modifications in the final Section~\ref{oper_kernels}. After collecting some preliminaries in Section~\ref{preparations}, the proof of Theorem~\ref{mainth} is presented in Sections \ref{decfun} through \ref{synthesis}:

Sections \ref{decfun} and \ref{normest} present a detailed analysis, and related inequalities, of functions $f\in L^p(\R^N,\mu;X)$ and $g\in L^q(\R^N,\mu;X^*)$ in terms of appropriate adapted martingale difference operators $D^a_Q$. In Section~\ref{dec_cald}, these expansions of functions then give a representation of the operator $T$ in terms of matrix elements $T_{RQ}$, where $R$ and $Q$ range over dyadic cubes, and the rest of the proof is concerned with the estimation of different parts of this matrix.

Section \ref{decoul} presents a general martingale decoupling inequality --- our best substitute for orthogonality estimates in $L^2$ ---, which will be used several times during the proof. 
The parts of the matrix $T_{RQ}$ leading to different types of treatment are as follows: the separated cubes (handled in Section~\ref{separated}), the deeply nested cubes (Sections \ref{prepare} through \ref{paraproducts}, where the last one deals with the paraproduct part of the operator), and the near-by cubes of comparable size (Sections \ref{comparable} and \ref{intersecting}). Finally, Section~\ref{synthesis} collects the different estimates together, and also takes care of the remaining `bad' cubes which were excluded from the previous cases.

\subsection*{Acknowledgements}
This paper benefited from the interaction with a simultaneously on-going (but earlier completed) project of the first author with Henri Martikainen; see \cite{hm}. We would like to thank him for this fruitful exchange of ideas.

\section{Preparations}\label{preparations}

\subsection*{Notation}
We denote $\N=\{1,2,\ldots\}$ and $\N_0=\{0,1,\ldots\}$.
All distances in $\R^N$ are measured in terms of the 
supremum norm, defined by $|x|=||x||_\infty$ for $x\in \R^N$. 
Accordingly, we henceforth write $B(x,t)$ for the $\ell^{\infty}$  ball in $\R^N$ centered at $x$ with radius $t>0$. (Note that the main assumption~\eqref{mesas} is still true, possibly after scaling $\mu$ by a constant.)
We assume that $K$ is a $d$-dimensional Calder\'on--Zygmund kernel in $\R^N$, satisfying
both \eqref{size} and \eqref{smooth} for some $\alpha>0$.
In the sequel $r>0$ is a (large) integer which is  to be quantified later. We fix
a constant $\gamma$,
\begin{equation}\label{gdef}
\gamma\in (0,1),\quad d\gamma/(1-\gamma)\le \alpha/4,\quad  \gamma\le \frac{\alpha}{2(d+\alpha)}.
\end{equation}
We 
denote
\[
\theta(j)=\Big\lceil\frac{\gamma j+r}{1-\gamma}\Big\rceil\quad \text{ for }j=0,1,2,\ldots.
\]

A {\em cube} $Q$ in $\R^N$ has sides parallel to the coordinate axes, and
its side length is denoted by $\ell(Q)$.
If $Q,R\subset \R^N$ are cubes, 
their {\em long distance} $D(Q,R)$ 
is defined by
$D(Q,R)=\ell(Q)+\dist(Q,R)+\ell(R)$.

\subsection*{$\mathcal{R}$-boundedness}
Let $(\epsilon_k)_{k\in\Z}$ be a sequence of Rademacher functions, i.e.,
a sequence of independent random variables attaining values $\pm 1$
with an equal probability $\mathbf{P}(\epsilon_k=-1)=\mathbf{P}(\epsilon_k=1)=1/2$.
By $\Omega$ we denote the probability
space supporting the distribution of $(\epsilon_k)_{k\in\Z}$. The Khintchine--Kahane inequality says that
\begin{equation*}
  \bigg\|\sum_{k=1}^n \epsilon_k \xi_k\bigg\|_{L^p(\Omega;X)}
  \eqsim\bigg\|\sum_{k=1}^n \epsilon_k \xi_k\bigg\|_{L^2(\Omega;X)}
\end{equation*}
for all $p\in(0,\infty)$. For $X=\C$, this is called just Khintchine's inequality, and the right hand side can be written as the quadratic expression $\big(\sum_{k=1}^n|\xi_k|^2\big)^{1/2}$. Because of this, inequalities for the random series involving the $\epsilon_k$ are often referred to as `square-function' estimates even in the vector-valued case, even if no squares explicitly appear.

We recall that an operator family $\mathcal{T}\subset \mathcal{L}(X_1,X_2)$ is
called {\em Rademacher-bounded}, or $\mathcal{R}$-bounded, if
there is a constant $c$ such that for all $n\in \N$, all $\xi_1,\ldots,\xi_n\in X_1$ and
all $T_1,\ldots,T_n\in \mathcal{T}$,
\begin{equation}\label{r_bounded}
\bigg\|\sum_{k=1}^n \epsilon_k T_k\xi_k\bigg\|_{L^2(\Omega;X_2)}
\le c\bigg\|\sum_{k=1}^n \epsilon_k \xi_k\bigg\|_{L^2(\Omega;X_1)}.
\end{equation}
Denote the smallest admissible $c$ by $\mathcal{R}(\mathcal{T})$.

We will often use the following {\em Stein's inequality} (more precisely, its vector-valued extension due to Bourgain \cite{bourgain}), which says that an increasing sequence of conditional expectations $E_k$ is $\mathcal{R}$-bounded on $L^p(\R^n,\mu;X)$ if $X$ is a UMD space:
\begin{equation*}
\bigg\|\sum_{k=1}^n \epsilon_k E_k f_k\bigg\|_{L^p(\Omega\times\R^N;X)}
\le c\bigg\|\sum_{k=1}^n \epsilon_k f_k\bigg\|_{L^p(\Omega\times\R^N;X)}.
\end{equation*}

A different condition arises by requiring the pointwise (in $x\in\R^N$) $\mathcal{R}$-boudedness of the sequence of vectors $E_k f(x)\in X\simeq\mathcal{L}(\C,X)$, where the last identification is the obvious one: $\xi\in X$ is identified with the operator $\lambda\in\C\mapsto \lambda\xi\in X$. We denote
\begin{equation*}
  M_R f(x):=\mathcal{R}(\{E_k f(x):k\in\Z\}),
\end{equation*}
and say that $X$ has the {\em RMF (Rademacher maximal function) property}, if $M_R:L^p(\R^N,\mu;X)\to L^p(\R^N,\mu)$ boundedly for some (and then all) $p\in(1,\infty)$. This notion was introduced in \cite{HMP}; see \cite{HytKem,HMP,Kemppainen} for more information.

\subsection*{Cotype of a Banach space}
A Banach space $X$ is said to have cotype $s\in[2,\infty)$, i.e., there is a constant $C>0$ such
that for all sequences $(x_j)_{j=1}^n$ in $X$ we have
\begin{equation}\label{cotype_s}
\bigg(\sum_{j=1}^n |\xi_j|_X^s\bigg)^{1/s}\le
C\bigg\|\sum_{j=1}^n \epsilon_j \xi_j\bigg\|_{L^2(\Omega;X)}.
\end{equation}
This leads to an improvement
of the contraction principle, \cite[Proposition 11.4]{hytonen1}.

\begin{prop}\label{improved}
Let $X$ be a Banach space of cotype $s\in [2,\infty)$
and suppose that $\{\rho_j\,:\,j\in \N\}\subset L^t(\tilde\Omega)$ for some
$\sigma$-finite measure space $\tilde\Omega$ and $t\in (s,\infty)$.
Then
\[
\bigg\| \sum_{j=1}^\infty \epsilon_j \rho_j\xi_j\bigg\|_{L^t(\tilde\Omega;L^2(\Omega;X))}
\lesssim \sup_{j} \|\rho_j\|_{L^t(\tilde\Omega)}\cdot \bigg\|\sum_{j=1}^\infty \epsilon_j\xi_j\bigg\|_{L^2(\Omega;X)}
\]
if $\{\xi_j\,:\,j\in \N\}\subset X$.
\end{prop}

Some of the subsequent estimates are based on the fact that every $\mathrm{UMD}$ has cotype $s$ for some $s\in [2,\infty)$. This well-known fact can be seen as follows: First, one can explicitly check that the $\mathrm{UMD}$ constants of the finite-dimensional spaces $\ell^{\infty}(n)$ blow up as $n\to\infty$, and therefore a $\mathrm{UMD}$ space cannot contain the copies of these spaces uniformly. Second, the property that the spaces $\ell^{\infty}(n)$ are not uniformly contained in $X$ is equivalent to some cotype $s\in[2,\infty)$ by the celebrated Maurey--Pisier theorem.

\subsection*{Generic dyadic systems}
Let $\hat{\mathcal{D}}$ denote the standard dyadic system, consisting
of all of the cubes of the form $2^k(m+[0,1[^N)$, where $k\in \Z$ and $m\in \Z^N$.
We also denote \[\hat{\mathcal{D}}_k=\{Q\in \hat{\mathcal{D}}\,:\,\ell(Q)=2^{k}\}.\]
A generic dyadic system, parametrized by $\beta\in (\{0,1\}^N)^\Z$, 
is of the form
\[\mathcal{D}(\beta)=\bigcup_{k\in\Z} \mathcal{D}_k(\beta),\] where
\[\mathcal{D}_k(\beta)=\{\hat R+x_k(\beta)\,:\,\hat R\in \hat{\mathcal{D}}_k\},\qquad x_k(\beta)=\sum_{j<k} \beta_j 2^j.\] 
Given $Q\in\mathcal{D}(\beta)$ and $n\in\N_0$, then the expression
$Q^{(n)}$ denotes the dyadic ancestor of $Q$ of the $n$'th generation, i.e.,
it is the unique cube such that $Q\subset Q^{(n)}\in\mathcal{D}(\beta)$ 
and $\ell(Q^{(n)})=2^n \ell(Q)$.

\subsection*{Random dyadic systems}
The generic dyadic systems give rise to random dyadic systems 
by assigning the complete product probability measure $\mathbf{P}_\beta$ on the set $(\{0,1\}^N)^\Z$ so that
the coordinate functions $\beta_j$, $j\in\Z$, are independent and $\mathbf{P}_\beta[\beta_j=\eta]=2^{-N}$
if $\eta\in \{0,1\}^N$. 
We use two independent random dyadic systems, denoted by
$\mathcal{D}=\mathcal{D}(\beta)$ and $\mathcal{D}'=\mathcal{D}(\beta')$.

Let $n\in\Z$. A cube $Q\in \mathcal{D}$ is called
{\em $n$-bad} (w.r.t. $\mathcal{D}'$)  if there exists $R\in\mathcal{D}'$ such
that 
\[\ell(Q)\le 2^{-(n\vee r)}\ell(R),\qquad \dist(Q,\partial R)\le \ell(Q)^\gamma\ell(R)^{1-\gamma}.\]
If $Q$ is not $n$-bad (w.r.t. $\mathcal{D}'$) then it is {\em $n$-good} (w.r.t $\mathcal{D}'$).
The set of $n$-good cubes
in $\mathcal{D}$  is denoted by \[\mathcal{D}_{n\textrm{-good}}=\mathcal{D}_{n\textrm{-good}(\gamma,r)}.\]
The family of $n$-bad cubes in $\mathcal{D}$ is denoted by $\mathcal{D}_{n\text{-bad}}=\mathcal{D}_{n\text{-bad}(\gamma,r)}$. 
A cube $Q\in\mathcal{D}_i$ is {\em $R$-bad}, $R\in\mathcal{D}'_j$, if
$Q$ is $(j-i-1)$-bad.
A cube $Q\in\mathcal{D}$ is {\em $R$-good} if it is not $R$-bad.
The family of $R$-good cubes in $\mathcal{D}$ is denoted by $\mathcal{D}_{R\textrm{-good}}=
\mathcal{D}_{R\text{-good}(\gamma,r)}$. 
The family of $R$-bad cubes in $\mathcal{D}$ is denoted by $\mathcal{D}_{R\textrm{-bad}}=
\mathcal{D}_{R\text{-bad}(\gamma,r)}$. 

In a symmetric manner we define the $n$/$Q$-bad and $n$/$Q$-good cubes in $\mathcal{D}'$.

\begin{huom}\label{new_hyvmeas}
Assume that $Q\in\mathcal{D}$ is $\bar R$-good, where
$\bar R\in\mathcal{D}'$. Then
\[
\dist(Q,\partial R)> \ell(Q)^\gamma \ell(R)^{1-\gamma}\ge \ell(Q)^{\alpha/2(d+\alpha)}
\ell(R)^{1-\alpha/2(d+\alpha)}.
\]
for every $R\in\mathcal{D}'$ satisfying $\ell(R)\ge 2^{-1}\ell(\bar R)\vee 2^r\ell(Q)$.
The second inequality follows from the estimate $\gamma\le \alpha/2(d+\alpha)$ in \eqref{gdef}.
On the other hand, assuming that $Q\in\mathcal{D}$ is $\bar R$-bad,
where $\bar R\in \mathcal{D}'$, then 
\[
\dist(Q,\partial R)\le \ell(Q)^\gamma \ell(R)^{1-\gamma}
\]
for some $R\in\mathcal{D}'$ for which $2^{-1}\ell(\bar R)\le \ell(R)$
and $\ell(Q)\le 2^{-r}\ell(R)$.
\end{huom}

Here is a useful lemma controlling the probability of bad cubes:

\begin{lem}\label{nmeas}
Let $n\in\Z$ and $Q\in\mathcal{D}=\mathcal{D}(\beta)$ be fixed. Then
\[
\mathbf{P}_{\beta'}[Q\in\mathcal{D}_{n\text{-bad}(\gamma,r)}]\le 2N\frac{2^{-(r\vee n)\gamma}}{1-2^{-\gamma}}.
\]
\end{lem}

\begin{proof}
Just follow the proof of \cite[Lemma~7.1]{NTV} with $n\vee r$ in place of $r$.
\end{proof}

Various estimates are
conducted while keeping the parameters $\beta,\beta'\in (\{0,1\}^N\}^\Z$, and hence also the associated dyadic systems, fixed.
During these estimates, we will assume that these fixed dyadic systems satisfy the following condition:
there are (fixed) cubes $Q_0\in\mathcal{D}(\beta)$ and $R_0\in\mathcal{D}(\beta')$
for which 
\begin{equation}\label{rsubs}
\ell(Q_0)=\ell(R_0)=2^s\text{ and }\mathrm{supp}\,\mu \subset Q_0\cap R_0.
\end{equation}
From the probabilistic point-of-view this assumption is justified by the following lemma,
when applied to the compact set $K=\mathrm{supp}\,\mu$:

\begin{lem}
Let $K\subset \R^N$ be a bounded set. Denote
by $A$ the set of parameters $\sigma\in (\{0,1\}^N)^\Z$ for which
$K$ is not contained in any cube $R\in\mathcal{D}(\sigma)$.
Then $\mathbf{P}(A)=0$.
\end{lem}

\begin{proof}
By using completeness of $\mathbf{P}$, it suffices to show that
$A$ is contained in a set of probability zero.
To this end, we use the fact
that $K$ is bounded as follows: there are dyadic cubes $Q_1,\ldots,Q_{2^N}\in\hat{\mathcal{D}}_k$
from the standard dyadic system of sufficiently large generation $k\in \Z$
such that $K\subset \cup_{j=1}^{2^N} Q_j$. Thus, if $\sigma\in A$ and  $n\in\N$, $n\ge r$, there exists an index $j\in \{1,2,\ldots,2^N\}$ and a cube $R\in\mathcal{D}_{k+n}(\sigma)$ so
that $\dist(Q_j,\partial R)=0$. Hence
$Q_j\in\hat{\mathcal{D}}_{n\textrm{-bad}(\gamma,r)}$ with respect to $\mathcal{D}(\sigma)$. We have shown that
\[
A\subset \bigcap_{n\ge r} \bigcup_{j=1}^{2^N}\{\sigma\,:\,Q_j\in\hat{\mathcal{D}}_{n\textrm{-bad}(\gamma,r)}\text{ w.r.t. }\mathcal{D}(\sigma)\}.
\]
Using Lemma \ref{nmeas}, we see that
the probability of the right hand side is zero.
\end{proof}

\subsection*{Layers of cubes}\label{layers}
Following \cite{NTV} we will define certain layers of cubes in a given dyadic system $\mathcal{D}$.
For this purpose, we fix $\beta,\beta'\in (\{0,1\}^\N)^\Z$, and
assume that $Q_0\in\mathcal{D}=\mathcal{D}(\beta)$ and $R_0\in \mathcal{D}'=\mathcal{D}(\beta')$ are
cubes such that \eqref{rsubs} holds true. 
By $s\in\Z$ we denote
a sufficiently large integer for which $Q_0\in\mathcal{D}_s$ and
$R_0\in\mathcal{D}'_s$.

Let $\mathcal{D}^0 = \{Q_0\}$ be the {\em zeroth layer of cubes}.
Assume that the layers $\mathcal{D}^0,\ldots, \mathcal{D}^{j-1}$ of cubes have been defined.
We then define the {\em $j$'th layer  of cubes}  $\mathcal{D}^j$  as follows.
If $\mathcal{D}^{j-1}=\emptyset$, we set $\mathcal{D}^j=\emptyset$. Otherwise 
we consider a cube $R\in\mathcal{D}^{j-1}$. We say
$Q\in\mathcal{D}$ is $R$-{\em maximal}, if it is the maximal cube in $\mathcal{D}$ satisfying
the conditions $Q\subsetneq R$ and
\[
\bigg|\int_{Q} b_R^1 \,d\mu\bigg|< \delta^2\,\mu(Q).
\]
Then we denote $\mathcal{D}^j = \cup_{R\in\mathcal{D}^{j-1}}\{Q\in\mathcal{D}\,:\,Q\text{ is }R\text{-maximal}\}$.
By analogy, we define the layers $(\mathcal{D}')^{j}$ for $j\ge 0$.

Assuming that  $\mathcal{D}\ni Q\subset Q_0$, we denote by $Q^a$ the smallest cube in 
$\cup_{j\ge 0}\mathcal{D}^j$ that
contains $Q$. Such a cube exists because $Q\subset Q_0\in\cup_{j\ge 0}\mathcal{D}^j$, and
it is also unique due to properties of dyadic cubes; hence $Q^a$ is well defined.
If $Q\not\subset Q_0$ then we denote $Q^a=Q_0$ for the sake of convenience.
Note that, in any case, we
have
\begin{equation}\label{a_iso}
\bigg|\int_{Q} b_{Q^a}^1 \,d\mu\bigg|\ge \delta^2\,\mu(Q).
\end{equation}
In an analogous manner, we define $R^a$ for cubes $R\in \mathcal{D}'$.

For a fixed  $j\in\N$, we have the estimate
\begin{equation}\label{basic}
\mu\bigg(\bigcup_{R\in\mathcal{D}^j:\,R\subsetneq Q}  R\bigg)\le (1-\tau)\mu(Q)\quad \text{ for }Q\in\mathcal{D}^{j-1}.
\end{equation}
A proof is in \cite[pp. 269--270]{NTV}.
Here $\tau\in (0,1)$ is a constant, depending only on $\delta_b$. 
This estimate generalizes by simple iteration as follows:

\begin{lem}\label{basicsum}
Let $Q\in\mathcal{D}$ and $Q^a\in\mathcal{D}^M$ for $M\in \N_0$. Then, for $j\ge 1$, we have the estimate
\[
\mu\bigg(\bigcup_{S\in\mathcal{D}^{M+j}:S\subsetneq Q} S\bigg)=\sum_{S\in\mathcal{D}^{M+j}:S\subsetneq Q}\mu(S)\le (1-\tau)^{j-1}\mu(Q).
\]
\end{lem}

%

\begin{huom}\label{harva}
Lemma \ref{basicsum} yields that
$\mu$-a.e. point $x\in Q_0$ belongs to at most finitely many cubes
in the family $\cup_{j\in\N_0} \mathcal{D}^j$. 
To prove this, let us denote
$f=\sum_{j=0}^\infty\sum_{Q\in\mathcal{D}^j} \chi_Q.$
The cubes in $\mathcal{D}^j$ for $j\in\N$ are disjoint,
and they are all included in $Q_0=Q_0^a\in\mathcal{D}^0$. Using
Lemma \ref{basicsum} with $Q=Q_0$ and $M=0$, we get
\begin{equation}\label{esti}
\begin{split}
\|f\|_{L^1(Q_0)}&\le \sum_{j=0}^\infty\sum_{Q\in\mathcal{D}^j} \mu(Q)\le \mu(Q_0)\bigg(1+\sum_{j=0}^\infty(1-\tau)^j\bigg)<\infty.
\end{split}
\end{equation}
The claim follows.
\end{huom}

\subsection*{Carleson embeddings}
Let 
$\mathcal{D}$ be a generic dyadic system and let
$d_k\in L^1(\R^N,\mu;\R)$, $k\in\Z$, be a sequence of functions, and
denote
\[
\|\{d_k\}_{k\in\Z}\|_{\mathrm{Car}^p(\mathcal{D})}=\sup_{\substack{Q\in\mathcal{D}\\ \mu(Q)\not=0}} \frac{1}{\mu(Q)^{1/p}} \bigg\|1_Q \sum_{k\,:\,2^k\le \ell(Q)} \epsilon_kd_k\bigg\|_{L^p(\R^N\times\Omega,\mu\otimes\mathbf{P};\R)}.
\]
If $d_k=E_kd_k$ for all $k\in\Z$, then the Carleson norms are equivalent for all $p\in [1,\infty)$. For a proof, see Proposition 3.1 in \cite{hytonen1}.
We recall  two Carleson embedding theorems; The following result is Theorem 3.4 in \cite{hytonen1}.

\begin{lause}\label{carleson}
Let $X$ be a $\mathrm{UMD}$ space and $1<p<\infty$.
Let $\{d_k\}_{k\in\Z}\subset L^1(\R^N,\mu;\R)$ be a sequence be such that $d_k=E_kd_k$ for every $k\in\Z$.  Then
\[
\bigg\|\sum_{k\in\Z} \epsilon_kd_k E_kf\bigg\|_{L^p(\R^N\times\Omega,\mu\otimes\mathbf{P};X)}\lesssim\|\{d_k\}_{k\in\Z}\|_{\mathrm{Car}^1(\mathcal{D})}\|f\|_{L^p(\R^N,\mu;X)}
\]
for every $f\in L^p(\R^N,\mu;X)$.
\end{lause}

The following embedding result for $\mathrm{RMF}$ spaces is essentially Theorem 8.2 in \cite{HMP}, where it is stated for the Lebesgue measure; see also \cite{HytKem} for a general measure $\mu$ and an interesting converse statement.

\begin{lause}\label{carlesonRMF}
Let $X$ be an $\mathrm{RMF}$ space, $1<p<\infty$, and $\eta>0$. Assuming that
 $\{d_k\}_{k\in\Z}$ is a sequence in $L^1(\R^N,\mu;\R)$, then
\[
\bigg\|\sum_{k\in\Z} \epsilon_kd_k E_kf\bigg\|_{L^p(\R^N\times\Omega,\mu\otimes\mathbf{P};X)}\lesssim\|\{d_k\}_{k\in\Z}\|_{\mathrm{Car}^{p+\eta}(\mathcal{D})}\|f\|_{L^p(\R^N,\mu;X)}
\]
for every $f\in L^p(\R^N,\mu;X)$.
\end{lause}

\section{Adapted martingale decompositions}\label{decfun}

Throughout this section we assume that $X$ is
a Banach space and that $b=b^1$ is
an $L^\infty$-accretive system. 
The assumption that $b=b^1$ is only for notational convenience, and
all of the
results throughout this section remain valid
if we replace $b^1$ with $b^2$ and
the random dyadic system $\mathcal{D}$
with $\mathcal{D}'$.

\subsection*{Adapted conditional expectations}
Let $\mathcal{D}=\bigcup_{k\in\Z}Ê\mathcal{D}_k$
be a generic system of dyadic cubes in $\R^N$ and
$f\in L^1_{\mathrm{\mathrm{loc}}}(\R^N;X)$. 
In what follows we will define various operators acting
on this function.
First of all, the conditional expectation for $k\in\Z$ is defined by
\[
E_k f := \sum_{Q\in\mathcal{D}_k} 1_Q \langle f\rangle_Q,\qquad \langle f\rangle_Q := 
\frac{1}{\mu(Q)}\int_Q f\,d\mu.
\]
If $\mu(Q)=0$ for a cube $Q$, we agree
that $\langle f\rangle_Q =0$. 
If $Q\in\mathcal{D}_k$ is a cube,  then
the local version of this conditional expectation is
defined by $E_Q f:=1_QE_k f$.
The corresponding
martingale difference is defined by
$D_k f:=E_{k-1}f-E_k f$ and
its local version is  $D_Q f:= 1_Q D_k f$.
For the $L^\infty$-accretive system $b=\{b_Q\}$ and $k\in\Z$, we define
\[
b_k^{a}:=\sum_{Q\in\mathcal{D}_k} 1_Q b_{Q^a}.
\]
The $b$-adapted conditional expectation and its local version,
for $k\in \Z$ and $Q\in\mathcal{D}_k$, are defined by
\begin{align*}
&E_k^{a} f:=b_k^{a} \frac{E_k f}{E_k b_k^{a}},\qquad E_Q^{a} f := 1_Q E_k^{a} f.
\end{align*}
The corresponding $b$-adapted martingale difference and its local version are 
\[
D_k^{a} f:=E_{k-1}^{a} f - E_k^{a} f,\qquad D_Q^{a} f := 1_Q D_k^{a} f,
\]
where $k\in \Z$ and $Q\in\mathcal{D}_k$.

We agree on the following slightly abusive notation:
\begin{equation}\label{bmaar}
\begin{split}
&{\{{b_{k-1}^{a} = b_{k}^{a}}\}}:=\bigcup_{\substack{ Q\in\mathcal{D}_{k-1} \\ Q^a = (Q^{(1)})^a}} Q, \\
&\chi_{k-1}:={\{{b_{k-1}^{a} \not= b_{k}^{a}}\}} :=\R^N\setminus {\{{b_{k-1}^{a} = b_{k}^{a}}\}}=
\bigcup_{\substack{ Q\in\mathcal{D}_{k-1} \\ Q = Q^a\not=Q_0}} Q,
\end{split}
\end{equation}
where $k\in\Z$.

\subsection*{A representation for $D_Q^{a}$}\label{secs}
Here we  compute a useful representation for the adapted
martingale differences of  $f\in L^1_{\mathrm{loc}}(\R^N;X)$. For this purpose, we let
 $Q\in\mathcal{D}_k$ and denote by $Q_i,\ldots,Q_{2^N}\in \mathcal{D}_{k-1}$
 the subcubes
of $Q$ (in some order) so that 
$\bigcup_{i=1}^{2^N} Q_i =Q$. Then
\begin{align*}
D_Q^a f&= 1_Q (E_{k-1}^a f - E_k^a f)
=\sum_{i=1}^{2^N} b_{Q_i^a}\frac{\langle f\rangle_{Q_i}}{\langle b_{Q_i^a}\rangle_{Q_i}}1_{Q_i}-b_{Q^a}\frac{\langle f\rangle_Q}{\langle b_{Q^a}\rangle_Q}1_Q.
\end{align*}
Writing $\langle f \rangle_Q = \frac{1}{\mu(Q)} \sum_{i=1}^{2^N} \mu(Q_i)\langle f\rangle_{Q_i}$, we get
\begin{align*}
D_Q^a f&=\sum_{i=1}^{2^N} \langle f\rangle_{Q_i}\bigg(\frac{b_{Q_i^a}}{\langle b_{Q_i^a}\rangle_{Q_i}}1_{Q_i}-\frac{\mu(Q_i)}{\mu(Q)}\frac{b_{Q^a}}{\langle b_{Q^a}\rangle_Q}1_Q\bigg).
\end{align*}
This computations motivates the following definition:
If $\mu(Q_i)\not=0$, we define
\begin{equation}\label{phideff}
\phi_{Q,i}^a:=\frac{b_{Q_i^a}}{\langle b_{Q_i^a}\rangle_{Q_i}}1_{Q_i}-\frac{\mu(Q_i)}{\mu(Q)}\frac{b_{Q^a}}{\langle b_{Q^a}\rangle_Q}1_Q.
\end{equation}
Otherwise we define $\phi_{Q,i}\equiv 0$.

The following lemma draws conclusions from above,
and provides further properties for the resulting decomposition.
The proof is straightforward.

\begin{lem}\label{phiprop} 
Let $f\in L^1_{\mathrm{loc}}(\R^N;X)$ and $Q\in\mathcal{D}_k$. 
Let $Q_1,\ldots,Q_{2^N}\in \mathcal{D}_{k-1}$ denote the subcubes
of $Q$ in some order.
Then we can write
\[
D_Q^a f=\sum_{i=1}^{2^N} \langle f\rangle_{Q_i}\phi_{Q,i}^a.
\]
Furthermore, if $i\in \{1,2,\ldots,2^N\}$, then
\begin{itemize}
\item[{\em a)}] $\|\phi_{Q,i}^a\|_{L^\infty(\mu)}\lesssim 1$;
\item[{\em b)}] $\|\phi_{Q,i}^a\|_{L^1(d\mu)}\lesssim \mu(Q_i)$;
\item[{\em c)}] $\mathrm{supp}(\phi_{Q,i}^a)\subset Q$;
\item[{\em d)}] $\int_{\R^N} \phi_{Q,i}^a\,d\mu=0$.
\end{itemize}
Within these estimates, the implicit
constant depends on $\delta$ in \eqref{accre}.
\end{lem}


\subsection*{A representation for $(D_k^a)^*$}
We compute the adjoint of
the $b$-adapted martingale difference operator $D_k^a$ for $k\in\Z$. For
this purpose we fix
$f\in L^p(\R^N;X)$ and $g\in L^{q}(\R^N;X^*)$, where $1/p+1/q=1$.
Recall that
\begin{equation}\label{adid}
D_k^a f = E_{k-1}^a f - E_{k}^a f = b_{k-1}^a\frac{E_{k-1}f}{E_{k-1}b_{k-1}^a}-
b_{k}^a\frac{E_{k}f}{E_{k}b_{k}^a}.
\end{equation}
The
self-adjointness of the expectation operator $E_k$ yields
\begin{align*}
\langle f,E_k^a g\rangle = \bigg\langle b_k^a \frac{f}{E_k b_k^a}, E_k g\bigg\rangle =
\bigg\langle E_k\bigg(b_k^a \frac{f}{E_k b_k^a}\bigg), g\bigg\rangle.
\end{align*}
As a consequence, we find that
\[
A_k f:=(E_k^a)^* f = E_k\bigg(b_k^a \frac{f}{E_k b_k^a}\bigg)
=\frac{E_k(b_k^a f)}{E_k b_k^a}.
\]
Substituting this identity to \eqref{adid}, we get the representation
\begin{equation}\label{vaid}
\begin{split}
(D_k^a)^*f = A_{k-1}^a f- A_k^a f 
= \frac{E_{k-1}(b_{k-1}^a f)}{E_{k-1} b_{k-1}^a}-\frac{E_k(b_k^a f)}{E_k b_k^a}.
\end{split}
\end{equation}

\subsection*{A representation for $(D_k^a)^2$}\label{secqr}

Assuming that $k,l\in \Z$,
\begin{align*}
E_k^a E_l^a f &=\frac{b_k^a}{E_k b_k^a} E_k\bigg(b_l^a \frac{E_l f}{E_l b_l^a}\bigg)=
\begin{cases}
\frac{b_k^a}{E_k b_k^a} E_k b_l^a \frac{E_l f}{E_l b_l^a},&\text{ if }l\ge k;\\
E_k^a f,&\text{ if }l\le k.
\end{cases}
\end{align*}
Here we used twice the identity $E_k=E_k E_l$ if $l\le k$. As a consequence, we have
the identity
\begin{align*}
(D_k^a)^2 = (E_{k-1}^a - E_k^a)^2&= (E_{k-1}^a)^2 - E_{k-1}^a E_k^a-E_k^a E_{k-1}^a + (E_k^a)^2\\
&=(E_{k-1}^a)^2 - E_{k-1}^a E_k^a \underbrace{-E_k^a + E_k^a}_{=0} = E_{k-1}^a D_k^a.
\end{align*}
Using the notation \eqref{bmaar}, we write $D_k^a f$ for a
function $f\in L^1_{\mathrm{loc}}(\R^N;X)$  as follows:
\begin{align*}
D_k^a f &= 1_{\{b_{k-1}^a =b_k^a\}} b_k^a \bigg(\frac{E_{k-1}f}{E_{k-1} b_k^a} - \frac{E_k f}{E_k b_k^a}\bigg)\\
&\qquad +1_{\{b_{k-1}^a \not= b_k^a\}} \bigg(b_{k-1}^a \frac{E_{k-1} f}{E_{k-1} b_{k-1}^a}-b_k^a \frac{E_k f}{E_k b_k^a}\bigg).
\end{align*}
Using this and the basic properties of the operator $E_{k-1}$, we get
\begin{align*}
E_{k-1}^a D_k^a f &= \frac{b_{k-1}^a}{E_{k-1} b_{k-1}^a} E_{k-1} D_k^a f\\&= 
\frac{b_{k-1}^a}{E_{k-1} b_{k-1}^a}1_{\{b_{k-1}^a =b_k^a\}} E_{k-1}b_k^a \bigg(\frac{E_{k-1}f}{E_{k-1} b_k^a} - \frac{E_k f}{E_k b_k^a}\bigg)\\
&\qquad +\frac{b_{k-1}^a}{E_{k-1} b_{k-1}^a}1_{\{b_{k-1}^a \not= b_k^a\}} \bigg(E_{k-1}b_{k-1}^a \frac{E_{k-1} f}{E_{k-1} b_{k-1}^a}-E_{k-1}b_k^a \frac{E_k f}{E_k b_k^a}\bigg)\\
&=1_{\{b_{k-1}^a =b_k^a\}} b_k^a\bigg(\frac{E_{k-1}f}{E_{k-1} b_k^a} - \frac{E_k f}{E_k b_k^a}\bigg)\\
&\qquad +1_{\{b_{k-1}^a \not= b_k^a\}} \bigg(b_{k-1}^a\frac{E_{k-1} f}{E_{k-1} b_{k-1}^a}-b_k^a\frac{E_k f}{E_k b_k^a}\bigg)\\
&\qquad + 1_{\{b_{k-1}^a \not= b_k^a\}} E_k f\bigg(\frac{b_k^a}{E_{k} b_{k}^a}-\frac{b_{k-1}^a}{E_{k-1} b_{k-1}^a}\frac{E_{k-1} b_k^a}{E_k b_k^a}\bigg)\\
&= D_k^a f + \omega_k^a E_k f,
\end{align*}
where we have denoted
\[
\omega_k^a  := 1_{\{b_{k-1}^a \not= b_k^a\}} \bigg(\frac{b_k^a}{E_{k} b_{k}^a}-\frac{b_{k-1}^a}{E_{k-1} b_{k-1}^a}\frac{E_{k-1} b_k^a}{E_k b_k^a}\bigg).
\]

The following
lemma draws conclusions from the computations above.

\begin{lem}\label{ombasic}
Let $k\in \Z$. Then 
\begin{equation}\label{qr}
(D_k^a)^2  f = D_k^a f + \omega_k^a E_k f,\qquad\text{ if }f\in L^1_{\mathrm{loc}}(\R^N;X).
\end{equation}
Furthermore, the functions $\omega_k^a$ have the following properties a)--c):
\begin{itemize}
\item[{\em a)}] $\omega_k^a(x)=0$ if $x\in\R^N\setminus \{b_{k-1}^a\not= b_k^a\}$,
\item[{\em b)}] $\|\omega_k^a\|_{L^\infty(\mu)}\lesssim 1$,
\item[{\em c)}] $E_{k-1} \omega_k^a= 0$.
\end{itemize}
The implicit constant in b) depends
on $\delta$ defined in \eqref{accre}.
\end{lem}

\begin{proof}
The identity \eqref{qr} is established above. 
The property a) is clear; b) follows from \eqref{accre} and \eqref{a_iso}.
For c) we notice that
\begin{align*}
E_{k-1} \omega_k^a = 1_{\{b_{k-1}^a \not= b_k^a\}}\bigg(\frac{E_{k-1}b_k^a}{E_{k} b_{k}^a}-\frac{E_{k-1}b_{k-1}^a}{E_{k-1} b_{k-1}^a}\frac{E_{k-1} b_k^a}{E_k b_k^a}\bigg)=0.
\end{align*}
\end{proof}

We also define the following two local versions of $\omega_k^a$.
Let $Q\in\mathcal{D}_k$ and denote by $Q_1,\ldots, Q_{2^N}\in\mathcal{D}_{k-1}$ the 
subcubes of $Q$. Then we define 
\begin{equation}\label{localde}
\omega_Q^a := 1_Q \omega_k^a,\qquad \omega_{Q,i}^a := 1_{Q_i} \omega_Q^a.
\end{equation}
The following lemma collects the basic properties of these local
versions.

\begin{lem}\label{localdq}
If $Q\in\mathcal{D}$ and $f\in L^1_{\mathrm{loc}}(\R^N;X)$, then
\begin{equation}\label{qr2}
(D_Q^a)^2 f=  D_Q^a f + \omega_Q^a E_Q f.
\end{equation}
Also,
\begin{equation}\label{loc_om}
\|\omega_{Q}^a\|_{L^1(\mu)}\lesssim \mu(Q)\text{ and } \|\omega_{Q,i}^a\|_{L^1(\mu)}\lesssim \mu(Q_i)
\text{ for } i=1,2,\ldots,2^N.
\end{equation}
\end{lem}

\begin{proof}
Assume that $Q\in\mathcal{D}_k$. Then, by
using \eqref{qr}, we get
\[
(D_Q^a)^2 f= 1_Q D_k^a (1_Q D_k^a f)= 1_Q (D_k^a)^2(1_Q f)= D_Q^a f + \omega_Q^a E_Q f.
\]
The estimate \eqref{loc_om} follows from Lemma \ref{ombasic}.
\end{proof}

\subsection*{A decomposition of functions}
Recall that $b$ is an $L^\infty$-accretive system.
In the sequel we assume that $b=b^1$ and consider the cube $Q_0\in\mathcal{D}_s$ 
that is defined in \eqref{rsubs}.
It is a large cube such that
the support of $\mu$ is contained in it.

We will show that
\[
f-b_{Q_0}\frac{\langle f\rangle_{\R^N}}{\langle b_{Q_0}\rangle_{\R^N}}
=\sum_{j=-\infty}^\infty D_j^a f
\]
where the convergence takes place both pointwise $\mu$-almost everywhere
and also in $L^p(\R^N,\mu;X)$-norm.
We begin with the
following lemma.

\begin{lem}\label{pist}
For $\mu$ almost every point $x$ in $Q_0$, we have
\[
b_{-\infty}^a(x):=\lim_{k\to -\infty}b_k^a(x)=\lim_{k\to -\infty}ÊE_k b_k^a(x).
\] 
Furthermore, the limit satisfies the
estimate $|b_{-\infty}^a (x)| \ge \delta^2$, where
$\delta>0$ is defined in connection with \eqref{accre}.
\end{lem}

\begin{proof}
By Remark \ref{harva}
we may restrict ourselves to those points $x\in Q_0$ that belong to at most finitely many cubes
in the family $\cup_{j\in\N_0} \mathcal{D}^j$. 
Because the family  $\cup_{j\in\N}\mathcal{D}^j$ is countable,
we can also assume that $\mu(Q_{x,k})\not=0$ for every $k\in\Z$ where
$(Q_{k,x})_{k\in \Z}$ is the unique sequence of cubes such that
$x\in Q_{k,x}\in\mathcal{D}_k$ 
for every $k\in\Z$. 
Finally, we can also assume that 
\begin{equation}\label{lim}
\lim_{k\to -\infty}\langle b_Q\rangle_{Q_{k,x}} =b_Q(x),\qquad \text{ if }Q\in\bigcup_{j\in\N_0}\mathcal{D}^j.
\end{equation}
Indeed, by martingale convergence,
the identity \eqref{lim} holds true 
for
almost every $x\in Q_0$ if $Q$ is fixed, and the family  $\cup_{j\in\N}\mathcal{D}^j$ is countable. 

Fix a point $x$ as described above
and consider the sequence $(Q_{k,x})_{k\in\Z}$.
Note that
$x\in Q_{k,x}\subset Q_{k,x}^a$ for every $k\le s$. In particular, there is an index $k(x)\le s$ such that $Q_{k,x}^a=Q_{k(x),x}^a$ if $k\le k(x)$. As a consequence,  for $k\le k(x)$, we can write
\[
b_k^a(x)=\sum_{Q\in\mathcal{D}_k} 1_Q(x) b_{Q^a}(x)=b_{Q_{k,x}^a}(x)=b_{Q_{k(x),x}^a}(x).
\]
It follows that $\lim_{k\to -\infty}Êb_k^a(x)=b_{Q_{k(x),x}^a}$.
On the other hand, if $k\le k(x)$, we use the assumption \eqref{lim} for 
\begin{equation}\label{valinta}
\begin{split}
E_k b_k^a(x)=\langle b_k^a\rangle_{Q_{k,x}} 
=\langle b_{Q_{k,x}^a}\rangle_{Q_{k,x}}&=\langle b_{Q_{k(x),x}^a}\rangle_{Q_{k,x}}
\\&\xrightarrow{k\to-\infty} b_{Q_{k(x),x}^a}(x)=\lim_{k\to -\infty}Êb_k^a(x).
\end{split}
\end{equation}
This is as required.
Finally, since $\mu(Q_{k,x})\not=0$ for every $k$, we can use
\eqref{a_iso} to conclude that
$|b_{-\infty}^a(x)|=\lim_{k\to -\infty}Ê|\langle b_{Q_{k,x}^a}\rangle_{Q_{k,x}}|\ge \delta^2$.
\end{proof}

With the aid of this lemma, we can establish useful
convergence results. For this purpose, we fix $f\in L^1_{\mathrm{loc}}(\R^N;X)$. 
By martingale convergence,
$\lim_{k\to -\infty} E_k f(x)=f(x)$
for
$\mu$-a.e. $x\in \R^N$ and,
as a consequence of Lemma \ref{pist}, we have
\begin{equation}\label{fkonv}
E_k^a f(x)=b_k^a(x)\frac{E_k f(x)}{E_k b_k^a(x)}\xrightarrow{k\to -\infty} f(x)
\end{equation}
for $\mu$-a.e. $x\in \R^N$.
Recall that $\ell(Q_0)=2^s$ and $\textrm{supp} \mu\subset Q_0$. Hence, for
points $x$ in $Q_0=Q_0^a$, we have
\[
E_s^a f(x) = b_s^a(x) \frac{E_s f(x)}{E_s b_s^a(x)}=b_{Q_0}(x)\frac{\langle f\rangle_{Q_0}}{\langle b_{Q_0}\rangle_{Q_0}}
=b_{Q_0}(x)\frac{\langle f\rangle_{\R^N}}{\langle b_{Q_0}\rangle_{\R^N}}.
\]
Using also \eqref{fkonv} yields  the decomposition
\begin{equation}\label{dec}
\begin{split}
f-b_{Q_0}\frac{\langle f\rangle_{\R^N}}{\langle b_{Q_0}\rangle_{\R^N}}&=\lim_{k\to-\infty} E_k^a f-
E_s^a f \\&=\lim_{k\to -\infty} \sum_{j=k+1}^s \underbrace{\big(E_{j-1}^a f- E_j^a f\big)}_{=D_j^a f}
=\sum_{j=-\infty}^\infty D_j^a f
\end{split}
\end{equation}
that is valid  $\mu$-almost everywhere in $\R^N$.
In the last step we used the 
identity $E_i^a f = E_s^a f$ if $i\ge s$ for $s$ defined in  \eqref{rsubs};
hence,
 $D_j^a f = 0$ $\mu$-almost everywhere if $j>s$.

Let us then consider the convergence in the $L^p$-norm with $1<p<\infty$.
In order to do this, we fix $f\in L^p(\R^N,\mu;X)$.
Let $j\le s$. Using \eqref{a_iso}, we
see that $|E_j b_j^a|\ge \delta^2$ almost everywhere and, by \eqref{accre}, we have $\|b_j^a\|_{L^\infty(\mu)}\le 1$. Hence the following
norm-estimates are valid pointwise $\mu$-almost everywhere
\begin{equation*}
\|E_j^a f\|_X=\bigg\|b_j^a \frac{E_j f}{E_j b_j^a}\bigg\|_X\le \delta^{-2}\|E_j f\|_X
\leq\delta^{-2}Mf\in L^p(\R^N,\mu;\R),
\end{equation*}
where $M$ is Doob's maximal operator.
Hence, by the dominated convergence theorem, we see that the
decomposition \eqref{dec} converges in $L^p(\R^N,\mu;X)$.

\section{Norm estimates for adapted martingales}\label{normest}

We prove various `square-function' estimates
for the adapted
martingale differences and their adjoints, see Theorems
\ref{nests}, \ref{haa} and \ref{normequ}.
The first result is true for general UMD spaces, whereas
the $\mathrm{UMD}$ function lattice property is needed for the last two theorems.
This dichotomy between the square-function estimates for the $D_j^a$ and their adjoints, $(D_j^a)^*$, seems somewhat unexpected: In the original scalar-valued argument of Nazarov, Treil and Volberg \cite[Section~3]{NTV}, only the estimate for $D_j^a$ is proven explicitly, while the dual case is just stated, suggesting that it should follow in a similar way. To some extent it does, but this similarity seems to break down in the vector-valued realm, and we will give a careful consideration of both estimates in this section. We begin with:

\begin{lause}\label{nests}
Let $X$ be a $\mathrm{UMD}$ space and $1<p<\infty$. Then
\begin{equation}\label{thal}
\bigg\|\sum_{j=-\infty}^\infty \epsilon_j D_j^a f\bigg\|_{L^p(\R^N\times\Omega,\mu\otimes \mathbf{P},X)}\lesssim\|f\|_{L^p(\R^N,\mu;X)}
\end{equation}
for every $f\in L^p(\R^N,\mu;X)$.
\end{lause}

In what follows
we prove Theorem \ref{nests}. For this purpose we first prove
various lemmata;  the following  is a consequence  of Theorem \ref{carleson}.
\begin{lem}\label{carut}
Let $X$ be a $\mathrm{UMD}$ space and $1<p<\infty$. Then
\begin{equation}\label{vals}
\bigg\|\sum_{j=-\infty}^\infty \epsilon_j 1_{\{b_{j-1}^a\not=b_j^a\}}E_{j-1}f \bigg\|_{L^p(\R^N\times\Omega,\mu\otimes\mathbf{P};X)}\lesssim \|f\|_{L^p(\R^N,\mu;X)}
\end{equation}
for every $f\in L^p(\R^N,\mu;X)$.
\end{lem}

\begin{proof}
If $j\in\Z$, then $\chi_{j-1}:=1_{\{b_{j-1}^a\not=b_j^a\}}=E_{j-1}1_{\{b_{j-1}^a\not=b_j^a\}}$
belongs to $L^1(\R^N,\mu;\R)$. Therefore we can invoke the Carleson
embedding theorem \ref{carleson}. Hence we can bound
the left hand side of \eqref{vals} by a constant multiple of
$\|\{\chi_j\}_{j\in\Z}\|_{\mathrm{Car}^1(\mathcal{D})}\|f\|_p$.
The first factor is estimated as follows:
\begin{align*}
\|\{\chi_j\}_{j\in\Z}\|_{\mathrm{Car}^1(\mathcal{D})}
\le 
\sup_{\substack{Q\in\mathcal{D}\\ \mu(Q)\not=0}} \frac{1}{\mu(Q)} \sum_{j\,:\,2^j\le \ell(Q)} \|1_Q\chi_j\|_{L^1}.
\end{align*}
Fix $Q\in\mathcal{D}$ with $\mu(Q)\not=0$. Fix $r\in\N$ such that
$Q^a\in\mathcal{D}^r$. Using the definition \eqref{bmaar} and
Lemma \ref{basicsum}, we obtain
\begin{align*}
 \sum_{j:2^j\le \ell(Q)} \|1_Q\chi_j\|_1
&\le \mu(Q)+\sum_{k=1}^\infty \sum_{S\in\mathcal{D}^{r+k}:S\subsetneq Q} \mu(S)\\&\le \mu(Q)+\sum_{k=1}^\infty (1-\tau)^{k-1}\mu(Q) \lesssim \mu(Q).
\end{align*}
Taking the supremum over $Q\in\mathcal{D}$ as above, we have
$\|\{\chi_j\}_{j\in\Z}\|_{\mathrm{Car}^1(\mathcal{D})}\lesssim 1$.
\end{proof}

Another useful estimate is the following.

\begin{lem}\label{carut2}
Let $X$ be a $\mathrm{UMD}$ space and $1<p<\infty$. Then
\begin{equation}\label{es12}
\bigg\|\sum_{j=-\infty}^\infty \epsilon_j 
1_{\{b_{j-1}^a=b_j^a\}}\bigg(\frac{E_{j-1}f}{E_{j-1} b_{j-1}^a}-\frac{E_j f}{E_j b_j^a}\bigg)
\bigg\|_{L^p(\R^N\times\Omega,\mu\otimes\mathbf{P};X)}\lesssim \|f\|_{L^p(\R^N,\mu;X)}
\end{equation}
for every $f\in L^p(\R^N,\mu;X)$.
\end{lem}

\begin{proof}
Using  \eqref{a_iso} and
that \begin{equation}\label{1ids}
1_{\{b_{j-1}^a=b_j^a\}}E_{j-1} b_{j-1}^a=1_{\{b_{j-1}^a=b_j^a\}}E_{j-1} b_{j}^a,
\end{equation}
we see that the following identities hold pointwise $\mu$-almost everywhere in
${\{b_{j-1}^a=b_j^a\}}$:
\begin{align*}
\frac{E_{j-1}f}{E_{j-1} b_{j-1}^a}-\frac{E_j f}{E_j b_j^a}&=
\frac{E_{j-1}f}{E_{j-1} b_j^a}-\frac{E_j f}{E_j b_j^a}\\&=E_{j-1}f\bigg(\frac{1}{E_{j-1} b_j^a}-\frac{1}{E_j b_j^a}\bigg)
+(E_{j-1}f-E_j f)\frac{1}{E_j b_j^a}\\
&=E_{j-1}f\frac{-D_j b_j^a}{E_{j-1}b_j^aE_j b_j^a} + D_j f\frac{1}{E_j b_j^a}.
\end{align*}
Hence, the left hand side of \eqref{es12} is dominated by
\begin{equation}\label{es10}
\begin{split}
\bigg\|\sum_{j=-\infty}^\infty \epsilon_j 1_{\{b_{j-1}^a=b_j^a\}}\frac{D_j b_j^a}{E_{j-1}b_j^aE_j b_j^a}E_{j-1}f \bigg\|_{L^p}+
\bigg\|\sum_{j=-\infty}^\infty \epsilon_j 1_{\{b_{j-1}^a=b_j^a\}}\frac{1}{E_j b_j^a}D_j f \bigg\|_{L^p}.
\end{split}
\end{equation}
Observe that by \eqref{a_iso} we have
$|E_j b_j^a|\ge \delta^{2}$  for $\mu$-almost every point. Therefore the contraction 
principle gives
the following estimate for the second term in \eqref{es10}:
\[
\bigg\|\sum_{j=-\infty}^\infty \epsilon_j 1_{\{b_{j-1}^a=b_j^a\}}\frac{1}{E_j b_j^a}D_j f \bigg\|_{L^p}
\lesssim\bigg\|\sum_{j=-\infty}^\infty \epsilon_jD_j f \bigg\|_{L^p}.
\]
The $\mathrm{UMD}$-property of $X$ allows us to dominate
the right hand side by a constant multiple of $\|f\|_p$.   

For the first term in \eqref{es10} we use the Carleson embedding theorem.
Using \eqref{1ids} and \eqref{a_iso} we see that
$1_{\{b_{j-1}^a=b_j^a\}}|E_{j-1} b_j^a|$ and $|E_{j} b_j^a|$ are bounded from below by $\delta^2$ in $\mu$-almost every point.
Thus, the contraction principle gives the estimate
\begin{equation}\label{caruse}
\begin{split}
&\bigg\|\sum_{j=-\infty}^\infty \epsilon_j 1_{\{b_{j-1}^a=b_j^a\}} \frac{D_j b_j^a}{E_{j-1}b_j^aE_j b_j^a}E_{j-1}f \bigg\|_{L^p}
\lesssim\bigg\|\sum_{j=-\infty}^\infty \epsilon_j D_j b_j^aE_{j-1}f \bigg\|_{L^p}.
\end{split}
\end{equation}
If $j\in\Z$, then $d_{j-1}:=D_j b_j^a\in L^1(\R^N,\mu;\R)$ and 
$D_j b_j^a=E_{j-1}D_j b_j^a$.
Therefore the Carleson embedding Theorem \ref{carleson} applies, and it gives
the estimate
\begin{equation*}
\bigg\|\sum_{j=-\infty}^\infty \epsilon_j D_j b_j^aE_{j-1}f \bigg\|_{L^p}\lesssim\|f\|_{L^p} \sup_{\substack{Q\in\mathcal{D}\\ \mu(Q)\not=0}} \frac{1}{\mu(Q)} \bigg\|1_Q \sum_{k\,:\,2^k\le \ell(Q)} \epsilon_kd_k\bigg\|_{L^1(\R^N\times\Omega,\mu\otimes\mathbf{P};\R)}.
\end{equation*}
In order to estimate the right hand side, we fix a cube $Q\in\mathcal{D}$ for which $\mu(Q)\not=0$. We have
\begin{align*}
\bigg\|1_Q \sum_{k\,:\,2^k \le \ell(Q)} \epsilon_kd_k\bigg\|_{L^1}\le\bigg\|1_Q \sum_{k\,:\,2^k\le \ell(Q)} \epsilon_{k-1} D_{k}b_{k}^a\bigg\|_{L^1}+C\mu(Q).
\end{align*}
The first term on the right hand side is
\begin{align*}
&\bigg\|1_Q \sum_{k\,:\,2^k\le  \ell(Q)} \epsilon_{k-1} \sum_{R\in\mathcal{D}_k:R\subset Q}1_RD_{k}b_{k}^a\bigg\|_{L^1}\\&=
\bigg\| \sum_{k\,:\,2^k\le  \ell(Q)} \epsilon_{k-1} \sum_{R\in\mathcal{D}_k:R\subset Q}D_{R}b_{R^a}\bigg\|_{L^1}
=\bigg\|\sum_{R\in\mathcal{D}:R\subset Q}\epsilon_RD_{R}b_{R^a}\bigg\|_{L^1}
=:\Sigma_Q.
\end{align*}
Assume that $Q^a\in \mathcal{D}^u$, where $u\in \N$. We
write  $\Sigma_Q$ in terms of the layers of cubes as follows
\begin{align*}
\Sigma_Q=\bigg\|\bigg(\sum_{R\subset Q:R^a=Q^a}+\sum_{j=1}^\infty \sum_{S\subsetneq Q:S\in\mathcal{D}^{u+j}}\sum_{R\subset S:R^a=S}\bigg)\epsilon_RD_Rb_{R^a}\bigg\|_{L^1}.
\end{align*}
The triangle inequality gives
\begin{equation}\label{kaksi}
\begin{split}
\Sigma_Q&\le \bigg\|1_Q\sum_{R\subset Q:R^a=Q^a}\epsilon_RD_Rb_{Q^a}\bigg\|_{L^1} \\&\qquad+ \sum_{j=1}^\infty\sum_{S\subsetneq Q:S\in\mathcal{D}^{u+j}}\bigg\|1_S\sum_{R\subset S:R^a=S}\epsilon_RD_Rb_{S}\bigg\|_{L^1}=:\Sigma_{Q,1}+\Sigma_{Q,2}.
\end{split}
\end{equation}
Observe that
$D_R b_{Q^a} = D_R (1_Q b_{Q^a})$ if $R\subset Q$. Hence, by applying H\"older's inequality and
\eqref{accre},
\begin{align*}
\Sigma_{Q,1}\le \|1_Q\|_{L^2}\bigg\|\sum_{R\subset Q:R^a=Q^a}\epsilon_RD_R (1_Qb_{Q^a})\bigg\|_{L^2}\lesssim \mu(Q)^{1/2}\|1_Qb_{Q^a}\|_{L^2}\lesssim \mu(Q).
\end{align*}
The second term $\Sigma_{Q,2}$ is first estimated in a similar manner. Then we use
Lemma \ref{basicsum} as follows:
\begin{equation}\label{carsumma}
\Sigma_{Q,2}\lesssim \sum_{j=1}^\infty\sum_{S\subsetneq Q:S\in\mathcal{D}^{u+j}}\mu(S)\le \mu(Q)\sum_{j=1}^\infty (1-\tau)^{j-1}\lesssim\mu(Q).
\end{equation}
We have shown that $\Sigma_Q\le \Sigma_{Q,1}+\Sigma_{Q,2}\lesssim \mu(Q)$. 
Collecting the estimates above, we see
that the left hand side of \eqref{caruse} 
is bounded by a constant multiple of $\|f\|_{L^p}$.
\end{proof}

We are now ready for the proof of Theorem \ref{nests}.

\begin{proof}[Proof of Theorem \ref{nests}]
We decompose $D_j^a f$ as follows:
\begin{equation}\label{dec2}
\begin{split}
D_j^a f &= E_{j-1}^a f- E_j^a f= b_{j-1}^a \frac{E_{j-1} f}{E_{j-1} b_{j-1}^a}-b_j^a\frac{E_j f}{E_j b_j^a}\\
&=1_{\{b_{j-1}^a=b_j^a\}} b_j^a\bigg(\underbrace{\frac{E_{j-1}f}{E_{j-1} b_{j-1}^a}-\frac{E_j f}{E_j b_j^a}}_{:=I(j)}\bigg)
\\&\qquad\qquad+1_{\{b_{j-1}^a\not=b_j^a\}} \bigg(\underbrace{ b_{j-1}^a\frac{E_{j-1} f}{E_{j-1} b_{j-1}^a}-b_j^a\frac{E_j f}{E_j b_j^a}}_{:=II(j)}\bigg).
\end{split}
\end{equation}
First, using the contraction principle followed by Lemma
\ref{carut2} yields
\begin{equation}\label{es14}
\begin{split}
&\bigg\|\sum_{j=-\infty}^\infty \epsilon_j 1_{\{b_{j-1}^a=b_j^a\}} b_j^a I(j)\bigg\|_{L^p}\le
\bigg\|\sum_{j=-\infty}^\infty \epsilon_j 1_{\{b_{j-1}^a=b_j^a\}} I(j)\bigg\|_{L^p}\lesssim
\|f\|_{L^p}.
\end{split}
\end{equation}

In order to estimate the remaining quantity, we fix $j\in\Z$ and use  
the identity $E_j f=E_{j-1}f - D_j f$ for
\begin{align*}
II(j)
=\bigg(\frac{b_{j-1}^a}{E_{j-1} b_{j-1}^a}-\frac{b_j^a}{E_j b_j^a}\bigg)E_{j-1}f
+\frac{b_j^a}{E_j b_j^a}D_j f.
\end{align*}
This representation leads to the estimate
\begin{equation}\label{es22}
\begin{split}
&\bigg\|\sum_{j=-\infty}^\infty \epsilon_j 1_{\{b_{j-1}^a\not =b_j^a\}} II(j)\bigg\|_{L^p}\\&\le
\bigg\|\sum_{j=-\infty}^\infty \epsilon_j 1_{\{b_{j-1}^a\not=b_j^a\}} \bigg(\frac{b_{j-1}^a}{E_{j-1} b_{j-1}^a}-\frac{b_j^a}{E_j b_j^a}\bigg)E_{j-1}f \bigg\|_{L^p}\\
&\qquad\qquad+
\bigg\|\sum_{j=-\infty}^\infty \epsilon_j 1_{\{b_{j-1}^a\not=b_j^a\}} \frac{b_j^a}{E_j b_j^a}D_j f \bigg\|_{L^p}.
\end{split}
\end{equation}
The last term above is estimated by using first 
\eqref{accre} and \eqref{a_iso} with the contraction principle, and then followed by the $\mathrm{UMD}$-property of $X$. This results
in the required upper bound  $c\|f\|_{L^p}$ for the term in question. 

Applying the contraction principle to the 
first term in the right hand side of \eqref{es22} yields the estimate
\begin{align*}
&\bigg\|\sum_{j=-\infty}^\infty \epsilon_j 1_{\{b_{j-1}^a\not=b_j^a\}} \bigg(\frac{b_{j-1}^a}{E_{j-1} b_{j-1}^a}-\frac{b_j^a}{E_j b_j^a}\bigg)E_{j-1}f \bigg\|_{L^p}\\&\lesssim \bigg\|\sum_{j=-\infty}^\infty \epsilon_j 1_{\{b_{j-1}^a\not=b_j^a\}}E_{j-1}f \bigg\|_{L^p}.
\end{align*}
Using Lemma \ref{carut}, we see that the last term can be dominated
by $c\|f\|_{L^p}$.

By collecting the estimates beginning from \eqref{es22}, we get 
\begin{equation}\label{2arvio}
\bigg\|\sum_{j=-\infty}^\infty \epsilon_j 1_{\{b_{j-1}^a\not =b_j^a\}} \bigg(b_{j-1}^a\frac{E_{j-1} f}{E_{j-1} b_{j-1}^a}-b_j^a\frac{E_j f}{E_j b_j^a}\bigg)\bigg\|_{L^p}\lesssim \|f\|_{L^p}.
\end{equation}
The required estimate \eqref{thal} follows now by combining the identity \eqref{dec2} with  the estimates \eqref{es14} and \eqref{2arvio}.
\end{proof}

\subsection*{Estimate for the adjoints $(D_k^a)^*$}
Here we prove a norm estimate under  the $\mathrm{UMD}$ function lattice assumption.
The need for this assumption was somewhat unexpected to us, but with our present techniques, we were unable to avoid it. Proving (or disproving) the dual square-function estimate in the absence of the lattice assumption would be an interesting question for a deeper understanding of the vector-valued theory.
 
\begin{lause}\label{haa}
Let $X$ be a $\mathrm{UMD}$ function lattice and $1<p<\infty$. Then
\begin{equation}\label{eeds}
\bigg\|\sum_{k\in\Z} \epsilon_k (D_k^a)^* f\bigg\|_{L^p(\R^N\times\Omega,\mu\otimes\mathbf{P},X)} \lesssim \|f\|_{L^p(\R^N,\mu;X)}
\end{equation}
for every $f\in L^p(\R^N,\mu;X)$.
\end{lause}

In order to prove Theorem \ref{haa}, we first prove the following Carleson embedding for $\mathrm{UMD}$ function lattices.

\begin{prop}\label{ylcarl}
Let $X$ be a $\mathrm{UMD}$ function lattice  and $1<p<\infty$. Let \[\{d_k\in L^1(\R^N,\mu;\R)\}_{k\in\Z},\qquad \{c_k\in L^\infty(\R^N,\mu;\R)\}_{k\in\Z},\] be such that $d_k=E_k d_k$ and $\|c_k\|_{L^\infty}\le 1$ for every $k\in\Z$. Then
\begin{equation}\label{adesa}
\bigg\|\sum_{k\in\Z} \epsilon_k d_k E_k(c_k f)\bigg\|_{L^p(\R^N\times\Omega,\mu\otimes \mathbf{P},X)}\lesssim\|\{d_k\}_{k\in\Z}\|_{\mathrm{Car}^1(\mathcal{D})}\|f\|_{L^p(\R^N,\mu;X)}
\end{equation}
for every $f\in L^p(\R^N,\mu;X)$.
\end{prop}

\begin{proof}
Since $X$ is a lattice of functions, for $\xi\in X$, we can consider its pointwise absolute value $|\xi|\in X$, which satisfies $\|\xi\|_X=\|\,|\xi|\,\|_X$. Moreover, we have the following inequalities, which can be seen by the argument on \cite[p.~212]{RdF}:
\[
\int_\Omega \bigg|\sum_{k\in\Z} \epsilon_k \xi_k\bigg|_X^pd\mathbf{P}(\epsilon)
\eqsim \bigg|\bigg(\sum_{k\in\Z} |\xi_k|^2\bigg)^{1/2}\bigg|_X^p
\eqsim \int_\Omega \bigg|\sum_{k\in\Z} \epsilon_k |\xi_k|\bigg|_X^pd\mathbf{P}(\epsilon).
\]
Since $E_k d_k = d_k$, we have $E_k |d_k|= |d_k|$. Hence,
by the Carleson embedding Theorem
\ref{carleson}, we can estimate the left hand side of \eqref{adesa} as follows
\begin{align*}
LHS\eqref{adesa}
&\eqsim \bigg\|\bigg(\sum_{k\in\Z} |d_k E_k(c_k f)|^2\bigg)^{1/2}\bigg\|_{L^p(\R^N,\mu,X)}
\\&\lesssim \bigg\|\bigg(\sum_{k\in\Z} \big(|d_k| E_k |f|\big)^2\bigg)^{1/2}\bigg\|_{L^p(\R^N,\mu,X)}
\\&\eqsim \bigg\|\sum_{k\in\Z} \epsilon_k |d_k| E_k |f|\bigg\|_{L^p(\R^N,\mu,X)}
\lesssim \|\{d_k\}_{k\in\Z}\|_{\mathrm{Car}^1(\mathcal{D})}\big|\big| |f| \big|\big|_{L^p(\R^N,\mu,X)}.
\end{align*}
The required estimate now follows because $\big|\big| |f| \big|\big|_p=\|f\|_p$.
\end{proof}


We also need  the following notation and representation formulae.

If $k\in \Z$, we write $Q_k(x)$ for the unique cube in $\mathcal{D}_k$ containing
the point $x\in \R^N$.

For $n\in\N_0$ and $x\in\R^N$, we denote by $Q^n(x)\in\mathcal{D}^n$  the cube in the
$n$th layer that contains the point $x$ (if such a cube exists), see Section
\ref{layers}.
We also denote $\sigma_n(x)=\log_2(\ell(Q^n(x))$ if
$x\in Q^n(x)\in \mathcal{D}^n$ and $\sigma_n(x)=-\infty$ if there is
no cube in $\mathcal{D}^n$ which contains the point $x$.

If $x\in \cup_{Q\in\mathcal{D}^n} Q$, we denote
$b^n(x)=b_{Q^n(x)}(x)$.
Note that, for $k\le s$ (recall that $Q_0\in \mathcal{D}_s$) and $x\in Q_0$, 
\begin{equation}\label{hdh}
\sigma_{n+1}(x)<k \le \sigma_n(x)\iff (Q_k(x))^a=Q^n(x)\in\mathcal{D}^n.
\end{equation}
In particular, if \eqref{hdh} is valid, then 
$b_k^a(x)=b_{(Q_k(x))^a}(x) = b^n(x)$. 

We also denote \[E_{\sigma_n} = \sum_{Q\in\mathcal{D}^n} E_Q,\qquad n\in \{0,1,\ldots\}.\]

\begin{lem}\label{yksiks}
Let $x\in\R^N$ and $p\in (1,\infty)$. Then
\begin{align*}
&\int_\Omega \bigg|\sum_{n=0}^\infty \epsilon_n E_{\sigma_n}(b^n f)(x)\bigg|^pd\mathbf{P}(\epsilon)
\\&\lesssim 
\int_\Omega \bigg|\sum_{k\in\Z} \epsilon_k 1_{\{b_{k-1}^a \not= b_k^a\}}(x) E_{k-1}(b_{k-1}^a f)(x)\bigg|^pd\mathbf{P}(\epsilon)
+|E_{Q_0}(b_{Q_0} f)(x)|^p
\end{align*}
for every 
$f\in L^p(\R^N,\mu;X)$.
\end{lem}

\begin{proof}
Denote $A_k=\mathcal{D}_{k}\cap (\cup_{n=1}^\infty \mathcal{D}^n)$. Let $\epsilon\in \Omega$ and
$x\in\R^N$. By \eqref{bmaar},
\begin{equation}\label{ess}
\begin{split}
\sum_{k\in\Z} \epsilon_k 1_{\{b_{k-1}^a \not= b_k^a\}}(x) E_{k-1}(b_{k-1}^a f)(x)&=\sum_{k\in\Z} \sum_{Q\in A_{k-1}} \epsilon_kE_Q(b_{Q^a}f)(x)\\
&=\sum_{n=1}^\infty \epsilon_{\sigma_n(x)+1} \sum_{Q\in\mathcal{D}^n} E_Q(b_{Q^a}f)(x).
\end{split}
\end{equation}
If $x\in Q\in\mathcal{D}^n$, then $Q=Q^a =Q^n(x)$. Hence, $1_Q b_{Q^a} f = 1_Q b^nf$, and applying the last
identity
to the right hand side of \eqref{ess}, we get
\begin{align*}
LHS\eqref{ess}=\sum_{n=1}^\infty \epsilon_{\sigma_n(x)+1} \sum_{Q\in\mathcal{D}^n} E_Q(b^nf)(x) =\sum_{n=1}^\infty \epsilon_{\sigma_n(x)+1} E_{\sigma_n} (b^n f)(x).\end{align*}
The required identity follows by taking $p$-absolute values, integrating, and relabeling the random variables $\epsilon_{\sigma_n(x)+1}$.
\end{proof}

\begin{lem}\label{kaksiks}
Let $x\in\R^N$ and $p\in (1,\infty)$. Then
\begin{align*}
&\int_{\Omega} \bigg|\sum_{n=0}^\infty \epsilon_nE_{\sigma_{n+1}}(b^n f)(x)\bigg|^pd\mathbf{P}(\epsilon)
=\int_\Omega \bigg|\sum_{k\in\Z} \epsilon_k 1_{\{b_{k-1}\not=b_k\}}(x) E_{k-1}(b_k^af)(x)\bigg|^p d\mathbf{P}(\epsilon)
\end{align*}
for every $f\in L^p(\R^N,\mu;X)$.
\end{lem}

\begin{proof}
Let $A_k$ be as in the previous proof.
Let
$x\in\R^N$ and $\epsilon\in\Omega$. Then, if $x\in Q\in\mathcal{D}^{n}$ with $n\ge 1$, we 
have $x\in (Q^{(1)})^a\in\mathcal{D}^{n-1}$, and therefore
\[b^{n-1}(x)=b_{Q^{n-1}(x)}(x)=b_{(Q^{(1)})^a}(x).\] Using this identity, we get
\begin{align*}
&\sum_{k\in\Z} \epsilon_k 1_{\{b_{k-1}\not=b_k\}}(x) E_{k-1}(b_k^af)(x) = \sum_{k\in\Z} \epsilon_k\sum_{Q\in A_{k-1}}
E_Q(b_{(Q^{(1)})^a}f)(x)\\
&=
\sum_{n=1}^\infty \epsilon_{\sigma_{n}(x)+1}\sum_{Q\in\mathcal{D}^{n}} E_Q(b^{n-1} f)(x)=\sum_{n=0}^\infty \epsilon_{\sigma_{n+1}(x)+1} E_{\sigma_{n+1}}(b^n f)(x).
\end{align*}
The required estimate follows by taking $p$-absolute values, integrating, and relabeling the random
variables $\sigma_{n+1}(x)+1$.
\end{proof}

We are ready for the proof of Theorem \ref{haa}.

\begin{proof}[Proof of Theorem \ref{haa}]
By \eqref{vaid}, we get
\begin{equation}
\begin{split}
(D_k^a)^* f&= \frac{E_{k-1}(b_{k-1}^a f)}{E_{k-1} b_{k-1}^a}-\frac{E_k(b_k^a f)}{E_k b_k^a}\\
&= \frac{E_k b_k^a-E_{k-1} b_{k-1}^a}{E_k b_k^a E_{k-1}b_{k-1}^a}E_{k-1}(b_{k-1}^a f)+\frac{E_{k-1}(b_{k-1}^a f) - E_k(b_k^a f)}{E_k b_k^a}.
\end{split}
\end{equation}
Denote $d_k=E_k b_k^a-E_{k-1} b_{k-1}^a$ for $k\in \Z$. Then $|d_k|\le 2$ and $E_{k-1}d_k=d_k$.
By \eqref{a_iso}, we have $|E_k b_k^a E_{k-1}b_{k-1}^a|\ge \delta^4$ for $\mu$-almost every point in $\R^N$.
Hence, by the contraction principle and Proposition \ref{ylcarl}, we obtain 
\begin{equation}\label{ekasss}
\begin{split}
\bigg\| \sum_{k\in\Z} \epsilon_k\frac{E_k b_k^a-E_{k-1} b_{k-1}^a}{E_k b_k^a E_{k-1}b_{k-1}^a}E_{k-1}(b_{k-1}^a f)\bigg\|_{L^p(\R^N\times\Omega,\mu\otimes \mathbf{P},X)}
\lesssim \|\{d_k\}_{k\in\Z}\|_{\mathrm{Car}^1(\mathcal{D})}\|f\|_{L^p(\R^N,\mu;X)}.
\end{split}
\end{equation}
Let us prove that $\|\{d_k\}_{k\in\Z}\|_{\mathrm{Car}^1(\mathcal{D})}\lesssim 1$.
Let $Q\in\mathcal{D}$ be such that $\mu(Q)\not=0$. Then $1_{\{b_{k-1}^a = b_k^a\}}d_k= -1_{\{b_{k-1}^a = b_k^a\}} D_k b_k^a$, and
therefore
\begin{equation}\label{sssd}
\begin{split}
\bigg\| 1_Q \sum_{2^k\le \ell(Q)} \epsilon_k d_k\bigg\|_{L^1(\R^N\times\Omega,\mu\otimes\mathbf{P};\R)}
\lesssim \bigg\| 1_Q \sum_{2^k\le \ell(Q)} \epsilon_k 1_{\{b_{k-1}^a\not=b_k^a\}}d_k\bigg\|_1
+\bigg\|\sum_{P\subset Q} \epsilon_P D_P b_{P^a}\bigg\|_1.
\end{split}
\end{equation}

The first term on the right hand side of \eqref{sssd} is first estimated 
by using contraction principle. Then proceeding as in the proof of Lemma \ref{carut}, gives
the upper bound $c\mu(Q)$ for that  term.

The second term on the right hand side of \eqref{sssd} is estimated as in \eqref{kaksi}, with the same upper bound $c\mu(Q)$. Combining the estimates above, we find that
$\|\{d_k\}_{k\in\Z}\|_{\mathrm{Car}^1(\mathcal{D})}\lesssim 1$.

In order to estimate the second term on the right hand side of \eqref{sssd}, we denote
\begin{align*}
\Sigma_Q=\bigg\|\sum_{P\subset Q} \epsilon_P D_P b_{P^a}\bigg\|_1.
\end{align*}
Then estimating $\Sigma_Q$ as in \eqref{kaksi}, we find that
$\Sigma_Q\lesssim \mu(Q)$. Combining the estimates above, we find that
$\|\{d_k\}_{k\in\Z}\|_{\mathrm{Car}^1(\mathcal{D})}\lesssim 1$. 

It remains to prove that
\begin{equation}\label{remsss}
\bigg\| \sum_{k\in\Z} \epsilon_k\frac{E_{k-1}(b_{k-1}^a f) - E_k(b_k^a f)}{E_k b_k^a}\bigg\|_{L^p(\R^N\times\Omega,\mu\otimes \mathbf{P},X)}
\lesssim \|f\|_{L^p(\R^N,\mu;X)}.
\end{equation}
By \eqref{a_iso},  we have $|E_k b_k^a|\ge \delta^2$ for $\mu$-almost every point in $\R^N$.
Using the contraction principle we eliminate the
terms $1/E_k b_k^a$ from the left hand side of \eqref{remsss}. Then we consider  
the following decomposition, where $k\in\Z$,
\begin{equation}\label{edss}
\begin{split}
E_{k-1}(b_{k-1}^a f)-E_k(b_k^a f) &= E_{k-1}\big(\chi_{k-1} (b_{k-1}^a-b_k^a)f\big)
+(E_{k-1}-E_k)(b_k^a f)\\
&=\chi_{k-1} E_{k-1}\big((b_{k-1}^a - b_k^a) f\big) + D_k(b_k^a f),
\end{split}
\end{equation}
where we have denoted $\chi_{k-1}=1_{\{b_{k-1}^a \not=b_k^a\}}$.

Using Proposition \ref{ylcarl}, we obtain the following norm-estimate involving the first term on the right hand side
of \eqref{edss}
\begin{equation}
\begin{split}
&\bigg\| \sum_{k\in\Z} \epsilon_k \chi_{k-1} E_{k-1}\big((b_{k-1}^a - b_k^a) f\big)\bigg\|_{L^p(\R^N\times\Omega,\mu\otimes \mathbf{P},X)}
\\&\qquad\qquad\qquad\qquad\lesssim \|\{\chi_k\}_{k\in\Z}\|_{\mathrm{Car}^1(\mathcal{D})}\|f\|_{L^p(\R^N,\mu;X)}.
\end{split}
\end{equation}
On the other hand, the proof of Lemma \ref{carut} shows that
$\|\{\chi_k\}_{k\in\Z}\|_{\mathrm{Car}^1(\mathcal{D})}\lesssim 1$.

In order to complete the proof of \eqref{remsss}, we still need to prove the following estimate involving the second term
on the right hand side of \eqref{edss},
\begin{equation}\label{vaikea}
\begin{split}
&\bigg\| \sum_{k\in\Z} \epsilon_kD_k(b_k^a f)\bigg\|_{L^p(\R^N\times\Omega,\mu\otimes \mathbf{P},X)}
\lesssim \|f\|_{L^p(\R^N,\mu;X)}.
\end{split}
\end{equation}
For this purpose, we introduce independent Rademacher variables $\tilde \epsilon\in(\tilde\Omega,\tilde{\mathbf{P}})$.
For $x\in Q_0$ and $k\le s$, we 
denote $\tilde\epsilon_k^a(x) = \tilde \epsilon_n$ if $n$ is such that $\sigma_{n+1}(x)<k \le \sigma_n(x)$.
By the fact that $\mu(\R^N\setminus Q_0)=0$ and \eqref{hdh}, we find that
the functions $\tilde\epsilon_k^a$, for $k\le s$, are defined $\mu$-almost everywhere and they are $\mathcal{D}_k$-measurable.

Then, for every $x\in Q_0$ and $\tilde\epsilon\in\tilde\Omega$, we have
\[
\int_\Omega \bigg|\sum_{k\le s} \epsilon_k D_k(b_k^a f)(x)\bigg|_X^pd\mathbf{P}(\epsilon)=
\int_\Omega \bigg|\sum_{k\le s} \epsilon_k D_k(\tilde \epsilon_k^ab_k^a f)(x)\bigg|_X^pd\mathbf{P}(\epsilon).
\]
Recall that $D_k b_k^a = 0$ $\mu$-almost everywhere if $k>s$ and $\mu(\R^N\setminus Q_0)=0$. Hence,
by integrating, and using the UMD-property of $X$ and measurability of $\tilde \epsilon_k^a$, we obtain
\begin{align*}
LHS\eqref{vaikea}&=
\bigg\|\sum_{k\le s}Ê \epsilon_k D_k(\tilde \epsilon_k^a b_k^a f)\bigg\|_{L^p(\R^N\times\Omega,\mu\otimes \mathbf{P},X)} \\&\lesssim \bigg\|\sum_{k\le s} D_k(\tilde \epsilon_k^a b_k^a f)\bigg\|_{L^p(\R^N,\mu;X)}
=\bigg\|\sum_{k\le s} \tilde \epsilon_k^aD_k( b_k^a f)\bigg\|_{L^p(\R^N,\mu;X)}.
\end{align*}
Reindexing the last sum gives
\begin{align*}
LHS\eqref{vaikea}&\lesssim
\bigg\|\sum_{n=0}^\infty \sum_{k\le s} 1_{\sigma_{n+1}< k\le \sigma_n} \tilde \epsilon_k^aD_k(b_k^a f)\bigg\|_{L^p(\R^N,\mu;X)}.
\\
&\lesssim\bigg\|\sum_{n=0}^\infty \tilde\epsilon_n\sum_{k\le s} 1_{\sigma_{n+1}< k\le \sigma_n}D_k(b^n f)\bigg\|_{L^p(\R^N,\mu;X)}.
\end{align*}
In the last inequality we used the fact that the indicators $x\mapsto 1_{\sigma_{n+1}(x)< k\le \sigma_n(x)}$ 
are $\mathcal{D}_k$-measurable by \eqref{hdh}.
Taking the expectation over $\tilde\epsilon\in\tilde\Omega$, we find that
\begin{equation}\label{jatkettu}
LHS\eqref{vaikea}\lesssim \bigg\|\sum_{n=0}^\infty \tilde \epsilon_n\sum_{k\le s} 1_{\sigma_{n+1}< k\le \sigma_n}D_k(b^n f)\bigg\|_{L^p(\R^N\times\{0,1\}^{\N_0},\mu\otimes{\tilde{\mathbf{P}}},X)}.
\end{equation}
By \eqref{hdh} and martingale convergence,
\begin{align*}
\sum_{k\le s} 1_{\sigma_{n+1}< k\le \sigma_n}D_k &= \sum_{Q\,:\,Q^a\in \mathcal{D}^n} D_Q = \sum_{Q\,:\,Q^a\in \cup_{m\ge n} \mathcal{D}^m} D_Q - \sum_{Q\,:\,Q^a \in \cup_{m>n} \mathcal{D}^m} D_Q\\& = (1_{\cup\mathcal{D}^n} - E_{\sigma_n}) - (1_{\cup \mathcal{D}^{n+1}}-E_{\sigma_{n+1}})
=E_{\sigma_{n+1}}-E_{\sigma_n}+1_{\cup\mathcal{D}^n\setminus \cup \mathcal{D}^{n+1}}.
\end{align*}
Apply this operator identity to  $b^n f\in L^p(\R^N,\mu;X)$ and substitute the resulting function
to the right hand side of \eqref{jatkettu}.  Using the triangle inequality
results in three terms; one of them can be  estimated (using that
$\cup\mathcal{D}^n\supset \cup\mathcal{D}^{n+1}$  and $\|b^n\|_\infty\lesssim 1$ if $n\in\{0,1,\ldots\}$) as follows
\[
\bigg\|\bigg(\underbrace{\sum_{n=0}^\infty \tilde \epsilon_n1_{\cup\mathcal{D}^n\setminus \cup \mathcal{D}^{n+1}} b^n}_{\lesssim  1}\bigg) f\bigg\|_{L^p(\R^N\times\{0,1\}^{\N_0},\mu\otimes{\tilde{\mathbf{P}}},X)}\lesssim \|f\|_{L^p(\R^N,\mu;X)}.
\]
The two remaining terms can be first estimated by using lemmata \ref{yksiks} and \ref{kaksiks}, and
then invoking
Proposition \ref{ylcarl}. This leads to the upper bound
 \[
c\|\{1_{\{b_{k-1}^a \not= b_k^a\}}\}_{k\in\Z}\|_{\mathrm{Car}^1(\mathcal{D})}\|f\|_{L^p(\R^N,\mu;X)}+c\|E_{Q_0}(b_{Q_0}f)\|_{L^p(\R^N,\mu;X)}\lesssim \|f\|_{L^p(\R^N,\mu;X)}.
 \]
for these two terms. 

Combining the estimates above yields
$LHS\eqref{vaikea}\lesssim \|f\|_{L^p(\R^N,\mu;X)}$.
\end{proof}

\subsection*{A norm equivalence for $\mathrm{UMD}$ function lattices}
We conclude this section  by establishing a certain square-function norm equivalence in $L^p(\R^N,\mu;X)$, which follows from the square-function estimates for the operators $D_j^a$ and $(D_j^a)^*$ studied earlier. This norm equivalence will not be exploited in the proof of the main Theorem~\ref{mainth}, but it is recorded here for the sake of curiosity.

We begin with the following lemma.

\begin{lem}\label{tech}
Assume that $X$ (and then also $X^*$) is a $\mathrm{UMD}$ function lattice  and let $p\in (1,\infty)$. Suppose that $\{\lambda_k\in L^1(\R^N,\mu;\R)\}_{k\in\Z}$ is such that $\lambda_k=E_k \lambda_k$ for every $k\in\Z$. Let $f\in L^p(\R^N,\mu;X)$. Then
\begin{align*}
\bigg\| \sum_{k\in\Z} \lambda_k D_k^a f\bigg\|_{L^p(\R^N,\mu;X)} &\lesssim \bigg\|\sum_{k\in\Z} \epsilon_k \lambda_k D_k^a f\bigg\|_{L^p(\R^N\times\Omega,\mu\otimes\mathbf{P},X)}\\&\qquad + \bigg\|\sum_{k\in\Z}Ê\epsilon_k\lambda_k\chi_{k-1}E_k f\bigg\|_{L^p(\R^N\times\Omega,\mu\otimes\mathbf{P},X)},
\end{align*}
where we have denoted $\chi_{k-1}=1_{\{b_{k-1}^a\not=b_{k}^a\}}$.
\end{lem}

\begin{proof}
Let $q\in (1,\infty)$ be such that $1/p+1/q=1$. Using the
identities $\lambda_kD_k^a=D_k^a\lambda_k$ and $(D_k^a)^2 f=D_k^a f+\omega_k^a E_k f$ from \eqref{qr}, we get
\begin{equation}\label{duaid}
\begin{split}
&\bigg\| \sum_{k\in\Z} \lambda_k D_k^a f\bigg\|_{L^p(\R^N,\mu;X)}  \\&= \sup_{\|g\|_{L^q(\R^N,\mu;X^*)} \le 1} \bigg|\langle g,\sum_{k\in\Z} \lambda_kD_k^a f\rangle\bigg|
=\sup_{\|g\|_{q} \le 1} \bigg|\langle g,\sum_{k\in\Z} \lambda_k[(D_k^a)^2 - \omega_k^a E_k] f\rangle\bigg|\\
&\le\sup_{\|g\|_q\le 1} \bigg\{ \bigg|\sum_{k\in\Z} \langle (D_k^a)^* g,\lambda_kD_k^a f\rangle\bigg|
+\bigg|\sum_{k\in\Z} \langle g,\lambda_k\omega_k^a E_k f\rangle\bigg|\bigg\}.
\end{split}
\end{equation}
The first term inside the last supremum is estimated by
using Theorem \ref{haa} and the assumption $\|g\|_{L^q(\R^N,\mu;X^*)} \le 1$,
\begin{align*}
&\bigg|\sum_{k\in\Z} \langle (D_k^a)^* g,\lambda_kD_k^a f\rangle\bigg|
=\bigg| \int_\Omega \bigg\langle  \sum_{\ell\in\Z} \epsilon_\ell (D_\ell^a)^*g,\sum_{k\in\Z} \epsilon_k \lambda_k D_k^a f\bigg\rangle d\mathbf{P}(\epsilon)\bigg|\\
&\le \bigg\| \sum_{\ell\in\Z}\epsilon_\ell(D_\ell^a)^*g\bigg\|_q\bigg\|\sum_k \epsilon_k\lambda_kD_k^a f\bigg\|_p
\lesssim \bigg\|\sum_k \epsilon_k\lambda_k D_k^a f\bigg\|_p.
\end{align*}
Then we estimate the second term inside last supremum in \eqref{duaid}.
By Lemma \ref{ombasic}, the fact that $\lambda_k\chi_{k-1}$ is $\sigma(\mathcal{D}_{k-1})$ measurable, and identity 
$E_{k-1}E_k = E_k$, we get
\begin{align*}
\bigg|\sum_{k\in\Z} \langle g,\lambda_k \omega_k^a E_k f\rangle\bigg|
&=\bigg|\sum_{k\in\Z} \langle \chi_{k-1} E_{k-1}(\omega_k^a g),\lambda_k\chi_{k-1} E_k f\rangle\bigg|\\
&=\bigg|\int_\Omega \bigg\langle \sum_{\ell\in\Z} \epsilon_\ell \chi_{\ell-1} E_{\ell-1} (w_\ell^a g),\sum_{k\in\Z} \epsilon_k \lambda_k\chi_{k-1} E_k f\bigg\rangle d\mathbf{P}(\epsilon)\bigg|\\
&\le \bigg\|\sum_{\ell\in\Z} \epsilon_\ell\chi_{\ell-1}E_{\ell-1}(w_\ell^a g)\bigg\|_q
\bigg\|\sum_{k\in\Z} \epsilon_k \lambda_k\chi_{k-1} E_k f\bigg\|_p \\&\lesssim  \bigg\|\sum_{k\in\Z} \epsilon_k \lambda_k\chi_{k-1} E_k f\bigg\|_p,
\end{align*}
where the last step is justified by using Proposition \ref{ylcarl} and the estimate $\|w_\ell^a\|_\infty+\|\{\chi_{\ell}\}_{\ell\in\Z}\|_{\mathrm{Car}^1(\mathcal{D})}\lesssim 1$; see the proof of Lemma
\ref{carut}.
\end{proof}

As a consequence we obtain the following
norm-equivalence.

\begin{lause}\label{normequ}
Assume that $X^*$ is a $\mathrm{UMD}$ function lattice and  let
$f\in L^p(\R^N,\mu;X)$, $p\in (1,\infty)$. Then
\begin{align*}
&\|f\|_{L^p(\R^N,\mu;X)}\eqsim \|E_{Q_0}^a f\|_{f\in L^p(\R^N,\mu;X)}  \\& +\bigg\|\sum_{k\in\Z} \epsilon_k D_k^a f\bigg\|_{{L^p(\R^N\times\Omega,\mu\otimes\mathbf{P},X)}} + \bigg\|\sum_{k\in\Z}Ê\epsilon_k
\chi_{k-1}E_{k} f\bigg\|_{{L^p(\R^N\times\Omega,\mu\otimes\mathbf{P},X)}},
\end{align*}
where $\chi_{k-1}=1_{\{b_{k-1}^a\not=b_{k}^a\}}$ and $E_{Q_0}^a f=b_{Q_0}\frac{\langle f\rangle_{Q_0}}{\langle b_{Q_0}\rangle_{Q_0}}$.
\end{lause}

\begin{proof}
Let $q\in (1,\infty)$ be such that $1/p+1/q=1$. By
\eqref{dec},
\begin{align*}
\|f\|_p &=  \bigg\|E_{Q_0}^a f+\sum_{k\in\Z} D_k^a f\bigg\|_p\le \|E_{Q_0}^af\|_p +\bigg\|\sum_{k\in\Z} D_k^a f\bigg\|_p
\end{align*}
The upper bound follows  from Lemma \ref{tech} with $\lambda_k=1$, $k\in\Z$.
For lower bound, we first have $\|E_{Q_0}^af\|_p\lesssim \|f\|_p$. The two
remaining  terms are then estimated by using Theorem \ref{nests} and Lemma \ref{carut};
for the last term we first write $E_k = -D_k +E_{k-1}$ and use the $\mathrm{UMD}$ property.
\end{proof}

\section{Decomposition of a Calder\'on--Zygmund operator}\label{dec_cald}
Let $T\in\mathcal{L}(L^p(\R^N,\mu;X))$ be a Calder\'on--Zygmund operator as in Theorem \ref{mainth}.
We establish the following estimate:
\begin{equation}\label{firstes}
|\langle g,Tf\rangle|\lesssim\|g\|_q\|f\|_p + \bigg|\sum_{Q\in\mathcal{D},\,R\in\mathcal{D}'} \langle D_R^{a,2} g,T(D_Q^{a,1} f)\rangle\bigg|,
\end{equation}
where $f\in L^p(\R^N,\mu;X)$ and $g\in L^q(\R^N,\mu,X^*)$. 

Estimate \eqref{firstes}
is uniform over all dyadic systems $\mathcal{D}=\mathcal{D}(\beta)$ and $\mathcal{D}'=\mathcal{D}(\beta')$, and
it is based on  decomposition of functions, treated in Section \ref{decfun}. In the subsequent
sections various good parts of the series
on the right hand side of \eqref{firstes} will be estimated. In Section
\ref{synthesis}, we finish the proof of Theorem \ref{mainth} by collecting
estimates of good parts, and also performing an estimate for the remaining bad part.

In order to prove \eqref{firstes}, we
recall the basic cubes $Q_0\in\mathcal{D}$ and $R_0\in\mathcal{D}'$ satisfying \eqref{rsubs}.
The $L^\infty$-accretive systems for $T$ and $T^*$, respectively, are denoted by $\{b^1_{Q}\}_{Q\in\mathcal{D}}$ and $\{b^2_R\}_{R\in\mathcal{D}'}$.
Let $q\in (1,\infty)$ be such that $p^{-1}+q^{-1}=1$ and let $f\in L^p(\R^N,\mu;X)$, $g\in L^q(\R^N,\mu; X^*)$. 

According to \eqref{dec}, and the reasoning therein, 
we have the decompositions
\begin{equation}\label{des2}
\begin{split}
&f-b^1_{Q_0}\frac{\langle f\rangle_{\R^N}}{\langle b^1_{Q_0}\rangle_{\R^N}}
=\sum_{j=-\infty}^\infty D_j^{a,1} f=\sum_{Q\in\mathcal{D}} D_Q^{a,1} f;\\
&g-b^2_{R_0}\frac{\langle g\rangle_{\R^N}}{\langle b^2_{R_0}\rangle_{\R^N}}
=\sum_{j=-\infty}^\infty D_j^{a,2} g=\sum_{R\in\mathcal{D}'} D_R^{a,2} g,
\end{split}
\end{equation}
which converge in $L^p$ and $L^q$, respectively. As a consequence, we see that
\begin{equation}\label{tarv}
\begin{split}
&\langle g, Tf\rangle = \langle g,T(\sum_{Q\in\mathcal{D}} D_Q^{a,1} f)\rangle+\langle g,T(b^1_{Q_0}\frac{\langle f\rangle_{\R^N}}{\langle b^1_{Q_0}\rangle_{\R^N}} )\rangle\\
&=
\langle \sum_{R\in\mathcal{D}'} D_R^{a,2} g,T(\sum_{Q\in\mathcal{D}} D_Q^{a,1} f)\rangle+
\langle T^*(b^2_{R_0}\frac{\langle g\rangle_{\R^N}}{\langle b^2_{R_0}\rangle_{\R^N}}),\sum_{Q\in\mathcal{D}} D_Q^{a,1} f\rangle
+\langle g,T(b^1_{Q_0}\frac{\langle f\rangle_{\R^N}}{\langle b^1_{Q_0}\rangle_{\R^N}} )\rangle.
\end{split}
\end{equation}
Using the H\"older's inequality and \eqref{accreass} for $T^*$, we have
\begin{equation}\label{es} 
\begin{split}
\bigg\|T^*(b^2_{R_0}\frac{\langle g\rangle_{\R^N}}{\langle b^2_{R_0}\rangle_{\R^N}})\bigg\|_q
&=\frac{|\langle g\rangle_{\R^N}|_{X^*}}{|\langle b^2_{R_0}\rangle_{\R^N}|}\|T^*(b^2_{R_0})\|_q\\
&\lesssim \frac{1}{\mu(R_0)}\int_{R_0} |g(x)|_{X^*}dx\cdot\|T^*(b^2_{R_0})\|_\infty\cdot\mu(R_0)^{1/q}\\
&\lesssim B\mu(R_0)^{1/q}\bigg(\frac{1}{\mu(R_0)}\int_{R_0}|g(x)|_{X^*}^qdx\bigg)^{1/q}\lesssim B\|g\|_{q}.
\end{split}
\end{equation}
In a similar manner, we have 
\begin{equation}\label{tsam}
\|b^1_{Q_0}\frac{\langle f\rangle_{\R^N}}{\langle b^1_{Q_0}\rangle_{\R^N}}\|_p\lesssim \|f\|_p.
\end{equation}
As a consequence of \eqref{es}, \eqref{tsam}, and \eqref{des2}, we have
\begin{align*}
|\langle T^*(b^2_{R_0}\frac{\langle g\rangle_{\R^N}}{\langle b^2_{R_0}\rangle_{\R^N}}),\sum_{Q\in\mathcal{D}} D_Q^{a,1} f\rangle|
\lesssim  B\|g\|_q\|f-b^1_{Q_0}\frac{\langle f\rangle_{\R^N}}{\langle b^1_{Q_0}\rangle_{\R^N}}\|_p\lesssim\|g\|_q\|f\|_p.
\end{align*}
Computing as above, we also find that
\begin{align*}
|\langle g,T(b^1_{Q_0}\frac{\langle f\rangle_{\R^N}}{\langle b^1_{Q_0}\rangle_{\R^N}} )\rangle|\lesssim\|g\|_q\|f\|_p.
\end{align*}
Combining the estimates above gives us \eqref{firstes}.
Within the summation on its right hand side we can tacitly assume that
the summation varies over cubes for which $Q\subset Q_0$ and $R\subset R_0$. Indeed,
otherwise $D_R^{a,2} g=0$ or $D_{Q}^{a,1} f=0$.






\section{Decoupling estimates}\label{decoul}

We begin with the following {\em tangent martingale trick} originating from \cite{McConnell}, and formulated in a way convenient for our purposes in \cite{hytonen1}.
Let $(E,\mathcal{M},\mu)$ be a $\sigma$-finite measure space having a refining
sequence of partitions as follows: For each $k\in\Z$, let
$\mathcal{A}_k$ be a countable partition of $E$ into sets of finite
positive measure so that
$\sigma(\mathcal{A}_k)\subset \sigma(\mathcal{A}_{k-1})\subset \mathcal{M}$,
and let $\mathcal{A}=\cup_{k\in\Z} \mathcal{A}_k$. 

For each $A\in\mathcal{A}$, let $\nu_A$ denote
the probability measure $\mu(A)^{-1}\cdot \mu|_A$. Let
$(F,\mathcal{N},\nu)$ be the space
$\prod_{A\in\mathcal{A}} A$ with the
product $\sigma$-algebra and measure. Its points
will be denoted by $y=(y_A)_{A\in\mathcal{A}}$.
By \cite[Theorem 6.1]{hytonen1}, the
following norm equivalence holds:

\begin{lause}\label{tangent} Suppose that $X$ is a $\mathrm{UMD}$ space
and $p\in (1,\infty)$. Then
\begin{align*}
&\iint_{E\times \Omega}\bigg|\sum_{k\in\Z} \epsilon_k \sum_{A\in\mathcal{A}_k} f_A(x) \bigg|_X^p d\mathbf{P}(\epsilon)\, d\mu(x)\\
&\qquad \eqsim \iiint_{F\times E\times \Omega} \bigg|\sum_{k\in\Z} \epsilon_k
\sum_{A\in\mathcal{A}_k} 1_A(x) f_A(y_A)\bigg|_X^p d\mathbf{P}(\epsilon)\,d\mu(x)\,d\nu(y).
\end{align*}
\end{lause}

As a consequence, we obtain the
following extension of \cite[Corollary 6.3]{hytonen1}.

\begin{lause}\label{trick}
Let $X$ be a $\mathrm{UMD}$ space and $p\in (1,\infty)$. For each $A\in\mathcal{A}$, let
\[
k_A:A\times A\to \mathcal{L}(X)
\]
be a jointly measurable function for which there is a constant $C>0$ such that
\begin{equation}\label{rbound}
\mathcal{R}\big(\{k_A(x,y_A)\,:\,x\in A\in\mathcal{A}\}\big)\le C<\infty,\qquad \text{ if }x\in E\text{ and }y\in F.
\end{equation}
Suppose also that, for each $A\in\mathcal{A}_k$ with $k\in\Z$
we are given a $\sigma(\mathcal{A}_{k-1})$-measurable
function $f_A:E\to X$, supported
on $A$. Then
\begin{align}\label{lefts}
&\iint_{E\times \Omega} \bigg|\sum_{k\in\Z} \epsilon_k \sum_{A\in\mathcal{A}_k} \frac{1_A(x)}{\mu(A)}
\int_A k_A(x,z)f_A(z)d\mu(z)\bigg|_X^pd\mathbf{P}(\epsilon)d\mu(x)\\
&\qquad \lesssim C\iint_{E\times \Omega} \bigg|\sum_{k\in\Z} \epsilon_k 
\sum_{A\in\mathcal{A}_k} f_A(x)\bigg|_X^p d\mathbf{P}(\epsilon)d\mu(x)\notag.
\end{align}
\end{lause}

\begin{proof}
Observe that \eqref{lefts} can be written as
\begin{align*}
\iint_{E\times \Omega} \bigg|\int_F \sum_{k\in\Z} \epsilon_k\sum_{A\in\mathcal{A}_k} 1_A(x) k_A(x,y_A) f_A(y_A) d\nu(y)\bigg|_X^p d\mathbf{P}(\epsilon)\,d\mu(x).
\end{align*}
By H\"older's inequality and Fubini's theorem, this quantity is bounded by
\[
\iiint_{F\times E\times \Omega} \bigg|\sum_{k\in\Z} \epsilon_k\sum_{A\in\mathcal{A}_k} 1_A(x) k_A(x,y_A) f_A(y_A) \bigg|_X^p d\mathbf{P}(\epsilon)\,d\mu(x)\,d\nu(y).
\]
By the $\mathcal{R}$-boundedness assumption \eqref{rbound} and
the fact that each $\mathcal{A}_k$, $k\in\Z$, is a countable partition of $E$,
we can further bound \eqref{lefts} by 
\begin{align*}
C\iiint_{F\times E\times \Omega} \bigg|\sum_{k\in\Z} \epsilon_k\sum_{A\in\mathcal{A}_k} 1_A(x) 
f_A(y_A) \bigg|_X^p d\mathbf{P}(\epsilon)\,d\mu(x)\,d\nu(y).
\end{align*}
The proof is finished by using Theorem \ref{tangent}.
\end{proof}

Let $Q_v$ and $R_u$, $u,v\in \{1,2,\ldots,2^N\}$,
denote the son cubes of $Q\in\mathcal{D}$ and $R\in\mathcal{D}'$
in a fixed order. Fix $u$ and $v$ as above, and
assume that the elements of a matrix
\[\{T_{RQ} \in\C\,:\,R\in\mathcal{D}',\,Q\in\mathcal{D}_{R\rm{-good}},\,\ell(Q)\le \ell(R)\}
\]
satisfy the estimate
\begin{equation}\label{ades}
\frac{|T_{RQ}|}{\mu(R_u)\mu(Q_v)}\lesssim \frac{\ell(Q)^{\alpha/2}\ell(R)^{\alpha/2}}{D(Q,R)^{d+\alpha}}.
\end{equation}
Recall that
$D(Q,R)=\ell(Q)+\dist(Q,R)+\ell(R)$.

Assume that $\{f_k\in L^{1}_{\mathrm{loc}}(\R^N,\mu;X)\}_{k\in\Z}$ and 
$\{g_k\in L^1_{\mathrm{loc}}(\R^N,\mu;X^*)\}_{k\in\Z}$ are 
such that
$E_{k-1} f_k=f_k$ and $E_{k-1} g_k=g_k$ for every $k\in\Z$.
If $Q\in\mathcal{D}_k$, we denote
$f_{Q}=1_Q f_k$
and $g_R=1_R g_k$ if $R\in\mathcal{D}'_k$.

\begin{lem}\label{tut}
Assume that $E_{k-1}f_k=f_k$ and $E_{k-1}g_k=g_k$, where
$f_k$ and $g_k$, $k\in\Z$, are as above. Assume also
that the estimate \eqref{ades} holds. Then
\begin{equation}\label{treq}
\begin{split}
&\bigg|\sum_{R\in\mathcal{D}'} \sum_{\substack{Q\in\mathcal{D}_{R\textrm{-good}} \\ \ell(Q)\le \ell(R)}} \langle g_R\rangle_{R_u}T_{RQ}\langle f_Q\rangle_{Q_v}\bigg|\\
&\qquad\qquad\qquad\lesssim \bigg\|\sum_{k=-\infty}^\infty \epsilon_k g_{k}\bigg\|_{L^{q}(\mathbf{P}\otimes\mu;X^*)} 
\bigg\|\sum_{k=-\infty}^\infty \epsilon_k f_{k}\bigg\|_{L^p(\mathbf{P}\otimes \mu;X)}.
\end{split}
\end{equation}
\end{lem}

Here and throughout the paper in what follows, we use the following convention in order not to burden the notation too much: The duality pairing between elements $\phi\in X^*$ and $\xi\in X$ is written in the simple product notation as $\phi\xi$. Thus, above, the expression $\langle g_R\rangle_{R_u}T_{RQ}\langle f_Q\rangle_{Q_v}$ is the duality action of $\langle g_R\rangle_{R_u}\in X^*$ on $T_{RQ}\langle f_Q\rangle_{Q_v}\in X$, where this latter term is in turn the product of $T_{RQ}\in\C$ and $\langle f_Q\rangle_{Q_v}\in X$. (We keep the scalar $T_{RQ}$ in the middle to anticipate the operator-kernel case in which $T_{RQ}\in\mathcal{L}(X)$, in which case $\langle g_R\rangle_{R_u}T_{RQ}\langle f_Q\rangle_{Q_v}$ is the only logical order of the `product'.)

\begin{proof}
Consider first the part of the series where the ratio
$\ell(R)/\ell(Q)$ is a fixed
number $2^n$, $n\in\N_0$, and
$2^j<D(Q,R)/\ell(R)\le 2^{j+1}$ for a momentarily fixed
$j\in\N_0$. The last double inequality will be abbreviated as
$D(Q,R)/\ell(R)\sim 2^j$. If moreover $R\in\mathcal{D}'_k$, the estimate
\eqref{ades} reads as
\begin{equation}\label{mest}
\frac{|T_{RQ}|}{\mu(R_u)\mu(Q_v)}\lesssim \frac{2^{(k-n)\alpha/2}2^{k\alpha/2}}{2^{(k+j)(d+\alpha)}}=2^{-n\alpha/2}2^{-j\alpha}2^{-(k+j)d}.
\end{equation}
First of all, we have
\begin{equation}\label{estnorm}
\begin{split}
&\bigg|\sum_{k\in\Z}\sum_{R\in\mathcal{D}'_{k}} \sum_{\substack{Q\in\mathcal{D}^{R\textrm{-good}}_{k-n} \\ D(Q,R)/\ell(R)\sim 2^j}} \langle g_R\rangle_{R_u}T_{RQ}\langle f_Q\rangle_{Q_v}\bigg|\\
&=\bigg|\sum_{k\in\Z}\sum_{Q\in\mathcal{D}_{k-n}} \sum_{\substack{R\in\mathcal{D}'_{k} \\
 Q\text{ is }R\text{-good}\\
 D(Q,R)/\ell(R)\sim 2^j}} \langle g_R\rangle_{R_u}T_{RQ}\langle f_Q\rangle_{Q_v}\bigg|\\
&=\bigg|\iint_{\Omega\times\R^N} \sum_{\ell\in\Z}\sum_{S\in\mathcal{D}_{\ell-n}}
\epsilon_S  f_S(x)\\
&\qquad\cdot\sum_{k\in\Z}\sum_{Q\in\mathcal{D}_{k-n}}
\epsilon_Q \frac{1_{Q_v}(x)}{\mu(Q_v)}\sum_{\substack{R\in\mathcal{D}'_{k} 
\\Q\textrm{ is }R\textrm{-good}
\\ D(Q,R)/\ell(R)\sim 2^j}}T_{RQ}\langle g_R\rangle_{R_u}d\mathbf{P}(\epsilon)d\mu(x)\bigg|\\
&\le \bigg\| \sum_{S\in\mathcal{D}} \epsilon_Sf_{S}\bigg\|_{L^{p}(\mathbf{P}\otimes\mu;X)}
\bigg\|\sum_{k\in\Z}\epsilon_k\sum_{Q\in\mathcal{D}_{k-n}}
\sum_{\substack{R\in\mathcal{D}'_{k} 
\\Q\textrm{ is }R\textrm{-good}
\\ D(Q,R)/\ell(R)\sim 2^j}}
1_{Q_v}\frac{T_{RQ}}{\mu(Q_v)}\langle g_R\rangle_{R_u}\bigg\|_{L^q(\mathbf{P}\otimes \mu;X^*)}.
\end{split}
\end{equation}
Reorganizing the summation, we have
\[
\bigg\| \sum_{S\in\mathcal{D}} \epsilon_S f_{S}\bigg\|_{L^{p}(\mathbf{P}\otimes\mu;X)}= \bigg\|\sum_{k=-\infty}^\infty \epsilon_k f_{k}\bigg\|_{L^p(\mathbf{P}\otimes\mu;X)}
\]
and we are left with estimating the quantity
\begin{equation}\label{vtama}
\begin{split}
\bigg\|\sum_{k\in\Z}\epsilon_k\sum_{Q\in\mathcal{D}_{k-n}}
\sum_{\substack{R\in\mathcal{D}'_{k} 
\\Q\textrm{ is }R\textrm{-good}
\\ D(Q,R)/\ell(R)\sim 2^j}}
1_{Q_v}\frac{T_{RQ}}{\mu(Q_v)}\langle g_R\rangle_{R_u}\bigg\|_{L^q(\mathbf{P}\otimes \mu;X^*)}.\end{split}
\end{equation}
In order to estimate this quantity, we
first prove that
if $Q\in\mathcal{D}_{k-n}$ and $R\in\mathcal{D}'_k$ 
are such that $Q$ is $R$-good and
$D(Q,R)/\ell(R)\sim 2^j$, then
\begin{equation}\label{ksis}
Q\subset R^{(j+\theta(j+n))}\in\mathcal{D}'_{k+j+\theta(j+n)},\qquad \text{ where }\theta(j)=\Big\lceil\frac{\gamma j+r}{1-\gamma}\Big\rceil.
\end{equation}

For this purpose we first show the following: for $t\in \N_0$, we have
\begin{equation}\label{set}
r\le n+t\Rightarrow \text{either }ÊQ\subset R^{(t)}\text{ or }Q\subset \R^N\setminus R^{(t)}.
\end{equation}
The condition on the left hand side is equivalent
to that $\ell(Q)\le 2^{-r}\ell(R^{(t)})$.
In order to prove \eqref{set} we assume the opposite. Then there
are points $x\in Q\cap R^{(t)}$ and $y\in Q\cap (\R^N\setminus R^{(t)})$.
In particular, we find that $Q\cap \partial R^{(t)}\not=\emptyset$, so
that $\dist(Q,\partial R^{(t)})=0$. However, 
by using Remark \ref{new_hyvmeas} and that $Q$ is $R$-good, we 
have
\[
\dist(Q,\partial R^{(t)})>\ell(Q)^\gamma \ell(R^{(t)})^{1-\gamma}>0
\]
because $\ell(R^{(t)})\ge \ell(R)>2^{-1}\ell(R)$ and $\ell(R^{(t)})\ge 2^r\ell(Q)$. This leads to a contradiction,
and \eqref{set} follows. 

We are now ready to prove \eqref{ksis}. Assume the opposite,
that is, 
\begin{equation}\label{opp}
Q\not\subset R^{(j+\theta(j+n))}.
\end{equation}
Note that $t:=j+\theta(j+n)\ge r$, so that $r\le n+t$. Hence
\eqref{set} applies, and it implies that $Q\subset \R^N\setminus R^{(t)}$.
Using this relation and Remark \ref{new_hyvmeas}, we find that
\begin{align*}
2^{\gamma(k-n)}2^{(1-\gamma)(k+t)}&=\ell(Q)^\gamma\ell(R^{(t)})^{1-\gamma}<\dist(Q,\partial R^{(t)}) 
\\&= \dist(Q,R^{(t)})\le \dist(Q,R)\le D(Q,R)\le 2^{j+1}\ell(R)=2^{j+1+k}.
\end{align*}
Simplifying this inequality leads to the estimate
\[
t\le \frac{1}{1-\gamma}(j+1+\gamma n)=j + \frac{\gamma(j+n)+1}{1-\gamma}
<j+\theta(j+n),
\]
which is a contradiction by definition of $t$. It follows
that \eqref{opp} fails, so that \eqref{ksis} holds true as desired.
 
Using \eqref{ksis} we can now reorganize the summation over $Q$ in
\eqref{vtama} so that
\[
\sum_{Q\in\mathcal{D}_{k-n}}
\sum_{\substack{R\in\mathcal{D}'_{k} 
\\Q\textrm{ is }R\textrm{-good}
\\ D(Q,R)/\ell(R)\sim 2^j}}=\sum_{S\in\mathcal{D}'_{k+j+\theta(j+n)}}
\sum_{\substack{Q\in\mathcal{D}_{k-n}\\ Q\subset S}}
\sum_{\substack{R\in\mathcal{D}'_{k} 
\\Q\textrm{ is }R\textrm{-good}
\\ D(Q,R)/\ell(R)\sim 2^j}}
\]
For $S,Q,R$ as in the last summation, we denote
\begin{align*}
\frac{T_{RQ}}{\mu(R_u)\mu(Q_v)}=: 2^{-n\alpha/2}2^{-j\alpha}\frac{t_{RQ}}{2^{(k+j)d}}=:2^{-(n+j)\alpha/4}\frac{\tilde t_{RQ}}{\mu(S)},
\end{align*}
 where $|\tilde t_{RQ}|\lesssim |t_{RQ}|\lesssim 1$ by \eqref{mest} and the following estimates:
 $2^{-j\alpha/2}\le 1$ and
\[\mu(S)\le  2^{d(k+j+\theta(j+n))}\le 2^{d(1+r/(1-\gamma))}2^{d(k+j)+(j+n)\alpha/4}.
\]
In the first inequality above we used \eqref{mesas}, and in the last inequality we used the assumption $d\gamma/(1-\gamma)\le \alpha/4$,
see \eqref{gdef}.

For each $S\in\mathcal{D}'_{k+j+\theta(j+n)}$, define a kernel
\[
K_S(x,y) :=\sum_{\substack{Q\in\mathcal{D}_{k-n}\\ Q\subset S}}
\sum_{\substack{R\in\mathcal{D}'_{k} 
\\Q\textrm{ is }R\textrm{-good}
\\ D(Q,R)/\ell(R)\sim 2^j}}
1_{Q_v}(x)\tilde t_{RQ}1_{R_u}(y).
\]
Then $K_S$ is supported on $S\times S$ 
(notice that $R^{(j+\theta(j+n))}=S$
because both of these cubes from $\mathcal{D}'_{k+j+\theta(j+n)}$ 
contain $Q$) and $|K_S(x,y)|\lesssim 1$ since
there is at most one non zero term in the double sum for any given pair of points $(x,y)$.

The quantity inside the norm in \eqref{vtama}
is $2^{-(n+j)\alpha/4}$ times
\begin{equation}\label{tarbio}
\sum_{k_0=0}^{j+\theta(j+n)}\sum_{\substack{k\in\Z;k\equiv k_0 \\ \mod j+\theta(j+n)+1}}
\epsilon_k\sum_{S\in\mathcal{D}'_{k+j+\theta(j+n)}} \frac{1_S(x)}{\mu(S)}\int_S
K_S(x,y)1_S(y) g_{k}(y)d\mu(y),
\end{equation}
where the fact that $1_{R_u} g_{k}=1_{R_u}g_{R}$ for $R\in\mathcal{D}'_{k}$
was also used.
For a fixed $k_0$,  the series over
$k\equiv k_0\mod j+\theta(j+n)+1$ is of the form considered
in Theorem \ref{trick}. Indeed, $1_Sg_{k}$ is supported on $S\in\mathcal{D}'_{k+j+\theta(j+n)}$,
and it is constant on cubes $Q'\in\mathcal{D}'_{k-1}=\mathcal{D}_{k'+j+\theta(j+n)}$,
where $k'=k-(j+\theta(j+n)+1)$. By
Theorem \ref{trick} and the contraction principle,
the ${L^q(\mathbf{P}\otimes \mu;X^*)}$-norm of
the quantity \eqref{tarbio}, for a fixed $k_0$, is dominated  by a constant multiple of
\begin{align*}
\bigg\|\sum_{k\equiv k_0} \epsilon_k
\sum_{S\in\mathcal{D}'_{k+j+\theta(j+n)}} 1_S g_{k}\bigg\|_{L^q(\mathbf{P}\otimes \mu;X^*)}
\lesssim \bigg\|\sum_{k\in\Z} \epsilon_k
g_{k}\bigg\|_{L^q(\mathbf{P}\otimes \mu;X^*)}.
\end{align*}

The full series over $k\in\Z$ consists of $j+\theta(j+n)\lesssim n+j+1$ subseries
like this, which implies that the quantity in \eqref{vtama} is dominated by
\begin{align*}
C2^{-(n+j)\alpha/4}(n+j+1)\bigg\|\sum_{k\in\Z} \epsilon_k
g_{k}\bigg\|_{L^q(\mathbf{P}\otimes \mu;X^*)}
\end{align*}
Since this is summable over $n\in \N_0$ and $j\in\N$, this proves the goal \eqref{treq}.
\end{proof}

\section{Separated cubes}\label{separated}

This section begins the case by case analysis of different subseries of the series \eqref{firstes} to be estimated. We start by dealing with cubes well separated from each other, and more precisely we prove the following proposition.

\begin{prop}\label{ktest}
Under the assumptions of Theorem \ref{mainth}, we have
\begin{equation}\label{ofint}
\bigg|\sum_{R\in\mathcal{D}'} \sum_{\substack{Q\in\mathcal{D}_{R\text{-good}} \\ \ell(Q)\le \ell(R)\wedge \dist(Q,R)}} \langle D_R^{a,2} g,T(D_Q^{a,1} f)\rangle\bigg|
\lesssim \|g\|_q\|f\|_p
\end{equation}
for every $f\in L^p(\R^N,\mu;X)$ and $g\in L^q(\R^N,\mu;X^*)$. Here $1/p+1/q=1$.
\end{prop}


For the following lemma,
we denote $\langle g_R\rangle_{R_j}=\langle 1_R g_k\rangle_{R_j}=\langle g_k\rangle_{R_j}$ if
$R\in\mathcal{D}'_k$. Recall that the auxiliary functions below are defined
in \eqref{phideff} and \eqref{localde}.

\begin{lem}\label{iavut}
The left hand side of \eqref{ofint} is bounded (up to a constant) 
by sum of four terms of the following form:
\begin{equation}\label{comp}
\sum_{i,j=1}^{2^N}\bigg|\sum_{R\in\mathcal{D}'} \sum_{\substack{Q\in\mathcal{D}_{R\text{-good}} \\ \ell(Q)\le \ell(R)\wedge \dist(Q,R)}} \langle g_R\rangle_{R_j}\langle \psi_{R,j},T\phi_{Q,i}\rangle\langle f_Q\rangle_{Q_i}\bigg|,
\end{equation}
where, for fixed $i$ and $j$, 
\[
(g_k,\psi_{R,j})= \begin{cases} (E_{k-1}D_k^{a,2} g, \phi_{R,j}^{a,2})\qquad &\forall k\in\Z\,\wedge R\in\mathcal{D}'_k\text{ or }\\
(1_{\{b_k^{a,2}\not=b_{k-1}^{a,2}\}}E_k g,\omega_{R,j}^{a,2})\qquad &\forall k\in\Z\,\wedge R\in\mathcal{D}'_k
\end{cases}
\]
and
\[
(f_k,\phi_{Q,i})=
\begin{cases} (E_{k-1}D_k^{a,1} f, \phi_{Q,i}^{a,1})\qquad &\forall k\in\Z\,\wedge Q\in\mathcal{D}_k\text{ or}\\
(1_{\{b_k^{a,1}\not=b_{k-1}^{a,1}\}}E_k f,\omega_{Q,i}^{a,1})\qquad &\forall k\in\Z\,\wedge Q\in\mathcal{D}_k.
\end{cases}
\]
In every case, these satisfy $E_{k-1} g_k=g_k$ and
$E_{k-1} f_k=f_k$.
\end{lem}

\begin{proof}
Using Lemma \ref{localdq} and Lemma \ref{phiprop}, we get
\begin{equation}\label{idec}
\begin{split}
&D_Q^{a,1} f= (D_Q^{a,1})^2 f -\omega_Q^{a,1} E_Q f=\sum_{i=1}^{2^N} \big(\langle D_Q^{a,1} f\rangle_{Q_i}\phi_{Q,i}^{a,1}-\omega_{Q,i}^{a,1} E_Q f\big),\\
&D_R^{a,2}g= (D_R^{a,2})^2 g -\omega_R^{a,2} E_R g=\sum_{j=1}^{2^N} \big(\langle D_R^{a,2} g\rangle_{R_j}\phi_{R,j}^{a,2}-\omega_{R,j}^{a,2} E_R g\big).
\end{split}
\end{equation}
Because $\mathrm{supp}(\omega_{Q,i}^{a,1})\subset Q_i\subsetneq Q$, we have
\[\omega_{Q,i}^{a,1} E_Q f=\omega_{Q,i}^{a,1}1_Q\langle f\rangle_Q=\omega_{Q,i}^{a,1}\langle f\rangle_Q.\] Similarly
$\omega_{R,j}^{a,2} E_R g=\omega_{R,j}^{a,2}\langle g\rangle_R$. 
As a consequence, 
for every $Q\in\mathcal{D}$ and $R\in\mathcal{D}'$, we can write
 $\langle D_R^{a,2} g,T(D_Q^{a,1} f)\rangle$ as
\begin{equation}\label{repr}
\begin{split}
\sum_{i,j=1}^{2^N} \bigg\{ \langle D_R^{a,2} g\rangle_{R_j}\langle &\phi_{R,j}^{a,2},T\phi_{Q,i}^{a,1}\rangle
\langle D_Q^{a,1}f\rangle_{Q_i}
 -\langle g\rangle_R\langle \omega_{R,j}^{a,2},T\phi_{Q,i}^{a,1}\rangle\langle D_Q^{a,1} f\rangle_{Q_i}
\\&- \langle D_R^{a,2} g\rangle_{R_j}\langle \phi_{R,j}^{a,2},T\omega_{Q,i}^{a,1}\rangle\langle f\rangle_Q
+\langle g\rangle_R\langle \omega_{R,j}^{a,2},T\omega_{Q,i}^{a,1}\rangle\langle f\rangle_Q\bigg\}.
\end{split}
\end{equation}
Thus, to conclude the proof, it suffices to consider the following computations and
their
symmetric counterparts for the function $f$. 

First,
if $(g_k,\psi_{R,j})= (E_{k-1}D_k^{a,2} g, \phi_{R,j}^{a,2})$, we have
\begin{align*}
\langle g_R\rangle_{R_j}=\langle 1_R g_k\rangle_{R_j}=\langle E_{k-1} (1_R D_k^{a,2} g)\rangle_{R_j}=\langle D_R^{a,2} g\rangle_{R_j}.
\end{align*}
Hence,
$\langle g_R\rangle_{R_j} \psi_{R,j} 
=\langle D_R^{a,2} g\rangle_{R_j} \phi_{R,j}^{a,2}$.

Next we assume that $(g_k,\psi_{R,j})= (1_{\{b_k^{a,2}\not=b_{k-1}^{a,2}\}}E_k g,\omega_{R,j}^{a,2})$.
Now
\begin{equation}\label{eids}
\langle g_R\rangle_{R_j}=  \langle g_k\rangle_{R_j}= \langle1_{\{b_k^{a,2}\not=b_{k-1}^{a,2}\}}\rangle_{R_j}\langle E_k g\rangle_{R_j} =  \langle1_{\{b_k^{a,2}\not=b_{k-1}^{a,2}\}}\rangle_{R_j}\langle g\rangle_R
\end{equation}
but also
\begin{equation}\label{tid}
\langle1_{\{b_k^{a,2}\not=b_{k-1}^{a,2}\}}\rangle_{R_j}\omega_{R,j}^{a,2}
=1_{R_j}(1_{\{b_k^{a,2}\not=b_{k-1}^{a,2}\}}\omega_k^{a,2})=1_{R_j} \omega_k^{a,2}=\omega_{R,j}^{a,2}.
\end{equation}
Combining the identities \eqref{eids} and \eqref{tid} above, we get
$\langle g_R\rangle_{R_j} \psi_{R,j}
 =\langle g\rangle_R \omega_{R,j}^{a,2}$.
 \end{proof}

To proceed further we need two lemmata.

\begin{lem}\label{tokaa}
Let $Q\in\mathcal{D}$, $R\in\mathcal{D}'$ satisfy $\ell(Q)\le \ell(R)\wedge \dist(Q,R)$.
Assume that $\phi_Q,\psi_R\in L^1(\R^N,\mu;\C)$ are such that 
$\mathrm{supp}(\phi_Q)\subset Q$, $\mathrm{supp}(\psi_R)\subset R$, and
\[
\int \phi_Q\,d\mu=0.
\]
Then 
\[
|\langle \psi_R,T\phi_Q\rangle|\le \frac{\ell(Q)^\alpha}{\dist(Q,R)^{d+\alpha}}\|\phi_Q\|_{L^1(\mu)}\|\psi_R\|_{L^1(\mu)}.
\]
\end{lem}

\begin{proof}
See \cite[Lemma 4.1]{NTV}.
%
\end{proof}

\begin{lem}\label{ekaa}
Suppose that $R\in\mathcal{D}'$ and $Q\in\mathcal{D}_{R\textrm{-good}}\cup \mathcal{D}_{R^{(1)}\textrm{-good}}$
are cubes such that $\ell(Q)\le \ell(R)\wedge \dist(Q,R)$. Then 
\begin{equation}\label{tget}
\frac{\ell(Q)^\alpha}{\dist(Q,R)^{d+\alpha}}\lesssim\frac{\ell(Q)^{\alpha/2}\ell(R)^{\alpha/2}}{D(Q,R)^{d+\alpha}}.
\end{equation}
\end{lem}

\begin{proof}
See \cite[Lemma 4.2]{NTV}.
%
%
%
\end{proof} 

We are ready for the proof of Proposition \ref{ktest}.

\begin{proof}[Proof of Proposition \ref{ktest}]
By Lemma \ref{iavut}, it suffices to estimate
\eqref{comp}. To this end, we fix $i,j\in \{1,2,\ldots,2^N\}$ and denote $(\psi_R,\phi_Q)=(\psi_{R,j},\phi_{Q,i})$. Combining lemmata
\ref{tokaa} and \ref{ekaa} and using the properties of functions
 $\psi_R$ and $\phi_Q$ that are described
 in lemmata \ref{phiprop} and \ref{localdq}, we have
\begin{align*}
|\langle \psi_R,T\phi_Q\rangle|
&\lesssim \frac{\ell(Q)^{\alpha/2}\ell(R)^{\alpha/2}}{D(Q,R)^{d+\alpha}}\mu(R_j)\mu(Q_i)
\end{align*}
if  $R\in\mathcal{D}'$ and $Q\in\mathcal{D}_{R\text{-good}}$ 
satisfy $\ell(Q)\le \ell(R)\wedge \dist(Q,R)$.
Invoking Lemma \ref{tut} 
with a matrix whose elements are defined by
\begin{equation}\label{matrix_scalar}
T_{RQ}=\langle \psi_R,T\phi_Q\rangle1_{\ell(Q)\le \ell(R)\wedge \dist(Q,R)}1_{Q\in\mathcal{D}_{R\text{-good}}},
\end{equation}
we see that the quantity \eqref{comp}  can be dominated by a constant multiple of 
\begin{equation}\label{nmult}
\begin{split}\bigg\|\sum_{k=-\infty}^\infty \epsilon_k g_{k}\bigg\|_{L^{q}(\mathbf{P}\otimes\mu;X^*)}
\bigg\|\sum_{k=-\infty}^\infty \epsilon_k f_{k}\bigg\|_{L^p(\mathbf{P}\otimes \mu;X)}.
\end{split}
\end{equation}
To estimate these quantities, consider first
the case $f_k=1_{\{b_k^{a,1}\not=b_{k-1}^{a,1}\}}E_k f$. In this case, we have
\[
f_k = -1_{\{b_k^{a,1}\not=b_{k-1}^{a,1}\}}(E_{k-1}f-E_k f)+ 1_{\{b_k^{a,1}\not=b_{k-1}^{a,1}\}}E_{k-1}f.
\]
Using the contraction principle, UMD-property of $X$, and  Lemma \ref{carut} we get the
estimate
\begin{equation}\label{fgest}
\begin{split}
&\bigg\|\sum_{k=-\infty}^\infty \epsilon_k f_{k}\bigg\|_{L^p(\mathbf{P}\otimes \mu;X)}
\le \bigg\|\sum_{k=-\infty}^\infty \epsilon_k D_k f\bigg\|_{L^p(\mathbf{P}\otimes \mu;X)}
\\&\qquad\qquad\qquad+ \bigg\|\sum_{k=-\infty}^\infty \epsilon_k 1_{\{b_k^{a,1}\not=b_{k-1}^{a,1}\}}E_{k-1}f\bigg\|_{L^p(\mathbf{P}\otimes \mu;X)}\lesssim \|f\|_p.
\end{split}
\end{equation}
Next consider the case 
$f_k=E_{k-1}D_k^{a,1} f$.
Invoking Stein's inequality and then using Theorem \ref{nests}, we obtain the estimate
\[
\bigg\|\sum_{k=-\infty}^\infty \epsilon_k f_{k}\bigg\|_{L^p(\mathbf{P}\otimes \mu;X)}
 \lesssim \|f\|_p.
\]
in this case. Combining these estimates with analogous estimates for $g$,
we obtain the upper bound $C\|g\|_q\|f\|_p$
for \eqref{nmult}, and therefore also for
 \eqref{comp}. 
%
\end{proof}

\section{Preparations for deeply contained cubes}\label{prepare}

In the analysis of \eqref{firstes}, we move on from the separated cubes to ones contained inside another one. To streamline the actual analysis, we start with some preparations. We will be summing over cubes of the following type:


\begin{lem}\label{rcub}
Let $R\in\mathcal{D}'$ and $Q\in\mathcal{D}_{R\textrm{-good}}$ be such that $Q\subset R$
and $\ell(Q)<2^{-r}\ell(R)$. Then $Q\subset R_1$ for some child, denoted by $R_1$, of $R$.
\end{lem}

\begin{proof}
Denote by
$R_1$ any child of $R$ for which $R_1\cap Q\not=\emptyset$.
It suffices to show that $Q\subset R_1$. Note that
$\ell(R_1)=2^{-1}\ell(R)$ and $\ell(R_1)\ge 2^r \ell(Q)$. Because $Q$ is $R$-good, we can invoke
Remark \ref{new_hyvmeas} in order to see that
\[
\dist(Q,\partial R_1)>\ell(Q)^\gamma\ell(R_1)^{1-\gamma}>0.
\]
Because $Q\cap R_1\not=\emptyset$, it follows that $Q\subset R_1$.
\end{proof}

Let $Q$ and $R$ be such cubes that are considered in Lemma \ref{rcub}.
The children of $R$ are denoted by $R_1,\ldots,R_{2^N}$. However,
we choose
the indexing such that
$Q\subset R_1$, see Lemma \ref{rcub}.
The indexing of children depends on  $Q$; in particular,
$R_1=R_1(Q)$ depends on $Q$. We will not indicate this dependence explicitly.

The children of $Q$ are denoted by $Q_1,\ldots,Q_{2^N}$ in some order.

Let
$u,v\in \{1,2,\ldots,2^N\}$ be fixed.
Here we consider a matrix
$\{T^u_{RQ}\}$
satisfying the estimate
\begin{equation}\label{ades2}
\begin{split}
&\frac{|T^u_{RQ}|}{\mu(R_u)\mu(Q_v)}\lesssim 
\bigg(\frac{\ell(Q)}{\ell(R)}\bigg)^{\alpha/2}\cdot \begin{cases}\mu(R)^{-1}\quad &\textrm{ if }u\not=1,\\
\mu(R_1)^{-1}\quad &\textrm{ if }u=1,
\end{cases}
\end{split}
\end{equation}
 
\begin{lem}\label{tut2}
Let $\{f_k\in L^{1}_{\mathrm{loc}}(\R^N,\mu;X)\}_{k\in\Z}$ and 
$\{g_k\in L^1_{\mathrm{loc}}(\R^N,\mu;X^*)\}_{k\in\Z}$ be 
 such that
$E_{k-1} f_k=f_k$ and $E_{k-1} g_k=g_k$ for every $k\in\Z$. 
Then, under the assumption \eqref{ades2}, we have
\begin{equation}\label{treq2}
\begin{split}
&\bigg|\sum_{R\in\mathcal{D}'} \sum_{\substack{Q\in\mathcal{D}_{R\textrm{-good}} \\ Q\subset R_1 \\ \ell(Q)<2^{-r}\ell(R)}} \langle g_R\rangle_{R_u}T_{RQ}^u\langle f_Q\rangle_{Q_v}\bigg|\\
&\qquad\qquad\qquad\lesssim \bigg\|\sum_{k=-\infty}^\infty \epsilon_k g_{k}\bigg\|_{L^{q}(\mathbf{P}\otimes\mu;X^*)} \cdot
\bigg\|\sum_{k=-\infty}^\infty \epsilon_k f_{k}\bigg\|_{L^p(\mathbf{P}\otimes \mu;X)}.
\end{split}
\end{equation}
\end{lem}

\begin{proof}
Consider first part of the series where the ratio
$\ell(Q)/\ell(R)$ is a fixed
number $2^{-n}$ with $n\in \{r+1,r+2,\ldots\}$.
If  $R\in\mathcal{D}'_k$, the estimate
\eqref{ades2} reads as
\begin{equation}\label{mest2}
\frac{|T^u_{RQ}|}{\mu(R_u)\mu(Q_v)}\lesssim 2^{-n\alpha/2}\cdot\begin{cases}\mu(R)^{-1}\quad &\textrm{ if }u\not=1,\\
\mu(R_1)^{-1}\quad &\textrm{ if }u=1.
\end{cases}
\end{equation}

Adapting \eqref{estnorm} to the present situation
yields the estimate
\begin{align*}
&\bigg|\sum_{k\in\Z}\sum_{R\in\mathcal{D}'_{k}} 
\sum_{\substack{Q\in\mathcal{D}_{k-n} \\ Q\textrm{ is }R\textrm{-good}Ê\\ÊQ\subset R_1}} \langle g_R\rangle_{R_u}T_{RQ}^u\langle f_Q\rangle_{Q_v}\bigg|\\
&\le \bigg\| \sum_{S\in\mathcal{D}} \epsilon_Sf_{S}\bigg\|_{L^{p}(\mathbf{P}\otimes\mu;X)}
 \bigg\|\sum_{k\in\Z}\epsilon_k
\sum_{Q\in\mathcal{D}_{k-n}}
\sum_{\substack{R\in\mathcal{D}'_{k} \\  Q\textrm{ is }R\textrm{-good}\\ Q\subset R_1}}
1_{Q_v}\frac{T_{RQ}^u}{\mu(Q_v)}\langle g_R\rangle_{R_u}\bigg\|_{L^q(\mathbf{P}\otimes \mu;X^*)}.
\end{align*}
Reorganizing the summation, we have
\[
\bigg\| \sum_{S\in\mathcal{D}} \epsilon_S f_{S}\bigg\|_{L^{p}(\mathbf{P}\otimes\mu;X)}= \bigg\|\sum_{k=-\infty}^\infty \epsilon_k f_{k}\bigg\|_{L^p(\mathbf{P}\otimes\mu;X)}
\]
so that we are left with estimating the quantity
\begin{equation}\label{vtama2}
\begin{split}
\bigg\|\sum_{k\in\Z}\epsilon_k\sum_{R\in\mathcal{D}'_{k}}
\sum_{\substack{Q\in\mathcal{D}_{k-n} \\ Q\textrm{ is }R\textrm{-good}\\Q\subset R_1}}
1_{Q_v}\frac{T_{RQ}^u}{\mu(Q_v)}\langle g_R&\rangle_{R_u}\bigg\|_{L^q(\mathbf{P}\otimes \mu;X^*)}.
\end{split}
\end{equation}
For each $R\in\mathcal{D}'_{k}$ and $m\in \{2,\ldots,2^N\}$, define kernel
\begin{align*}
K_R^{m}(x,y) &:=  
2^{n\alpha/2}\sum_{\substack{Q\in\mathcal{D}_{k-n} \\ Q\textrm{ is }R\textrm{-good}\\Q\subset R_1}}
\mu(R)\underbrace{ 1_{R}(x)1_{Q_v}(x)}_{=1_{Q_v}(x)}\frac{T_{RQ}^m}{\mu(Q_v)\mu(R_m)}1_{R_m}(y).
\end{align*}
For  $S\in\mathcal{D}'_{k-1}$, we define
\begin{align*}
K^{1}_{S}(x,y)&:=2^{n\alpha/2}\sum_{\substack{Q\in\mathcal{D}_{k-n} \\ Q\textrm{ is }S^{(1)}\textrm{-good}\\Q\subset S}}
\mu(S)\underbrace{1_{S}(x)1_{Q_v}(x)}_{=1_{Q_v}(x)}\frac{T_{S^{(1)}Q}^1}{\mu(Q_v)\mu(S)}1_{S}(y).
\end{align*}
We have
\[
|K^1_{S}(x,y)|+\sum_{m=2}^{2^N}|K^m_{R}(x,y)|\lesssim 1
\] by
 using \eqref{mest2} and the fact that
there is at most one non zero term in the sums above for any given pair of points $(x,y)$.
In the sequel we will use one of these kernels, depending on the
value of $u$.
If $u\not=1$, then $K^u_R$ is supported on $R\times R$. If $u=1$, then
$K^u_{S}=K^1_{S}$ is supported on $S\times S$.

The quantity inside the $L^p$-norm in \eqref{vtama2}
is $2^{-n\alpha/2}\Sigma_u$, where
\begin{equation}\label{tarbio2}
\begin{split}
\Sigma_u:=&\sum_{k\in\Z}
\epsilon_k\sum_{R\in\mathcal{D}'_{k}}
\frac{1_R(x)}{\mu(R)}\int_R
K^u_R(x,y)1_R(y) g_{k}(y)d\mu(y),\quad \text{ if }u\not=1;
\end{split}
\end{equation}
and
\begin{equation}\label{tarbio3}
\begin{split}
\Sigma_u:=&\sum_{k\in\Z}\epsilon_k
\sum_{\substack{S\in\mathcal{D}'_{k-1}}} \frac{1_{S}(x)}{\mu(S)}\int_{S}
K^1_{S}(x,y)1_{S}(y) g_{k}(y)d\mu(y),\quad \text{ if }u=1.
\end{split}
\end{equation}
Here the fact that $1_{R_u} g_{k}=1_{R_u}g_{R}$ for $R\in\mathcal{D}'_{k}$
was also used.

Then we do a case study;
assume first that $u\not=1$.
Then $1_R g_k$ is supported on $R\in\mathcal{D}'_k$, and
it is constant on cubes
$R'\in\mathcal{D}_{k-1}$. 
The tangent martingale trick (see Theorem \ref{trick})  implies that the
${L^q(\mathbf{P}\otimes \mu;X^*)}$-norm of
the quantity \eqref{tarbio2} is dominated  by a constant multiple of
\begin{equation}\label{fas1}
\begin{split}
\bigg\|\sum_{k\in\Z} \epsilon_k
\sum_{R\in\mathcal{D}'_{k}} g_R\bigg\|_{L^q(\mathbf{P}\otimes \mu;X^*)}= \bigg\|\sum_{k\in\Z} \epsilon_k
g_{k}\bigg\|_{L^q(\mathbf{P}\otimes \mu;X^*)}.
\end{split}
\end{equation}

Then we assume that $u=1$. In this case
$1_{S}g_k$ is supported on $S\in\mathcal{D}'_{k-1}$, and
it is constant on cubes $R'\in\mathcal{D}'_{k-2}$.
The tangent martingale trick (Theorem \ref{trick}) implies that 
${L^q(\mathbf{P}\otimes \mu;X^*)}$-norm of
the quantity \eqref{tarbio3} is dominated  by a constant multiple of
\begin{equation}\label{fas2}
\bigg\|\sum_{k\in\Z} \epsilon_k
\sum_{S\in\mathcal{D}'_{k-1}} 1_S g_k\bigg\|_{L^q(\mathbf{P}\otimes \mu;X^*)}
\lesssim \bigg\|\sum_{k\in\Z} \epsilon_k
g_{k}\bigg\|_{L^q(\mathbf{P}\otimes \mu;X^*)}.
\end{equation}


Combining
the estimates \eqref{fas1} and \eqref{fas2}, we find that
the quantity in \eqref{vtama2} is dominated by
\begin{align*}
C2^{-n\alpha/2}\bigg\|\sum_{k\in\Z} \epsilon_k
g_{k}\bigg\|_{L^q(\mathbf{P}\otimes \mu;X^*)}
\end{align*}
This is summable over $n\in \{r+1,r+2,\ldots\}$, and therefore we obtain \eqref{treq2}.
\end{proof}

\section{Deeply contained cubes}\label{nested}

During the course of Section \ref{nested} and Section \ref{paraproducts}
we establish the following estimate for the part of
the summation in \eqref{firstes} involving deeply contained cubes.

\begin{prop}\label{nestedlem}
Under the assumptions of Theorem \ref{mainth}, we have
\begin{equation}\label{ofint2}
\bigg|\sum_{R\in\mathcal{D}'} \sum_{\substack{Q\in\mathcal{D}_{R\text{-good}} \\Q\subset R \\ \ell(Q)<2^{-r}\ell(R)}} \langle D_R^{a,2} g,T(D_Q^{a,1} f)\rangle\bigg|\lesssim \|f\|_p\|g\|_q
\end{equation}
for every $f\in L^p(X)$ and $g\in L^q(X^*)$. Here $1/p+1/q=1$.
\end{prop}

Let $R$ and $Q$ be as in  \eqref{ofint2}.
Recall from 
beginning of Section \ref{prepare} that $R_1,\ldots,R_{2^N}$ are children of $R$ such that
$Q\subset R_1\subsetneq R$. 
By the proof of Lemma \ref{rcub}, we get
\begin{equation}\label{tkayt}
2^{r(1-\gamma)}\ell(Q)\le \ell(Q)^\gamma \ell(R_m)^{1-\gamma}<\dist(Q,\partial R_m),\qquad m\in \{1,2,\ldots, 2^N\}.
\end{equation}
This is a useful inequality later on.



Writing $1_R=\sum_{m=1}^{2^N} 1_{R_m}$ and using that
$\textrm{supp}(D_R^{a,2}g)\subset R$ 
yields
\begin{equation}\label{tsum}
\begin{split}
\langle D&_R^{a,2} g,T(D_Q^{a,1} f)\rangle = 
 \langle 1_{R_1}D_R^{a,2} g,T(D_Q^{a,1}f)\rangle
+\sum_{m=2}^{2^N} \langle 1_{R_m} D_R^{a,2}g,T(D_Q^{a,1}f)\rangle.
\end{split}
\end{equation}
The point is that $Q$ is contained in $R_1$, so $Q$ is separated from
the children $R_2,\ldots,R_m$. Hence, arguments
developed in Section \ref{separated} can be applied to these terms. Treating the main part of the term  associated with the child $R_1$
requires so called paraproducts; these are discussed in the following section.

Let us sketch what are the estimates that are performed
in the remaining part of this section.
First we will show that
\begin{equation}\label{equa}
\sum_{m=2}^{2^N} \bigg|\sum_{R\in\mathcal{D}'} \sum_{\substack{Q\in\mathcal{D}_{R\textrm{-good}}\\Q\subset R_1 \\ \ell(Q)<2^{-r}\ell(R)}} \langle 1_{R_m} D_R^{a,2}g,T(D_Q^{a,1}f)\rangle\bigg|\lesssim \|f\|_p\|g\|_q.
\end{equation}
Then, in order to treat the remaining (first) term on the right hand side of \eqref{tsum}, we write
$R_1^c = \R^N\setminus R_1$ and
\begin{align*}
1_{R_1}D_R^{a,2}g =  
1_{R_1}\bigg(b_{R_1^a}\frac{\langle g\rangle_{R_1}}{\langle b_{R_1^a}\rangle_{R_1}}-b_{R^a}\frac{\langle g\rangle_R}{\langle b_{R^a}\rangle_R}\bigg)
&=(1-1_{R_1^c})\bigg(b_{R_1^a}\frac{\langle g\rangle_{R_1}}{\langle b_{R_1^a}\rangle_{R_1}}-b_{R^a}\frac{\langle g\rangle_R}{\langle b_{R^a}\rangle_R}\bigg).
\end{align*}
In this section we establish the 
estimate
\begin{equation}\label{lesq}
\bigg|\sum_{R\in\mathcal{D}'} \sum_{\substack{Q\in\mathcal{D}_{R\textrm{-good}}\\Q\subset R_1 \\ \ell(Q)<2^{-r}\ell(R)}}
\bigg\langle 1_{R_1^c}\bigg(b_{R_1^a}\frac{\langle g\rangle_{R_1}}{\langle b_{R_1^a}\rangle_{R_1}}-b_{R^a}\frac{\langle g\rangle_R}{\langle b_{R^a}\rangle_R}\bigg),T(D_Q^{a,1}f)\bigg\rangle\bigg|\lesssim \|f\|_p\|g\|_q.
\end{equation}
The remaining term is treated in Section \ref{paraproducts} by using paraproducts.

\subsection*{Proving estimate \eqref{equa}}
Proceeding as in the proof of Lemma \ref{iavut}, we see that the
left hand side of \eqref{equa} is dominated by a series of four terms, each of them being of the form
\begin{equation}\label{comp2}
\sum_{m=2}^{2^N}\sum_{i,j=1}^{2^N}\bigg|\sum_{R\in\mathcal{D}'} \sum_{\substack{Q\in\mathcal{D}_{R\textrm{-good}}\\Q\subset R_1 \\ \ell(Q)<2^{-r}\ell(R)}} \langle g_R\rangle_{R_j}\langle \psi_{R,j,m},T\phi_{Q,i}\rangle\langle f_Q\rangle_{Q_i}\bigg|,
\end{equation}
where we denote $\langle g_R\rangle_{R_j}=\langle 1_R g_k\rangle_{R_j}=\langle g_k\rangle_{R_j}$ if
$R\in\mathcal{D}'_k$ (similarly for $f$), and the four summands are determined
by the following possibilities:
\begin{equation}\label{gchi}
(g_k,\psi_{R,j,m})\in \{(E_{k-1}D_k^{a,2} g, 1_{R_m}\phi_{R,j}^{a,2}),
(1_{\{b_k^{a,2}\not=b_{k-1}^{a,2}\}}E_k g, 1_{R_m}\omega_{R,j}^{a,2})\}
\end{equation}
and
\begin{equation}\label{fchi}
(f_k,\phi_{Q,i})\in \{(E_{k-1}D_k^{a,1} f, \phi_{Q,i}^{a,1}),
(1_{\{b_k^{a,1}\not=b_{k-1}^{a,1}\}}E_k f,\omega_{Q,i}^{a,1})\}.
\end{equation}
Note that, in any case, $E_{k-1} g_k=g_k$ and
$E_{k-1} f_k=f_k$.

\begin{lem}\label{rest}
Assume that $R\in\mathcal{D}'$ and $Q\in\mathcal{D}_{R\textrm{-good}}$,
$Q\subset R_1$ and $\ell(Q)<2^{-r}\ell(R)$.
Let $\psi_{R,j,m}$, $m\ge 2$, and $\phi_{Q,i}$ be any of those functions
that are quantified  in \eqref{gchi} and \eqref{fchi}, respectively.
Then  $T_{RQ}^j := \langle \psi_{R,j,m},T\phi_{Q,i}\rangle$, $i,j\in \{1,\ldots,2^N\}$,  satisfies
\[
\frac{|T_{RQ}^j|}{\mu(R_j)\mu(Q_i)}\lesssim \bigg(\frac{\ell(Q)}{\ell(R)}\bigg)^{\alpha/2}
\mu(R)^{-1}.
\]
\end{lem}

\begin{proof}
Because $m\ge 2$, we have
$Q\cap R_m \subset R_1\cap R_m=\emptyset$ so that
\[\ell(Q)<\dist(Q,\partial R_m)=\dist(Q,R_m)\] by \eqref{tkayt} and the assumption that
$1\le 2^{r(1-\gamma)}$, see \eqref{gdef}. We also
have $\ell(Q)\le \ell(R_m)$. Hence, by using lemmata
\ref{tokaa} and \ref{ekaa}, the properties of
functions $\psi_{R,j,m}$ and $\phi_{Q,i}$ 
that follow from
lemmata \ref{phiprop} and \ref{ombasic},
and \eqref{mesas}, we obtain
\begin{equation}\label{sest}
\begin{split}
&|\langle \psi_{R,j,m},T\phi_{Q,i}\rangle| 
\lesssim\frac{\ell(Q)^{\alpha/2}\ell(R_m)^{\alpha/2}}{D(Q,R_m)^{d+\alpha}}\|\psi_{R,j,m}\|_{L^1(\mu)}\|\phi_{Q,i}\|_{L^1(\mu)}\\
&\lesssim \bigg(\frac{\ell(Q)}{\ell(R)}\bigg)^{\alpha/2}\frac{\|\psi_{R,j,m}\|_{L^1(\mu)}\|\phi_{Q,i}\|_{L^1(\mu)}}{\ell(R)^d}
\lesssim \bigg(\frac{\ell(Q)}{\ell(R)}\bigg)^{\alpha/2}\frac{\mu(R_j)\mu(Q_i)}{\mu(R)}.
\end{split}
\end{equation}
This is the desired estimate.
\end{proof}
Combining lemmata \ref{tut2} and \ref{rest} and then
estimating as in the end of Section \ref{separated}, we see that the quantity
\eqref{comp2} can be dominated by 
a constant multiple of $\|f\|_p\|g\|_q$. 
As a consequence, we see that the left hand side of
\eqref{equa} is be dominated by a constant multiple
of $\|f\|_p\|g\|_q.$

\subsection*{Proving estimate \eqref{lesq}}
Let $R\in\mathcal{D}'_k$. We write
\begin{align*}
&1_{R_1^c}\bigg(b_{R_1^a}\frac{\langle g\rangle_{R_1}}{\langle b_{R_1^a}\rangle_{R_1}}-b_{R^a}\frac{\langle g\rangle_R}{\langle b_{R^a}\rangle_R}\bigg)
=1_{R_1^c} b_{R^a}\langle s_k\rangle_{R_1}
+1_{R_1^c} b_{R_1^a}\langle h_k\rangle_{R_1}+1_{R_1^c} b_{R^a}\langle u_k\rangle_{R_1},
\end{align*}
where 
\begin{align*}
s_{k}=1_{\{b_{k-1}^{a,2}=b_{k}^{a,2}\}}\bigg(\frac{E_{k-1} g}{E_{k-1}b_{k-1}^{a,2}}
- \frac{E_k g}{E_k b_k^{a,2}}\bigg),
\end{align*}
and
\[
h_{k}=1_{\{b_{k-1}^{a,2}\not=b_{k}^{a,2}\}}\frac{E_{k-1} g}{E_{k-1}b_{k-1}^{a,2}},\quad u_{k}=
-1_{\{b_{k-1}^{a,2}\not=b_{k}^{a,2}\}}\frac{E_k g}{E_k b_k^{a,2}}.
\]
By \eqref{idec}, we see that the left hand side of \eqref{lesq}
can be dominated from above by a sum of six terms, each  of them being of
the form
\begin{equation}\label{eeq}
\sum_{i=1}^{2^N} \bigg|\sum_{R\in\mathcal{D}'} \sum_{\substack{Q\in\mathcal{D}_{\mathrm{good}}\\Q\subset R_1 \\ \ell(Q)<2^{-r}\ell(R)}}
\langle g_{R}\rangle_{R_1}
\langle \psi_R,T\phi_{Q,i}\rangle\langle f_Q\rangle_{Q_i}\bigg|,
\end{equation}
where  $\langle g_R\rangle_{R_j}=\langle 1_R g_k\rangle_{R_j}=\langle g_k\rangle_{R_j}$ if
$R\in\mathcal{D}'_k$ (similarly for $f$), and the six terms
are determined by the following choices:
\begin{equation}\label{gco}
(g_k,\psi_R)\in \{(s_k,1_{R_1^c} b_{R^a}), (h_k,1_{R_1^c} b_{R_1^a}),
(u_k,1_{R_1^c} b_{R^a})\}
\end{equation}
and
\begin{equation}\label{fco}
(f_k,\phi_{Q,i})\in \{(E_{k-1}D_k^{a,1} f, \phi_{Q,i}^{a,1}),
(1_{\{b_k^{a,1}\not=b_{k-1}^{a,1}\}}E_k f,\omega_{Q,i}^{a,1})\}.
\end{equation}
Note that, in any case, $E_{k-1} g_{k}=g_{k}$ and
$E_{k-1} f_k=f_k$.

\begin{lem}\label{szo}
Let $\psi_R$ and $\phi_{Q,i}$ be any of those functions
that are quantified  in \eqref{gco} and \eqref{fco}
for $R\in\mathcal{D}'$ and $Q\in\mathcal{D}_{R\textrm{-good}}$ satisfying
$Q\subset R_1$ and $\ell(Q)<2^{-r}\ell(R)$.
Then $T_{RQ}^1 := \langle \psi_R,T\phi_{Q,i}\rangle$ satisfies
the estimate
\[
\frac{|T_{RQ}^1|}{\mu(Q_i)}\lesssim \bigg(\frac{\ell(Q)}{\ell(R)}\bigg)^{\alpha/2}.
\]
\end{lem}

\begin{proof}
Denote
by $y_Q$ the midpoint of $Q$.
Let  $x\in R_1^c$ and $y\in Q$. By \eqref{gdef}, \eqref{tkayt} and the fact that
$Q\subset R_1$, we have
\begin{align*}
2|y-y_Q|\le 2^{r(1-\gamma)}\ell(Q)<\dist(Q,\partial R_1)=\dist(Q,R_1^c)\le |x-y_Q|.
\end{align*}
Using the kernel estimate \eqref{smooth} and
the facts  $\int \phi_{Q,i}=0$ and $\textrm{supp}(\phi_{Q,i})\subset Q$, we get
\begin{equation}\label{sums}
\begin{split}
&|\langle \psi_R,T\phi_{Q,i}\rangle|\\&=
\bigg|\int_{\R^N} \int_{\R^N} \psi_R(x)1_{R_1^c}(x)\big(K(x,y)-K(x,y_Q)\big)\phi_{Q,i}(y)d\mu(y) d\mu(x)\bigg|\\
&\lesssim \int_{R_1^c} \int_{\R^N} \frac{|y-y_Q|^{\alpha}}{|x-y_Q|^{d+\alpha}}|\phi_{Q,i}(y)|d\mu(y) d\mu(x)\lesssim \|\phi_{Q,i}\|_1 \int_{R_1^c} \frac{\ell(Q)^{\alpha}}{|x-y_Q|^{d+\alpha}}d\mu(x).
\end{split}
\end{equation}
Denoting $A_k=\{x\,:\,2^k\mathrm{dist}(R_1^c,Q)\le |x-y_Q|<2^{k+1}\mathrm{dist}(R_1^c,Q)\}$, we can estimate 
the last integral as follows
\begin{equation}\label{ssss}
\begin{split}
\int_{R_1^c} \frac{\ell(Q)^{\alpha}}{|x-y_Q|^{d+\alpha}}d\mu(x)
&\le \sum_{k=0}^\infty \int_{A_k}\frac{\ell(Q)^\alpha}{(2^k\mathrm{dist}(R_1^c,Q))^{d+\alpha}}d\mu(x)\\
&\le\sum_{k=0}^\infty \frac{\ell(Q)^\alpha\mu(B(y_Q, 2^{k+1}\mathrm{dist}(R_1^c,Q)))}{(2^k\mathrm{dist}(R_1^c,Q))^{d+\alpha}}\\
&\lesssim 
\frac{\ell(Q)^\alpha}{{\mathrm{dist}(R_1^c,Q)^{\alpha}}}\sum_{k=0}^\infty \frac{1}{2^{\alpha k}}
\lesssim\bigg(\frac{\ell(Q)}{{\mathrm{dist}(R_1^c,Q)}}\bigg)^\alpha.
\end{split}
\end{equation}
This can be further estimated by using that $\gamma\le \alpha(2(d+\alpha))^{-1}<2^{-1}$, see \eqref{gdef}.
Combining this with \eqref{tkayt} yields the estimate
\begin{align*}
\ell(Q)^{1/2}\ell(R_1)^{1/2}\le\ell(Q)^\gamma
\ell(R_1)^{(1-\gamma)}\le \dist(Q,\partial R_1)=\dist(Q,R_1^c).
\end{align*}
Substituting this into  \eqref{sums}, we find that
\[
|\langle \psi_R,T\phi_{Q,i}\rangle|\lesssim
\bigg(\frac{\ell(Q)}{{\mathrm{dist}(R_1^c,Q)}}\bigg)^\alpha\|\phi_{Q,i}\|_1
\lesssim \bigg(\frac{\ell(Q)}{\ell(R_1)}\bigg)^{\alpha/2}\|\phi_{Q,i}\|_1.
\]
This is as required because $\|\phi_{Q,i}\|_1\lesssim \mu(Q_i)$
and $\ell(R_1)=2^{-1}\ell(R)$.
\end{proof}

Combining lemmata \ref{szo} and \ref{tut2} we find that each of the six terms of the form
\eqref{eeq} are bounded (up to a constant) by
by
\[
\bigg\|\sum_{k=-\infty}^\infty \epsilon_k g_{k}\bigg\|_{L^{q}(\mathbf{P}\otimes\mu;X^*)}\cdot
\bigg\|\sum_{k=-\infty}^\infty \epsilon_k f_{k}\bigg\|_{L^p(\mathbf{P}\otimes \mu;X)}.
\]
At the end of Section \ref{separated} we verified
that the second factor above can be  dominated by
$\|f\|_p$.
Hence, it remains to verify the following estimate,
\begin{equation}\label{ges}
\bigg\|\sum_{k=-\infty}^\infty \epsilon_k g_{k}\bigg\|_{L^q(\mathbf{P}\otimes \mu;X^*)}
\lesssim \|g\|_q.
\end{equation}
The cases $g_k\in \{h_k,u_k\}$ have been cleared
in connection with the separated cubes:  \eqref{ges} follows from the contraction principle and
\eqref{fgest}
 if we recall that $|E_k b_k^{a,2}|\ge \delta^2 $ $\mu$-almost everywhere. The remaining case
\begin{align*}
g_k = s_k&=1_{\{b_{k-1}^{a,2}=b_{k}^{a,2}\}}\bigg(\frac{E_{k-1} g}{E_{k-1}b_{k-1}^{a,2}}
- \frac{E_k g}{E_k b_k^{a,2}}\bigg).
\end{align*}
is treated by Lemma \ref{carut2}.

This concludes the proof of 
estimate \eqref{lesq}.

\section{Paraproducts}\label{paraproducts}

\noindent
In order to finish the proof of Proposition \ref{nestedlem}, we still need to establish the following
estimate
\begin{equation}\label{post}
\begin{split}
\bigg|\sum_{R\in\mathcal{D}'} \sum_{\substack{Q\in\mathcal{D}_{R\text{-good}} \\Q\subset R \\ \ell(Q)<2^{-r}\ell(R)}} 
\bigg\langle b_{R_1^a}\frac{\langle g\rangle_{R_1}}{\langle b_{R_1^a}\rangle_{R_1}}-b_{R^a}\frac{\langle g\rangle_R}{\langle b_{R^a}\rangle_R},T(D_Q^{a,1} f)\bigg\rangle\bigg|\lesssim \|f\|_p\|g\|_q.
\end{split}
\end{equation}
We will draw inspiration from the work of Hyt\"onen and Martikainen \cite{hm}, and
the following standing assumptions in Theorem \ref{mainth} are crucial while proving \eqref{post}:
\begin{itemize}
\item
$X^*$ is an RMF-space;
\item
$\|T^*b^2_R\|_{L^\infty(\R^N,\mu;\C)}\le 1$ if $R$ is a cube in $\R^N$.
\end{itemize}

For $Q\in\mathcal{D}$ and $R\in\mathcal{D}'$, we denote
\begin{equation*}
\chi_{Q,R}=\begin{cases}Ê
1,\qquad &\textrm{if }Q\textrm{ is }R\text{-good},\;Q\subset R,\textrm{ and }\ell(Q)<2^{-r}\ell(R);\\
0,\qquad &\textrm{otherwise}.
\end{cases}
\end{equation*}
Suppose that $\chi_{Q,R}=1$. Then we write
\[
G_{Q,R}:=b_{R_1^a}\frac{\langle g\rangle_{R_1}}{\langle b_{R_1^a}\rangle_{R_1}}-b_{R^a}\frac{\langle g\rangle_R}{\langle b_{R^a}\rangle_R}
\]
for a quantity that 
depends on $Q$ and $R$, as $R_1$ stands for the
child of $R$ for which $Q\subset R$.
Using the notation above, we can rewrite the left hand side of \eqref{post}
as follows
\begin{equation}\label{eteen}
\bigg|\sum_{R\in\mathcal{D}'} \sum_{\substack{Q\in\mathcal{D}\\\chi_{Q,R}=1}}
\big\langle G_{Q,R},T(D_Q^{a,1} f)\big\rangle \bigg|
=\bigg| \sum_{Q\in\mathcal{D}}
\bigg\langle \sum_{\substack{R\in\mathcal{D}'\\ \chi_{Q,R}=1}} G_{Q,R},T(D_Q^{a,1} f)\bigg\rangle \bigg|.
\end{equation}
It is straightforward to verify that, if
$\chi_{Q,R}=1$, then $\chi_{Q,R^{(m)}}=1$ for every $m\in\N_0$.
It follows that, if $Q\in\mathcal{D}$ and the inner sum on the right hand side is nonempty, there
exists a unique cube $S=S(Q)\in\mathcal{D}'$ containing $Q$ such that
 $\chi_{Q,R}=1$ if, and only if, $S\subsetneq R\in \mathcal{D}'$.
 If the inner sum in question is empty, we let $S=S(Q)=\emptyset$.
As a consequence, if $S(Q)\not=\emptyset$, 
\begin{align*}
\sum_{\substack{R\in\mathcal{D}'\\ \chi_{Q,R}=1}} G_{Q,R}
=\sum_{\substack{R\in\mathcal{D}'\\ S\subsetneq R}} G_{Q,R}
=b_{S^a}\frac{\langle g\rangle_S}{\langle b_{S^a}\rangle_S} - b_{R_0}\frac{\langle g\rangle_{\R^N}}{\langle b_{R_0}\rangle_{\R^N}}.
\end{align*}
Substituting this identity to the right hand side of \eqref{eteen}, we get
\begin{equation}\label{itsa}
\begin{split}
\bigg| \sum_{Q\in\mathcal{D}}
\bigg\langle \sum_{\substack{R\in\mathcal{D}'\\ \chi_{Q,R}=1}} G_{Q,R},T(D_Q^{a,1} f)\bigg\rangle \bigg|
&\le |\langle \Pi g,f\rangle| + \bigg|\sum_{\substack{Q\in\mathcal{D}\\S(Q)\not=\emptyset}} \bigg\langle T^*b_{R_0}\frac{\langle g\rangle_{\R^N}}{\langle b_{R_0}\rangle_{\R^N}},D_Q^{a,1} f\bigg\rangle\bigg|,
\end{split}
\end{equation}
where the {\em paraproduct operator} $g\mapsto \Pi g$ is defined by
\begin{equation}\label{parop}
\Pi g := \sum_{\substack{Q\in\mathcal{D}\\S(Q)\not=\emptyset}}
\frac{\langle g\rangle_S}{\langle b_{S^a}\rangle_S} (D_Q^{a,1})^*(T^* b_{S^a})= \sum_{R\in\mathcal{D}'}\sum_{\substack{Q\in\mathcal{D}\\ S(Q)=R}}
\frac{\langle g\rangle_R}{\langle b_{R^a}\rangle_R} (D_Q^{a,1})^*(T^* b_{R^a}).
\end{equation}
Throughout the rest of this section, we will prove the following estimates:

\begin{prop}
Under the standing assumptions, the paraproduct just defined satisfies
\begin{equation}\label{estpar}
|\langle \Pi g,f\rangle|\lesssim \|f\|_p\|g\|_q,
\end{equation}
and we also have the estimate
\begin{equation}\label{estpar2}
\bigg|\sum_{\substack{Q\in\mathcal{D}\\S(Q)\not=\emptyset}} \bigg\langle T^*b_{R_0}\frac{\langle g\rangle_{\R^N}}{\langle b_{R_0}\rangle_{\R^N}},D_Q^{a,1} f\bigg\rangle\bigg|\lesssim \|f\|_p\|g\|_q.
\end{equation}
\end{prop}

Observe that these estimates imply \eqref{post} which in turn, combined with estimates
in Section \ref{nested}, implies
Proposition \ref{nestedlem}.

\subsection*{Proving estimate \eqref{estpar}}
Here we will concentrate on paraproducts, and begin with the following lemma.
\begin{lem}\label{genemb}
Suppose that $t>q\vee s$, where $X^*$ has cotype $s$.
Assume that  a sequence $\{d_j\}_{j\in\Z}$ of functions $\R^N\to L^t(\Omega;\C)$ satisfies
$d_j\in L^1(\R^N;L^t(\Omega;\C))$, then
\[
\bigg\| \sum_{j\in\Z} \epsilon_j^\star d_j E_j g\bigg\|_{L^q(\Omega^\star\times \R^N;L^t(\Omega;X^*))}
\lesssim \|\{|d_j(\cdot)|_{L^t(\Omega;\C)}\}_{j\in\Z}\|_{\mathrm{Car}^t(\mathcal{D}')}\cdot \|g\|_{L^q(\R^N;X^*)}.
\]
\end{lem}

\begin{proof}
This will be a special case of Theorem 3.5 in \cite{hytonen1}, which says that
\[
  \bigg\| \sum_{j\in\Z} \epsilon_j^\star d_j E_j g\bigg\|_{L^q(\Omega^\star\times \R^N;X_3)}
  \lesssim \|\{|d_j(\cdot)|_{X_2}\}_{j\in\Z}\|_{\mathrm{Car}^t(\mathcal{D}')}\cdot \|g\|_{L^q(\R^N;X_1)},
\]
whenever $X_1,X_2,X_3$ are three Banach spaces with $X_1$ having the RMF property, and $X_2\subseteq\mathcal{L}(X_1,X_3)$ embedded in such a way that the unit-ball $\bar{B}_{X_2}$ is $R$-bounded.

Denote $X_1=X^*$, $X_2=L^t(\Omega;\C)$, and $X_3=L^t(\Omega;X^*)$. 
Then $X_1$ is an $\mathrm{RMF}$ space by assumption. By the result of \cite{hytonen1} just stated,
it suffices to verify that the closed unit ball of $X_2$ is a Rademacher-bounded subspace of $\mathcal{L}(X_1,X_3)$ when the action of 
$\rho\in X_2$ is defined by
\[
X_1 \ni x\mapsto \rho(x):=\rho\otimes x:\rho\otimes x(\epsilon)= \rho(\epsilon)x.
\]
To this end, let $\{\rho_j\,:\,j\in\N\}$ be a sequence in $\bar B_{X_2}$ and $\{x_j\,:\,j\in\N\}$ be a sequence in $X^*$. By Fubini's theorem and Proposition \ref{improved},
\begin{align*}
\bigg(\mathbf{E}_{\epsilon^\star}\bigg\|\sum_{j=1}^\infty  & \epsilon_j^\star\rho_j\otimes x_j\bigg\|_{L^t(\Omega;X^*)}^t\bigg)^{1/t}
=\bigg\|\sum_{j=1}^\infty \epsilon_j^\star \rho_j\otimes x_j\bigg\|_{L^t(\Omega;L^t(\Omega^\star;X^*))}\\
&\lesssim \underbrace{\sup_{j\in\N} \|\rho_j\|_{L^t(\Omega)}}_{\le 1}\cdot \bigg\| \sum_{j=1}^\infty \epsilon^{\star}_j x_j\bigg\|_{L^t(\Omega^\star;X^*)}
\le \bigg(\mathbf{E}_{\epsilon^\star}Ê\bigg\|\sum_{j=1}^\infty \epsilon_j^\star x_j\bigg\|_{X^*}^t\bigg)^{1/t}.
\end{align*}
By Kahane--Khinchine inequality, this is as required.
\end{proof}

We need further preparations for establishing \eqref{estpar}. 

Recall that
$D_Q^{a,1} = (D_Q^{a,1})^2-\omega_Q^{a,1} E_Q$ by
\eqref{qr2}. Denote
\begin{equation}\label{chies}
\chi_Q:=1_Q\chi_{k-1}:=1_Q1_{\{b_{k-1}^{a,1}\not=b_k^{a,1}\}},\qquad \text{ if }Q\in\mathcal{D}_k.
\end{equation}
By Lemma \ref{ombasic}, we have
$\omega_Q^{a,1}=\chi_Q\omega_Q^{a,1}$. 
Furthermore, $\chi_Q E_Q f = E_{k-1}(\chi_Q E_Qf)$ if $Q\in\mathcal{D}_k$. Hence,
we can write
\begin{align*}
\langle \Pi g,f\rangle
&=\sum_{\substack{Q\in\mathcal{D}\\S(Q)\not=\emptyset}}
\frac{\langle g\rangle_S}{\langle b_{S^a}\rangle_S} \langle (D_Q^{a,1})^*(T^*b_{S^a}),D_Q^{a,1}f\rangle\\
&\qquad\qquad - \sum_{\substack{Q\in\mathcal{D}\\S(Q)\not=\emptyset}}
\frac{\langle g\rangle_S}{\langle b_{S^a}\rangle_S} \langle \omega_Q^{a,1}
T^*b_{S^a},E_{\log_2(\ell(Q))-1}(\chi_QE_Qf)\rangle\\
&=\int_\Omega \bigg\langle\sum_{\substack{Q\in\mathcal{D}\\S(Q)\not=\emptyset}}\epsilon_Q
\frac{\langle g\rangle_S}{\langle b_{S^a}\rangle_S}  (D_Q^{a,1})^*(T^*b_{S^a}),\sum_{\substack{Q'\in\mathcal{D}}} \epsilon_{Q'}D_{Q'}^{a,1}f\bigg\rangle d\mathbf{P}(\epsilon)\\
&\qquad -
\int_\Omega \bigg\langle \sum_{\substack{Q\in\mathcal{D}\\S(Q)\not=\emptyset}} \epsilon_Q
\frac{\langle g\rangle_S}{\langle b_{S^a}\rangle_S}  E_{\log_2(\ell(Q))-1}(\omega_Q^{a,1}T^*b_{S^a}),\sum_{\substack{Q'\in\mathcal{D}}} 
\epsilon_{Q'}\chi_{Q'}E_{Q'}f\bigg\rangle d\mathbf{P}(\epsilon).
\end{align*}
Taking the absolute values, and
using H\"older's inequality, we get
\begin{align*}
|\langle \Pi g,f\rangle| &\le
\bigg\|\sum_{\substack{Q\in\mathcal{D}\\S=S(Q)\not=\emptyset}} \epsilon_Q\frac{\langle g\rangle_S}{\langle b_{S^a}\rangle_S}  (D_Q^{a,1})^*(T^*b_{S^a})\bigg\|_q
\bigg\| \sum_{Q'\in\mathcal{D}} \epsilon_{Q'}D_{Q'}^{a,1}f\bigg\|_p\\
&\qquad\quad +\bigg\|\sum_{\substack{Q\in\mathcal{D}\\S=S(Q)\not=\emptyset}} \epsilon_Q
\frac{\langle g\rangle_S}{\langle b_{S^a}\rangle_S}  E_{\log_2(\ell(Q))-1}(\omega_Q^{a,1}T^*b_{S^a})\bigg\|_q\bigg\|\sum_{Q'\in\mathcal{D}} \epsilon_{Q'}\chi_{Q'}E_{Q'}f\bigg\|_p.
\end{align*}
Using \eqref{a_iso} and contraction principle, followed by
Theorem \ref{nests} and Lemma \ref{carut}, we see
that $|\langle \Pi g,f\rangle|$ is bounded by a sum
of two terms, both of them being (a constant multiple) of the general form
\begin{equation}\label{paraest}
\begin{split}
\bigg\| \sum_{R\in\mathcal{D}'}\sum_{\substack{Q\in\mathcal{D}\\ S(Q)=R}} \epsilon_Q
\pi_{Q,R^a}\langle g\rangle_R
\bigg\|_q\cdot \|f\|_p.
\end{split}
\end{equation}
Here the two terms are determined by the following choices:
\begin{equation}\label{pi_choice}
\pi_{Q,R^a}\in \{(D_Q^{a,1})^*(T^*b_{R^a}),E_{\log_2(\ell(Q))-1}(\omega_Q^{a,1}T^*b_{R^a})\}.
\end{equation}
Observe
that, if $R$ and $Q$ are as in \eqref{paraest}, then 
\begin{equation}\label{pieq}\pi_{Q,R^a}=1_R\pi_{Q,R^a}.\end{equation}

In order to estimate quantities of the form \eqref{paraest}, we will use the following lemma.

\begin{lem}\label{boundest}
Assume that $U\in\mathcal{D}'$ and  $t\in (1,\infty)$.
Then
\begin{equation}\label{sest2}
\bigg\|\sum_{\substack{R\in\mathcal{D}':R\subset U\\R^a=U^a}}\sum_{\substack{Q\in\mathcal{D} \\ÊS(Q)=R}} \epsilon_Q \pi_{Q,U^a}\bigg\|_{L^t(\R^N\times\Omega;\C)}\lesssim 
\mu(U)^{1/t}.
\end{equation}
\end{lem}

\begin{proof}
Denote $h=T^*b_{U^a}$ and first consider the case $\pi_{Q,U^a} = (D_Q^{a,1})^*(h)$.
Because 
$Q\subset S(Q)$ if $S(Q)\not=\emptyset$, we see that the left hand side of \eqref{sest2} is 
\begin{align*}
\bigg\|\sum_{\substack{R\in\mathcal{D}':R\subset U\\R^a=U^a}}\sum_{\substack{Q\in\mathcal{D} \\ÊS(Q)=R}} \epsilon_Q (D_Q^{a,1})^*(1_Uh)\bigg\|_t \le \bigg\|\sum_{\substack{Q\in\mathcal{D}}} \epsilon_Q (D_Q^{a,1})^*(1_Uh)\bigg\|_{t}.
\end{align*}
Using Theorem \ref{haa} with $X=\C$, followed
by \eqref{accreass}, we find that
the last quantity is bounded
by a constant multiple of
\begin{align*}
\|1_U h\|_{L^t(\R^N,\mu;\C)}\le \|h\|_{L^\infty(\R^N,\mu;\C)}\|1_U\|_{L^t(\R^N,\mu;\C)}=\mu(U)^{1/t}\|h\|_{L^\infty(\R^N,\mu;\C)}\le B\mu(U)^{1/t}.
\end{align*}
This is the required estimate in the present case.

Then consider the case
\[\pi_{Q,U^a}=E_{\log_2(\ell(Q))-1}(\omega_Q^{a,1}h).\] Recall that
the expectation is taken with respect to $\mathcal{D}_{\log_2(\ell(Q))-1}$.
By the contraction principle and the facts that $Q\subset S(Q)$ if $S(Q)\not=\emptyset$ and $\omega_Q^{a,1}=\chi_Q\omega_Q^{a,1}$, see \eqref{chies},
we get
\begin{align*}
\bigg\| \sum_{\substack{R\in\mathcal{D}':R\subset U\\R^a=U^a}}
\sum_{\substack{Q\in\mathcal{D} \\ S(Q)=R}}
\epsilon_{Q} 
\pi_{Q,U^a}\bigg\|_t
&\lesssim \bigg\|\sum_{Q\in\mathcal{D}}
\epsilon_{Q} \chi_Q
E_{\log_2(\ell(Q))-1}(\omega_{Q}^{a,1} 1_Uh)\bigg\|_t\\
&\lesssim \bigg\|\sum_{k\in\Z}\epsilon_k\chi_{k-1}\sum_{Q\in\mathcal{D}_k}
1_Q
E_{k-1}(\omega_{Q}^{a,1} 1_Uh)\bigg\|_t\\
&\lesssim
\bigg\|\sum_{k\in\Z}\epsilon_k\chi_{k-1}E_{k-1}(\omega_{k}^{a,1} 1_Uh)\bigg\|_t.
\end{align*}
Here $\chi_{k-1}=1_{\{b_{k-1}^{a,1}\not=b_k^{a,1}\}}$ satisfies $\chi_{k-1}=E_{k-1} \chi_{k-1}$. Also, $\sup_{k\in\Z}Ê\|\omega_{k}^{a,1}\|_{L^\infty(\mu)}\lesssim 1$ 
by Lemma \ref{ombasic}. Hence, by Proposition \ref{ylcarl} with $X=\C$,
\begin{align*}
\bigg\|\sum_{k\in\Z}\epsilon_k\chi_{k-1}E_{k-1}(\omega_{k}^{a,1} 1_Uh)\bigg\|_t
\lesssim \|\{\chi_k\}_{k\in\Z}\|_{\mathrm{Car}^1(\mathcal{D})}\cdot \|1_Uh\|_{L^t(\R^N,\mu;\C)}.
\end{align*}
Using \eqref{accreass} and reasoning as in the proof of Lemma \ref{vals}, we 
conclude that the right hand side above is bounded by a constant multiple of $\mu(U)^{1/t}$.
\end{proof}

We finish the proof of \eqref{estpar}. 

Recall 
that it suffices to  estimate  \eqref{paraest}. 
Fix a real number $t>q\vee s$, where $s\in [2,\infty)$ is such that $X^*$ has cotype $s$. 
Let us also introduce Rademacher variables $\epsilon'=\{\epsilon'_R\}_{R\in\mathcal{D}'}\in\Omega'$ 
that are independent
of $\{\epsilon_Q\}_{Q\in\mathcal{D}}$. By \eqref{pieq}
\begin{align}\label{sell}
&\bigg\| \sum_{R\in\mathcal{D}'}\sum_{\substack{Q\in\mathcal{D}\\ S(Q)=R}} \epsilon_Q
\pi_{Q,R^a}\langle g\rangle_R
\bigg\|_{L^q(\R^N\times\Omega;X^*)}\notag\\&=\bigg\|\sum_{R\in\mathcal{D}'}{\epsilon'_R}\bigg(\sum_{\substack{Q\in\mathcal{D}\\ S(Q)=R}} \epsilon_Q \pi_{Q,R^a}\bigg)1_R\langle g\rangle_R\bigg\|_{L^q(\Omega'\times\R^N\times \Omega;X^*)}.
\end{align}
By H\"older's inequality 
\eqref{sell} is bounded by
\begin{equation}\label{difficult_start}
\bigg\| \sum_{j\in \Z}Ê\epsilon_j^\star d_j E_j g\bigg\|_{L^q(\Omega^\star\times\R^N\times \Omega;X^*)}\le 
\bigg\| \sum_{j\in \Z}Ê\epsilon_j^\star d_j E_j g\bigg\|_{L^q(\Omega^\star\times\R^N;L^t(\Omega;X^*))},
\end{equation}
where $\epsilon^\star=\{\epsilon_j^\star\,:\,j\in\Z\}\in \Omega^\star$ are Rademacher random variables and
\[
d_j:\R^N\to L^t(\Omega):x\mapsto \bigg(\epsilon \mapsto \sum_{R\in\mathcal{D}_j'} \sum_{\substack{Q\in\mathcal{D}\\S(Q)=R}} \epsilon_Q\pi_{Q,R^a}(x)\bigg).
\]
Concluding from above and using
Lemma \ref{genemb}, we see that left hand side of \eqref{sell} is bounded by (a constant multiple of)
\begin{equation}\label{Ies}
\begin{split}
\sup_{\substack{P\in\mathcal{D}'\\ \mu(P)\not=0}}\frac{1}{\mu(P)^{1/t}}\cdot
\underbrace{\bigg\| 1_P \sum_{j:2^j\le \ell(P)}\epsilon_j^\star
\|d_j(\cdot)\|_{L^t(\Omega;\C)}\bigg\|_{L^t(\R^N\times \Omega^\star;\R)}}_{=:\Sigma(P)}\cdot\|g\|_q.
\end{split}
\end{equation}
To estimate the Carleson norm, we fix $P\in\mathcal{D}'$ for which $\mu(P)\not=0$.  By \eqref{pieq},
\begin{align*}
\Sigma(P)
&=\bigg\| 
\sum_{\substack{R\in\mathcal{D}'\\R\subset P}} \epsilon_R'
\bigg\|\sum_{\substack{Q\in\mathcal{D}\\S(Q)=R}}
\epsilon_Q \pi_{Q,R^a}\bigg\|_{L^t(\Omega;\C)}\bigg\|_{L^t(\R^N\times \Omega';\C)}.
\end{align*}
By Khinchine and Kahane--Khinchine inequalities,
\begin{equation}\label{cotype_cont}
\Sigma(P)\lesssim
\bigg\|   \bigg(\sum_{\substack{R\in\mathcal{D}'\\R\subset P}} 
\bigg\|\sum_{\substack{Q\in\mathcal{D}\\S(Q)=R}}
\epsilon_Q \pi_{Q,R^a}\bigg\|^2_{L^2(\Omega;\C)}\bigg)^{1/2}
\bigg\|_{L^t(\R^N;\C)}.
 \end{equation}
Since  $L^2(\Omega;\C)$ has cotype $2$, see \eqref{cotype_s}, we obtain
\begin{align*}
&\Sigma(P)\lesssim\bigg\|\,
\bigg\|
\sum_{\substack{R\in\mathcal{D}'\\R\subset P}}\epsilon_{R}'
\sum_{\substack{Q\in\mathcal{D}\\S(Q)=R}}
\epsilon_Q \pi_{Q,R^a}\bigg\|_{L^2(\Omega';L^2(\Omega;\C))}
 \bigg\|_{L^t(\R^N;\C)}\\
 &=\bigg\|\,\bigg\|
\sum_{\substack{R\in\mathcal{D}'\\R\subset P}}
\sum_{\substack{Q\in\mathcal{D}\\S(Q)=R}}
\epsilon_Q \pi_{Q,R^a}\bigg\|_{L^2(\Omega;\C)}
 \bigg\|_{L^t(\R^N;\C)}
\lesssim
 \bigg\|
\sum_{\substack{R\in\mathcal{D}'\\R\subset P}}
\sum_{\substack{Q\in\mathcal{D}\\S(Q)=R}}
\epsilon_Q \pi_{Q,R^a}
 \bigg\|_{L^t(\R^N\times\Omega;\C)}.
\end{align*}

Suppose that $M\in\N_0$ is such that $P^a\in\mathcal{D}'^M$. Because
cubes in a fixed layer $\mathcal{D}'^m$, $m>M$, are disjoint, we can estimate
as follows
\begin{equation}\label{tseuraa}
\begin{split}
\Sigma(P)&\lesssim \bigg\|\sum_{\substack{R\in\mathcal{D}':R\subset P \\ R^a = P^a}}\sum_{\substack{Q\in\mathcal{D} \\ÊS(Q)=R}} \epsilon_Q\pi_{Q,{P^a}}\bigg\|_{L^t(\R^N\times\Omega;\C)}\\
&\qquad\qquad + \sum_{m=M+1}^\infty \bigg(\sum_{\substack{U\in\mathcal{D}'^m \\ÊU\subsetneq P}}
\bigg\|\sum_{\substack{R\in\mathcal{D}':R\subset U\\R^a=U}}\sum_{\substack{Q\in\mathcal{D} \\ÊS(Q)=R}}
\epsilon_Q\pi_{Q,U}\bigg\|_{L^t(\R^N\times\Omega;\C)}^t
\bigg)^{1/t}.
\end{split}
\end{equation}
Using Lemma \ref{boundest} and Lemma \ref{basicsum}, we can estimate the right hand side of \eqref{tseuraa}
as follows
\begin{align*}
\Sigma(P)&\lesssim \mu(P)^{1/t}+\sum_{m=M+1}^\infty\bigg(\sum_{U\in\mathcal{D}'^m:U\subsetneq P} \mu(U)\bigg)^{1/t}
\\
&\lesssim \mu(P)^{1/t}+\sum_{m=M+1}^\infty (1-\tau)^{((m-M)-1)/t}\mu(P)^{1/t}\lesssim \mu(P)^{1/t}.
\end{align*}
The proof of  \eqref{estpar} finishes by substituting the estimate above in \eqref{Ies}.

\subsection*{Proving estimate \eqref{estpar2}}
Randomizing and using H\"older's inequality as in connection
with the paraproduct operator, we get
\begin{align*}
&\bigg|\sum_{\substack{Q\in\mathcal{D}\\S(Q)\not=\emptyset}} \bigg\langle T^*b_{R_0}\frac{\langle g\rangle_{\R^N}}{\langle b_{R_0}\rangle_{\R^N}},D_Q^{a,1} f\bigg\rangle\bigg|\\&
\lesssim |\langle g\rangle_{\R^N}|\cdot \|f\|_p\cdot \bigg\{\bigg\|\sum_{\substack{Q\in\mathcal{D}}} \epsilon_Q(D_Q^{a,1})^*(T^*b_{R_0})\bigg\|_q
+\bigg\|\sum_{\substack{Q\in\mathcal{D}}} \epsilon_Q
E_{\log_2(\ell(Q))-1}(\omega_Q^{a,1}T^*b_{R_0})\bigg\|_q\bigg\}.
\end{align*}
Observe that $|\langle g\rangle_{\R^N}|\lesssim \mu(\R^N)^{-1/q}\|g\|_q$. Hence, it suffices to show
that the quantity inside the parentheses is bounded by a constant multiple of $\mu(\R^N)^{1/q}$.

To this end, we first use Theorem
\ref{haa} with $X=\C$ and \eqref{accreass}, we get
\[
\bigg\|\sum_{\substack{Q\in\mathcal{D}}} \epsilon_Q(D_Q^{a,1})^*(T^*b_{R_0})\bigg\|_q\lesssim \|T^* b_{R_0}\|_q\lesssim \mu(\R^N)^{1/q}.
\]
On the other
hand, since the family $\{E_k\}_{k\in\Z}$ 
of operators in $L^q(\R^N,\mu)$ 
is $\mathcal{R}$-bounded 
by Stein's inequality
\cite{stein}, we find that
\begin{align*}
&\bigg\|\sum_{\substack{Q\in\mathcal{D}}} \epsilon_Q
E_{\log_2(\ell(Q))-1}(\omega_Q^{a,1}T^*b_{R_0})\bigg\|_q\\
&=\bigg\|\sum_{k\in\Z} \epsilon_k
E_{k-1}(\omega_k^{a,1}T^*b_{R_0})\bigg\|_q\lesssim \|T^*b_{R_0}\|_\infty\bigg\|\sum_{k\in\Z}\epsilon_k \omega_k^{a,1}\bigg\|_q.
\end{align*}
By \eqref{accreass}, we have $\|T^*b_{R_0}\|_\infty\lesssim 1$.
Because
$|\omega_{k}^{a,1}|\lesssim 1_{\{b_{k-1}^a\not= b_{k}^a\}}$ $\mu$-almost everywhere, see
Lemma \ref{ombasic}, 
we can use Lemma \ref{carut} with $f\equiv 1$ for
\[
\bigg\|\sum_{k\in\Z}\epsilon_k \omega_k^{a,1}\bigg\|_q\lesssim \|1\|_q\lesssim \mu(\R^N)^{1/q}.
\]
This conclude the proof of estimate \eqref{estpar2}.

\section{Preparations for comparable cubes}\label{comparable}

During the course of the present and following section,
we prove Proposition~\ref{complem}. It controls a part of
the summation in \eqref{firstes}, involving
cubes that are close to each
other in their position and size. 

We write
$Q\sim R$ for $Q\in\mathcal{D}$ and $R\in\mathcal{D}'$ if 
\begin{equation}\label{rems}
2^{-r}\ell(R)\le \ell(Q)\le \ell(R)\text{ and }\dist(Q,R)<\ell(Q)=\ell(Q)\wedge \ell(R).
\end{equation} 
Note that if $Q\sim R$, then $\ell(Q)\le \ell(R)\le D(Q,R)\le (2+2^r)\ell(Q)$, so that
all of these quantities are comparable.

A few words about implicit constants:
In the previous sections we have performed estimates where the implicit constants
can depend on the parameter $r$, introduced in Section \ref{preparations}.
At this stage we introduce two new auxiliary parameters $\eta\in (0,1)$ and $\epsilon\in (0,1)$. In the sequel
we need to keep track of the dependence of estimates
on the parameters $r,\epsilon$ and $\eta$  explicitly.

For the following proposition, we recall that all $\mathrm{UMD}$ spaces have
a finite cotype.

\begin{prop}\label{complem}
Under the assumptions of Theorem \ref{mainth}, we have
\begin{equation}\label{remain}
\begin{split}
&\mathbf{E}_{\mathcal{D}}\mathbf{E}_{\mathcal{D}'}\bigg|\sum_{R\in\mathcal{D}'} \sum_{\substack{Q\in\mathcal{D}_{R\text{-good}}\\Q\sim R}} \langle D_R^{a,2} g,T(D_Q^{a,1} f)\rangle\bigg|\\
&\lesssim 
\big(C(r,\eta,\epsilon)
+(C(r,\eta)\epsilon^{1/t} + C(r)\eta^{1/t})\|T\|_{\mathcal{L}(L^p(\mu;X))}\big)\|g\|_{L^q(\mu;X^*)}\|f\|_{L^p(\mu;X)}
\end{split}
\end{equation}
for every $f\in L^p(X)$ and $g\in L^q(X^*)$. Here $1/p+1/q=1$ and $t>(s\vee q)\vee p$, where both
$X$ and $X^*$ have cotype $s\in [2,\infty)$.
\end{prop}

The strategy of the  proof of this proposition is as follows: 
at the end of this section we consider a separated part of the summation in \eqref{remain},
where expectations over dyadic systems are not required. In the following section
a (more complicated) intersecting part of the sum in \eqref{remain} is treated, and the expectations
are crucial therein.

Here are preparations for the proof of Proposition \ref{complem}:
given $R\in\mathcal{D}'$, there are at most $C=C(r,N)$
cubes $Q\in\mathcal{D}$ satisfying \eqref{rems}. Hence,
without essential loss of generality,
it suffices consider a finite number of subseries of the general form
\begin{equation}\label{subser}
\mathbf{E}_{\mathcal{D}}\mathbf{E}_{\mathcal{D}'}\bigg|\sum_{R\in\mathcal{D}'}  \langle D_R^{a,2} g,T(D_Q^{a,1} f)\rangle\bigg|,
\end{equation}
where $Q=Q(R)\in \mathcal{D}_{R\text{-good}}$ inside the summation satisfies $Q\sim R$. 
At this stage we fix one series like this, and  the convention that $Q$ is implicitly a function
of $R$ will be maintained without further notice. Furthermore, without 
loss of generality, it is possible to act as if the map $R\mapsto Q(R)$ was
invertible, so that \eqref{subser} could also be written in terms of the
summation variable $Q\in\mathcal{D}$.

Proceeding as in Section \ref{nested}, we find that
\eqref{subser}
can be dominated from above by a sum of nine terms, each  of them being of
the general form
\begin{equation}\label{eeqa}
\sum_{i,j=1}^{2^N} \mathbf{E}_{\mathcal{D}}\mathbf{E}_{\mathcal{D}'}\bigg|\sum_{R\in\mathcal{D}'} 
\langle g_{R}\rangle_{R_j}
\langle 1_{R_j}\psi_{R,j},T(1_{Q_i}\phi_{Q,i})\rangle\langle f_Q\rangle_{Q_i}\bigg|,\\
\end{equation}
where  $\langle g_R\rangle_{R_j}=\langle 1_R g_k\rangle_{R_j}=\langle g_k\rangle_{R_j}$ if
$R\in\mathcal{D}'_k$ (similarly for $f$'s), and the summands
are determined by the following choices:
\begin{equation}\label{gcoa_l}
(g_k,\psi_{R,j})\in \{(s_k, b_{R^a}), (h_k, b_{R_j^a}),
(u_k, b_{R^a})\}
\end{equation}
and
\begin{equation}\label{fcoa_l}
(f_k,\phi_{Q,i})\in \{(\bar s_k, b_{Q^a}), (\bar h_k,b_{Q_i^a}),
(\bar u_k,b_{Q^a})\}.
\end{equation}
Here
\begin{align*}
s_{k}=1_{\{b_{k-1}^{a,2}=b_{k}^{a,2}\}}\bigg(\frac{E_{k-1} g}{E_{k-1}b_{k-1}^{a,2}}
- \frac{E_k g}{E_k b_k^{a,2}}\bigg),\quad \bar s_{k}=1_{\{b_{k-1}^{a,1}=b_{k}^{a,1}\}}\bigg(\frac{E_{k-1} f}{E_{k-1}b_{k-1}^{a,1}}
- \frac{E_k f}{E_k b_k^{a,1}}\bigg);
\end{align*}
\[
h_{k}=1_{\{b_{k-1}^{a,2}\not=b_{k}^{a,2}\}}\frac{E_{k-1} g}{E_{k-1}b_{k-1}^{a,2}},\quad
\bar h_{k}=1_{\{b_{k-1}^{a,1}\not=b_{k}^{a,1}\}}\frac{E_{k-1} f}{E_{k-1}b_{k-1}^{a,1}};
\]
and
\[
u_{k}=
-1_{\{b_{k-1}^{a,2}\not=b_{k}^{a,2}\}}\frac{E_k g}{E_k b_k^{a,2}},\quad  \bar u_{k}=
-1_{\{b_{k-1}^{a,1}\not=b_{k}^{a,1}\}}\frac{E_k f}{E_k b_k^{a,1}}.
\]
Observe that, in any case, $E_{k-1} g_{k}=g_{k}$ and
$E_{k-1} f_k=f_k$.

Fix $i,j\in \{1,2,\ldots,N\}$.

For each cube $Q$ in $\R^N$, define the boundary region
\[
\delta_Q^\eta:= (1+\eta)Q\setminus (1-\eta)Q,
\]
where the parameter $\eta>0$ is to be chosen later.
If $R\in\mathcal{D}'$ and $Q=Q(R)$, we write
\begin{align*}
&Q_{i,\partial}:=Q_i\cap \delta^\eta_{R_j};\quad Q_{i,\mathrm{sep}} :=(Q_i\setminus Q_{i,\partial})\setminus Q_i\cap R_j;\quad \Delta_{Q_i}:=Q_i\cap R_j\setminus Q_{i,\partial};\\
&R_{j,\partial}:=R_j\cap \delta^\eta_{Q_i};\quad R_{j,\mathrm{sep}} :=(R_j\setminus R_{j,\partial})\setminus Q_i\cap R_j;\quad \Delta_{R_j}:=Q_i\cap R_j\setminus R_{j,\partial}.
\end{align*}
Observe that the following unions are disjoint:
\[
Q_i=\Delta_{Q_i}\cup Q_{i,\mathrm{sep}}\cup Q_{i,\partial},\quad R_j=\Delta_{R_j}\cup R_{j,\mathrm{sep}}\cup R_{j,\partial}.
\]
Hence, we  can
write the matrix coefficient in \eqref{eeqa}  as
\begin{equation}\label{tenf}
\begin{split}
\langle 1_{R_j}\psi_{R,j},T(1_{Q_i}\phi_{Q,i})\rangle &= \langle 1_{R_{j,\mathrm{sep}}}
\psi_{{R,j}},T(1_{Q_i}\phi_{Q,i})\rangle + \langle 1_{R_{j,\partial}}\psi_{R,j},T(1_{Q_i}\phi_{Q,i})\rangle\\
&\qquad\qquad +\langle 1_{\Delta_{R_j}}\psi_{R,j},T(1_{\Delta_{Q_i}}\phi_{Q,i})\rangle\\
&\quad +\langle 1_{\Delta_{R_j}} \psi_{R,j},T(1_{Q_{i,\partial}}\phi_{Q,i})\rangle + \langle
1_{\Delta_{R_j}} \psi_{R,j},T(1_{Q_{i,\mathrm{sep}}}\phi_{Q,i})\rangle\\
&=:M_1(R)+M_2(R)+M_3(R)+M_4(R)+M_5(R).
\end{split}
\end{equation}
Using these preparations, it suffices to estimate
the following quantity:
\begin{equation}\label{reds}
\mathbf{E}_{\mathcal{D}}\mathbf{E}_{\mathcal{D}'}\bigg|\sum_{\substack{R\in\mathcal{D}'}} \langle g_{R}\rangle_{R_j} (M_1(R)+M_2(R)+M_3(R)+M_4(R)+M_5(R)) \langle f_Q\rangle_{Q_i}
\bigg|,
\end{equation}
where $i,j\in \{1,2,\ldots,2^N\}$ and $Q=Q(R)\in\mathcal{D}_{R\text{-good}}$ satisfies
the condition $Q\sim R$ inside the summation.

\subsection*{The separated part}
Recall that our aim now is to prove Proposition \ref{complem}.
We have reduced
this to a problem of estimating the sum \eqref{reds} involving, among others, terms
of the form
\[
M_1(R)+M_5(R)=\langle 1_{R_{j,\mathrm{sep}}}
\psi_{{R_j}},T(1_{Q_i}\phi_{Q,i})\rangle+\langle
1_{\Delta_{R_j}} \psi_{R,j},T(1_{Q_{i,\mathrm{sep}}}\phi_{Q,i})\rangle,
\]
where $Q=Q(R)\in\mathcal{D}_{R\textrm{-good}}$ satisfies $Q\sim R$.
In both cases,  $M_1$ and $M_5$, the two indicators
are associated with sets separated from each other.  Hence, a
decoupling estimate can be used to establish the following lemma.

\begin{lem}\label{m1lem}
Suppose that $f\in L^p(\R^N,\mu;X)$ and $g\in L^q(\R^N,\mu;X^*)$. Then
\begin{align*}
\bigg|
\sum_{\substack{R\in\mathcal{D}'}} \langle g_{R}\rangle_{R_j} (M_1(R)+M_5(R)) \langle f_Q\rangle_{Q_i}
\bigg|
\lesssim C(r,\eta)\|g\|_{q}\|f\|_{p}.
\end{align*}
\end{lem}

\begin{proof}
We focus on summation over the terms $M_1(R)$. The treatment of summation over the terms $M_5(R)$ is analogous.
If $R\in\mathcal{D}'$ and $Q\in\mathcal{D}_{R\text{-good}}$ satisfies $\ell(Q)\le \ell(R)$, we write
$T_{RQ}:=1_{Q=Q(R)}\langle 1_{R_{j,\mathrm{sep}}}\psi_{{R,j}},T(1_{Q_i}\phi_{Q,i})\rangle$.
It suffices to estimate the series
\[
\Sigma:=\sum_{\substack{R\in\mathcal{D}'}}\sum_{\substack{Q\in\mathcal{D}_{R\text{-good}} \\ \ell(Q)\le\ell(R)}} \langle g_{R}\rangle_{R_j} T_{RQ} \langle f_Q\rangle_{Q_i}.
\]
Assume that $T_{RQ}\not=0$ inside the summation. Then
 $Q=Q(R)$, so that $Q\sim R$ and, by
\eqref{action} and \eqref{size}, 
\begin{align*}
|T_{RQ}|&=\bigg|\int_{R_{j,\mathrm{sep}}} \int_{Q_i} \psi_{R,j}(x)K(x,y)\phi_{Q,i}(y)d\mu(y)d\mu(x)\bigg|\\
&\le \frac{\mu(R_{j,\mathrm{sep}})\mu({Q_i})}{\dist(R_{j,\mathrm{sep}},{Q_i})^d}\lesssim C(\eta)\frac{\mu({R_j})\mu({Q_i})}{\ell({R_j})^d}
\simeq C(\eta)\mu({R_j})\mu({Q_i})\frac{\ell(Q)^{\alpha/2}\ell(R)^{\alpha/2}}{D(Q,R)^{d+\alpha}}.
\end{align*}
Using Lemma \ref{tut} and then estimating as in the end of Section \ref{nested}, we find that
\[
|\Sigma|\lesssim C(r,\eta)\bigg\|\sum_{k=-\infty}^\infty \epsilon_k g_{k}\bigg\|_{L^{q}(\mathbf{P}\otimes\mu;X^*)}
\bigg\|\sum_{k=-\infty}^\infty \epsilon_k f_{k}\bigg\|_{L^p(\mathbf{P}\otimes \mu;X)}\lesssim C(r,\eta)\|g\|_q\|f\|_p.
\]
\end{proof}

\section{Intersecting part of comparable cubes}\label{intersecting}

In this section we deal with the remaining part of the comparable cubes, finishing the proof of Proposition \ref{complem}.
This will be the most technical part of the entire proof: It still involves various further decompositions and case-by-case analysis, until all different pieces are finally estimated.

Recalling the preparations in Section \ref{comparable}, we observe that it remains to estimate a summation like \eqref{reds} but involving 
only terms of the form $M_2(R)+M_3(R)+M_4(R)$. Part of this summation
involves boundary terms that are
handled by probabilistic methods, e.g. by taking
expectations over the random dyadic systems $\mathcal{D}$ and $\mathcal{D}'$, but
we will also introduce  a third random dyadic system $\mathcal{D}^\star$.
The assumption
that there is an $L^\infty$-accretive system for $T^*$ is
used to handle the non-boundary terms.

We aim to prove the following lemma.

\begin{lem}\label{m3lem}
We have
\begin{align*}
&\mathbf{E}_{\mathcal{D}}\mathbf{E}_{\mathcal{D}'}\bigg|\sum_{\substack{R\in\mathcal{D}'}} \langle g_{R}\rangle_{R_j}
  \big( M_2(R)+M_3(R)+M_4(R)\big)\langle f_Q\rangle_{Q_i}
\bigg|\\&
\lesssim \big(C(r,\eta,\epsilon)
+(C(r,\eta)\epsilon^{1/t} + C(r)\eta^{1/t})\|T\|_{\mathcal{L}(L^p(\mu;X))}\big)\|g\|_{L^q(\mu;X^*)}\|f\|_{L^p(\mu;X)}.
\end{align*}
\end{lem}

The 
 proof of this lemma is a consequence of various lemmata, namely: \ref{alpha4}, \ref{aep}, and \ref{etalem}.
Let us briefly indicate the structure of the proof.
Since $M_2(R)$ and $M_4(R)$ are so called $\eta$-boundary terms, the
main difficulties lie in estimating summation involving  the term
\[
M_3(R)=\langle 1_{\Delta_{R_j}}\psi_{R,j},T(1_{\Delta_{Q_i}}\phi_{Q,i})\rangle=\alpha_1(R)+\alpha_2(R)+\alpha_3(R),
\]
where the last decomposition depends
on a new random dyadic system $\mathcal{D}^\star$, see \eqref{ekahaj}. The terms $\alpha_2(R)$
and $\alpha_3(R)$ are also $\eta$-boundary terms.

The term $\alpha_1(R)$ will further
be expanded in \eqref{alpha1} and Lemma \ref{talkaa} as \[\alpha_1(R)=A_1(R)+A_2(R)+A_3(R)'+(A_3(R)-A_3'(R)),\] where
$A_1(R)$, $A_2(R)$, and $A_3'(R)$ are so called
$\epsilon$-boundary terms.  Hence, the main obstacle
is to estimate $A_3(R)-A_3'(R)$; 
the assumption that there is an $L^\infty$-accretive systems for
$T^*$ will be exploited here.

\subsection*{Decomposition of $M_3(R)$}
In order to decompose $M_3(R)$, we first  introduce a random dyadic system \[\mathcal{D}^\star=\mathcal{D}(\beta^\star)\] 
that is independent of both $\mathcal{D}$ and $\mathcal{D}'$.
Fix $j(\eta)\in\Z$ such that
$\eta/64\le 2^{j(\eta)}<\eta/32$.
Then, for every $R\in\mathcal{D}'$, we define a family
\[\mathcal{G}=\mathcal{G}(R):=\mathcal{D}^\star_{\log_2(s)}\] of cubes with side length \begin{equation}\label{sl}
s=2^{j(\eta)}\ell(Q_i)=2^{j(\eta)}\cdot (\ell(Q_i)\wedge \ell(R_j)),\end{equation}
where $Q=Q(R)\in\mathcal{D}$.
More precisely, $\mathcal{G}$ is a subfamily of $\mathcal{D}^\star$ that depends on $R$, $Q=Q(R)$, and $\eta$.

Let $\Delta_{Q_i}^\mathcal{G},\Delta_{R_j}^\mathcal{G}\subset Q_i\cap R_j$ be
the following adaptations of $\Delta_{Q_i}$ and $\Delta_{R_j}$ to $\mathcal{G}$: we enlargen the latter sets
so that the boundary of $\Delta_{Q_i}^\mathcal{G}\cap \Delta_{R_j}^\mathcal{G}$ doesn't intersect
interiors of cubes in $\mathcal{G}$, and
$5G\subset Q_i\cap R_j$ if the interior of $G\in \mathcal{G}$ intersects
$\Delta_{R_j}^\mathcal{G}\cap \Delta_{Q_i}^\mathcal{G}$. We also require that we can write
 \[
\Delta_{Q_i}^\mathcal{G}=\Delta_{Q_i}\cup \Delta^\partial_{Q_i},\quad \Delta_{R_j}^\mathcal{G} = \Delta_{R_j}\cup \Delta^\partial_{R_j}
 \]
such that the unions are disjoint,
$\Delta^\partial_{Q_i}\subset  Q_{i,\partial}\cap R_j$, and $\Delta^\partial_{R_j}\subset R_{j,\partial}\cap Q_i$.
Now
observe that 
\begin{equation}\label{ekahaj}
\begin{split}
&M_3(R)=\langle 1_{\Delta_{R_j}}\psi_{R,j},T(1_{\Delta_{Q_i}}\phi_{Q,i})\rangle
\\&=\langle 1_{\Delta_{R_j}^\mathcal{G}}\psi_{R,j},T(1_{\Delta_{Q_i}^\mathcal{G}}\phi_{Q,i})\rangle
-\langle 1_{\Delta_{R_j}^\partial}\psi_{R,j},T(1_{\Delta_{Q_i}^\mathcal{G}}\phi_{Q,i})\rangle
-\langle 1_{\Delta_{R_j}}\psi_{R,j},T(1_{\Delta_{Q_i}^\partial}\phi_{Q,i})\rangle\\
&=:\alpha_1(R)+\alpha_2(R)+\alpha_3(R).
\end{split}
\end{equation}
The terms in this decomposition depend on $\mathcal{D}^\star$.

In order to define $\epsilon$-boundary terms, we let $R\in\mathcal{D}'$ and write
\[
G_\epsilon=G_\epsilon(R) = \bigcup_{G\in \mathcal{G}(R)} \delta_G^\epsilon,\quad \delta_G^\epsilon=(1+\epsilon)G\setminus (1-\epsilon)G.
\]
We also write $\tilde G=G\setminus G_\epsilon$ if $G\in\mathcal{G}=\mathcal{G}(R)$. Define
\[\Delta'_{Q_i}=\Delta^\mathcal{G}_{Q_i} \cap G_\epsilon,\qquad \tilde \Delta_{Q_i}=\Delta^\mathcal{G}_{Q_i}\setminus G_\epsilon\] and
similarly for $\Delta_{R_j}^\mathcal{G}$. Then we have the disjoint unions
\[
\Delta_{Q_i}^\mathcal{G} = \Delta'_{Q_i}\cup \tilde \Delta_{Q_i},\qquad \Delta_{R_j}^\mathcal{G}= \Delta'_{R_j}\cup \tilde \Delta_{R_j}.
\]
Hence, we can write
\begin{equation}\label{alpha1}
\begin{split}
&\alpha_1(R)=\langle 1_{\Delta_{R_j}^\mathcal{G}}\psi_{R,j},T(1_{\Delta_{Q_i}^\mathcal{G}}\phi_{Q,i})\rangle
\\&=\langle 1_{\Delta'_{R_j}}\psi_{R,j},T(1_{\Delta_{Q_i}^\mathcal{G}}\phi_{Q,i})\rangle+ \langle 1_{\tilde \Delta_{R_j}}\psi_{R,j},T(1_{\Delta'_{Q_i}}\phi_{Q,i})\rangle
+\langle 1_{\tilde \Delta_{R_j}}\psi_{R,j},T(1_{\tilde \Delta_{Q_i}}\phi_{Q,i})\rangle
\\&=:A_1(R)+A_2(R)+A_3(R).
\end{split}
\end{equation}


\subsection*{Estimate for a non-boundary part}

We need to extract the non-boundary terms. This is done in the following
lemma which gives us a decomposition of $A_3(R)$; therein $A_3(R)-A_3'(R)$ is a non-boundary term.
The proof of the lemma uses the fact that there is an $L^\infty$-accretive system for $T^*$.

\begin{lem}\label{talkaa}
Let $R\in\mathcal{D}'$. Then
 $A_3(R)$ can be written as $A_3'(R)+\big(A_3(R)-A_3'(R)\big)$, where
\[
|A_3(R)-A_3'(R)|\lesssim C(r,\eta,\epsilon)\mu(Q_i\cap R_j)\]
and
there are functions $b_{R,G,j}:\R^N\to \C$, satisfying
$\|b_{R,G,j}\|_{L^\infty(\mu)}\lesssim 1$ if $R\in\mathcal{D}'$ and $G\in\mathcal{G}(R)$, such
that
\[
A_3'(R)=\sum_{\substack{G\in \mathcal{G}(R)\\\tilde G\subset \Delta_{Q_i}^\mathcal{G}\cap \Delta_{R_j}^\mathcal{G}}} \langle 1_{\tilde G}b_{R,G,j},T(1_{(1+\epsilon)\tilde G\setminus \tilde G}\phi_{Q,i})\rangle.
\]
Here $\tilde G=G\setminus G_\epsilon$ for every $G\in\mathcal{G}(R)$.
\end{lem}

\begin{proof}
We expand
$A_3(R)$ into a double series, where
a typical summand is of the
form
\begin{equation}\label{typsum}
\langle 1_G 1_{\tilde \Delta_{R_j}}\psi_{R,j},T(1_{H}1_{\tilde \Delta_{Q_i}}\phi_{Q,i})\rangle,\quad G,H\in \mathcal{G}.
\end{equation}
Let us begin with estimating these quantities, and
there are two 
cases to be treated.

First, if $G\not=H$, then 
\[\ell(Q_i)\lesssim C(\eta,\epsilon)\,\mathrm{dist}(G\cap \tilde \Delta_{R_j},H\cap \tilde \Delta_{Q_i}).\]
Hence, by \eqref{action} and \eqref{size},
\begin{equation}\label{kestim}
\begin{split}
&|\langle 1_G 1_{\tilde \Delta_{R_j}}\psi_{R,j},T(1_{H}1_{\tilde \Delta_{Q_i}}\phi_{Q,i})\rangle|
\\&=\bigg|\int_{G\cap \tilde \Delta_{R_j}} \int_{H\cap \tilde \Delta_{Q_i}} \psi_{R,j}(x)K(x,y)\phi_{Q,i}(y)d\mu(y)d\mu(x)\bigg|\\
&\lesssim C(\eta,\epsilon)\frac{\mu(Q_i\cap R_j)\mu(Q_i\cap R_j)}{\ell(Q_i)^d}\lesssim C(\eta,\epsilon)\mu(Q_i\cap R_j).
\end{split}
\end{equation}
Here we also used the facts that $\tilde\Delta_{R_j}\cup \tilde \Delta_{Q_i}\subset Q_i\cap R_j$
and $\mu(Q_i)\lesssim \ell(Q_i)^d$.

Then we consider the case $G=H$.  Using that boundary of $\Delta_{Q_i}^\mathcal{G}\cap \Delta_{R_j}^\mathcal{G}$ doesn't
intersect the interior of $G$, we see that
\begin{equation}\label{dsh}
\langle 1_G 1_{\tilde \Delta_{R_j}}\psi_{R,j},T(1_{H}1_{\tilde \Delta_{Q_i}}
\phi_{Q,i})\rangle
=\begin{cases}
\langle 1_{\tilde G} \psi_{R,j},T(1_{\tilde G}\phi_{Q,i})\rangle,\quad &\text{if }\tilde G\subset \Delta_{Q_i}^\mathcal{G}\cap \Delta_{R_j}^\mathcal{G};\\
0&\text{otherwise}.
\end{cases}
\end{equation}
In what follows, we assume that $\tilde G\subset \Delta_{Q_i}^\mathcal{G}\cap \Delta_{R_j}^\mathcal{G}$.

Consider the decomposition
\begin{equation}\label{hang}
\begin{split}
&\langle 1_{\tilde G} \psi_{R,j},T(1_{\tilde G}\phi_{Q,i})\rangle
=\langle 1_{\tilde G} \psi_{R,j},T(\phi_{Q,i})\rangle
-\langle 1_{\tilde G} \psi_{R,j},T(1_{\R^N\setminus 5\tilde G}\phi_{Q,i})\rangle\\
&\quad-\langle 1_{\tilde G} \psi_{R,j},T(1_{5\tilde G\setminus(1+\epsilon) \tilde G)}\phi_{Q,i})\rangle-\langle 1_{\tilde G} \psi_{R,j},T(1_{(1+\epsilon)\tilde G\setminus \tilde G}\phi_{Q,i})\rangle.
\end{split}
\end{equation}
The fourth term in the right hand side will only contribute to $A_3'(R)$.
The first and third terms in the right hand side are estimated as follows.

Using both \eqref{gcoa_l} and \eqref{fcoa_l} together with \eqref{accre} and \eqref{accreass}, we obtain the estimate
$\|T(\phi_{Q,i})\|_{L^\infty(\mu)}+\|\psi_{R,j}\|_{L^\infty(\mu)}\lesssim 1$. In particular,
\[
|\langle 1_{\tilde G} \psi_{R,j},T(\phi_{Q,i})\rangle|\lesssim \mu(\tilde G)\le \mu (Q_i\cap R_j).
\]
Then we consider the third term in the right hand side of \eqref{hang}. Since the interior of $G$ intersects $\Delta_{Q_i}^\mathcal{G}\cap \Delta_{R_j}^\mathcal{G}$, we see that $5\tilde G\subset 5G\subset Q_i\cap R_j$. For this reason we can repeat
the argument \eqref{kestim} which, in turn, gives
\[
|\langle 1_{\tilde G} \psi_{R,j},T(1_{5\tilde G\setminus(1+\epsilon) \tilde G)}\phi_{Q,i})\rangle|
\lesssim C(\eta,\epsilon)\mu(Q_i\cap R_j).
\]

It remains to consider the second term in the right hand side of \eqref{hang}. 
Part of it will contribute to $A_3(R)'$. We begin with
certain preparations, and first denote
\[\tau:=T(1_{\R^N\setminus 5\tilde G}\phi_{Q,i}).\]
Using \eqref{smooth}, reasoning as in \eqref{ssss}
with $R_1$ replaced by $5\tilde G$, and observing
the fact that  $\|\phi_{Q,i}\|_{L^\infty(\mu)}\lesssim 1$, 
we see that 
\begin{equation}\label{keres}
|\tau(x)-\tau(y)|\lesssim 1,\quad x,y\in \tilde G.
\end{equation}
{\em We use the fact that there exists an $L^\infty$-accretive system for $T^*$}.
Let $b_{\tilde G}$ be a function which is
supported on $\tilde G$, whose average over $\tilde G$ is one, and
\[\|b_{\tilde G}\|_{L^\infty(\mu)} + \|T^*(b_{\tilde G})\|_{L^\infty(\mu)} \lesssim 1.\]
Let
us denote $\beta_{\tilde G}=\langle b_{\tilde G}/\mu(\tilde G),\tau\rangle$. By
properties of $b_{\tilde G}$ and  \eqref{keres},
\begin{equation}\label{kerkay}
|\tau(x)-\beta_{\tilde G}|=|\langle b_{\tilde G}/\mu(\tilde G),\tau(x)-\tau\rangle|\lesssim 1,\quad x\in \tilde G.
\end{equation}

After these preparations, we write
\begin{align*}
\langle 1_{\tilde G} \psi_{R,j},T(1_{\R^N\setminus 5\tilde G}\phi_{Q,i})\rangle&
= \langle  1_{\tilde G} \psi_{R,j},\tau-\beta_{\tilde G}\rangle  +\langle  1_{\tilde G} \psi_{R,j},\beta_{\tilde G}\rangle.
\end{align*}
By \eqref{kerkay} and \eqref{gcoa_l}, we have the estimate $|\langle  1_{\tilde G} \psi_{R,j},\tau-\beta_{\tilde G}\rangle |\lesssim \mu(\tilde G)\le \mu(Q_i\cap R_j)$.
To treat the term $\langle  1_{\tilde G} \psi_{R,j},\beta_{\tilde G}\rangle=\beta_{\tilde G}\langle  1_{\tilde G} \psi_{R,j},1\rangle$, we write
\begin{equation}\label{montabe}
\begin{split}
\beta_{\tilde G}=\bigg\langle \frac{b_{\tilde G}}{\mu(\tilde G)},\tau\bigg\rangle&=\bigg\langle \frac{b_{\tilde G}}{\mu(\tilde G)},T(\phi_{Q,i})\bigg\rangle
-\bigg\langle \frac{b_{\tilde G}}{\mu(\tilde G)},T(1_{5\tilde G\setminus (1+\epsilon)\tilde G}\phi_{Q,i})\bigg\rangle
\\&\qquad-\bigg\langle \frac{b_{\tilde G}}{\mu(\tilde G)},T(1_{(1+\epsilon)\tilde G\setminus \tilde G}\phi_{Q,i})\bigg\rangle
-\bigg\langle \frac{b_{\tilde G}}{\mu(\tilde G)},T(1_{\tilde G}\phi_{Q,i})\bigg\rangle.
\end{split}
\end{equation}
Observe that the first and second term on the right hand side of \eqref{montabe} are
bounded in absolute value by a constant $C\lesssim C(\eta,\epsilon)$. 
This follows from
the properties of $b_{\tilde G}$ and the fact $\|T(\phi_{Q,i})\|_{L^\infty(\mu)}\lesssim 1$ for the first term
and,  by
reasoning as in \eqref{kestim}, for the second term.

For the last term in the right hand side of \eqref{montabe}, we use  $\|T^*(b_{\tilde G})\|_{L^\infty(\mu)} \lesssim 1$ for
\[
\bigg|\bigg\langle \frac{b_{\tilde G}}{\mu(\tilde G)},T(1_{\tilde G}\phi_{Q,i})\bigg\rangle\bigg|=|\langle T^*(b_{\tilde G}),1_{\tilde G}\phi_{Q,i})\rangle /\mu(\tilde G)|
\lesssim 1.
\]

Regrouping the terms shows that
$\langle 1_{\tilde G} \psi_{R,j},T(1_{\R^N\setminus 5\tilde G}\phi_{Q,i})\rangle$
can be expressed as a sum of two terms, the first one being
bounded in absolute value by a constant $C\lesssim C(\eta,\epsilon)\mu(Q_i\cap R_j)$, and
the second one being 
\[
-\langle b_{\tilde G},T(1_{(1+\epsilon)\tilde G\setminus \tilde G}\phi_{Q,i})\rangle\langle 1_{\tilde G}\psi_{R,j},1/\mu(\tilde G)\rangle.
\]

As said in the beginning of the proof, we expand $A_3(R)$ by using the cubes in $\mathcal{G}$ as follows:
\begin{align*}
A_3(R)&=\langle 1_{\tilde \Delta_{R_j}}\psi_{R,j},T(1_{\tilde \Delta_{Q_i}}\phi_{Q,i})\rangle
\\&= \sum_{\substack{G,H\in \mathcal{G}\\G\not=H}} \langle 1_G 1_{\tilde \Delta_{R_j}}\psi_{R,j},T(1_{H}1_{\tilde \Delta_{Q_i}}\phi_{Q,i})\rangle
+\sum_{G\in \mathcal{G}}
\langle  1_G 1_{\tilde \Delta_{R_j}}\psi_{R,j},T(1_{G}1_{\tilde \Delta_{Q_i}}\phi_{Q,i})\rangle.
\end{align*}
In both of the series above, 
the finite number of summands depends on
$N$ and $\eta$. Hence, using the estimates
above for a typical summand \eqref{typsum}, we get
\begin{align*}
A_3(R)&= A_3'(R)+(A_{3}(R)-A_3'(R)),
\end{align*}
where $|A_{3}(R)-A_3'(R)|\lesssim C(\eta,\epsilon)\mu(Q_i\cap R_j)$ and $A_3'(R)$ is the following quantity:
\begin{align*}
\sum_{\substack{G\in \mathcal{G}\\\tilde G\subset \Delta_{Q_i}^\mathcal{G}\cap \Delta_{R_j}^\mathcal{G}}} \Big[
\langle b_{\tilde G},T(1_{(1+\epsilon)\tilde G\setminus \tilde G}\phi_{Q,i})\rangle\langle 1_{\tilde G}\psi_{R,j},1/\mu(\tilde G)\rangle
-
\langle 1_{\tilde G}\psi_{R,j},T(1_{(1+\epsilon)\tilde G\setminus \tilde G}\phi_{Q,i})\rangle\Big].
\end{align*}
It is straightforward to verify that $A_3'(R)$ is of the required form.
\end{proof}

The summation involving the non-boundary terms $A_{3}(R)-A_{3}'(R)$, given by Lemma
\ref{talkaa}, is controlled by the following result. 
It gives a uniform estimate for the sum with respect to all systems of dyadic cubes $\mathcal{D}^\star$; in particular, no
expectations over $\mathcal{D}^\star$ are needed.

\begin{lem}\label{alpha4}
Let $f\in L^p(\R^N,\mu;X)$ and $g\in L^q(\R^N,\mu;X^*)$. Then estimate
\begin{align*}
\bigg|
\sum_{\substack{R\in\mathcal{D}'}} \langle g_{R}\rangle_{R_j} \big(A_3(R)-A_3'(R)\big) \langle f_Q\rangle_{Q_i}
\bigg|
\lesssim C(r,\eta,\epsilon)\|g\|_{q}\|f\|_{p}
\end{align*}
is valid for every dyadic system $\mathcal{D}^\star$.
\end{lem}

\begin{proof}
Let us denote $A_4(R):=A_{3}(R)-A_{3}'(R)$.
By randomizing and using H\"older's inequality,
\begin{equation}\label{estnormad}
\begin{split}
&\bigg|\sum_{R\in\mathcal{D}'} \langle g_R\rangle_{R_j}A_4(R)\langle f_Q\rangle_{Q_i}\bigg|\\
&=\bigg|\iint_{\Omega\times\R^N} \sum_{S\in\mathcal{D}'}
\epsilon_S  1_{S_j}(x)\langle g_S\rangle_{S_j}\sum_{R\in\mathcal{D}'}
\epsilon_R 1_{Q_i}(x)\frac{A_4(R)}{\mu(Q_i\cap R_j)}\langle f_Q\rangle_{Q_i}d\mathbf{P}(\epsilon)d\mu(x)\bigg|\\
&\le \bigg\| \sum_{S\in\mathcal{D}'} \epsilon_S 1_{S_j}\langle g_S\rangle_{S_j}\bigg\|_{L^{q}(\R^N,\mathbf{P}\otimes\mu;X^*)} \bigg\|
\sum_{R\in\mathcal{D}'}\epsilon_R
1_{Q_i}\frac{A_4(R)}{\mu(Q_i\cap R_j)}\langle f_Q\rangle_{Q_i}\bigg\|_{L^p(\R^N,\mathbf{P}\otimes \mu;X)}.
\end{split}
\end{equation}
Note that
\[
1_{S_j}\langle g_S\rangle_{S_j} = 1_{S_j} \langle g_k\rangle_{S_j} = 1_{S_j}ÊE_{k-1} g_k = 1_{S_j}Êg_k,\quad S\in\mathcal{D}'_k.
\]
Hence, by using also the contraction principle, we have
\[
\bigg\| \sum_{S\in\mathcal{D}'} \epsilon_S 1_{S_j}\langle g_S\rangle_{S_j}\bigg\|_{L^{q}(\R^N,\mathbf{P}\otimes\mu;X^*)}\lesssim \bigg\|\sum_{k=-\infty}^\infty \epsilon_k g_{k}\bigg\|_{L^q(\R^N,\mathbf{P}\otimes\mu;X^*)}
\lesssim \|g\|_{L^q(\R^N,\mu;X^*)}.
\]
In the last step we reasoned as in the end of Section \ref{nested}.

Rewrite the second summation in the last line of \eqref{estnormad} in terms of $\mathcal{D}$. Then using the contraction principle and 
the fact that $|A_4(R)|\lesssim C(r,\eta,\epsilon)\mu(Q_i\cap R_j)$, given by Lemma \ref{talkaa}, results in estimates
\begin{align*}
&\bigg\|\sum_{R\in\mathcal{D}'}\epsilon_R
1_{Q_i}\frac{A_4(R)}{\mu(Q_i\cap R_j)}\langle f_Q\rangle_{Q_i}\bigg\|_{L^p(\R^N,\mathbf{P}\otimes \mu;X)}\\
&\lesssim C(r,\eta,\epsilon)\bigg\|\sum_{Q\in\mathcal{D}}\epsilon_Q
1_{Q_i}\langle f_Q\rangle_{Q_i}\bigg\|_{L^p(\R^N,\mathbf{P}\otimes \mu;X)}
\lesssim C(r,\eta,\epsilon)\bigg\|\sum_{k=-\infty}^\infty \epsilon_k f_{k}\bigg\|_{L^p(\R^N,\mathbf{P}\otimes\mu;X)}.
\end{align*}
Reasoning as in the end of Section \ref{nested} finishes the proof.
\end{proof}

\subsection*{Estimate for $\epsilon$-boundary terms}
The following lemma controls summation involving the $\epsilon$-boundary terms 
\[A_1(R)+A_2(R)+A_{3}'(R).\]  Taking the
expectations over the dyadic system $\mathcal{D}^\star$ is invaluable here and, on the other hand, this is the only place where 
these expectations are required. Elsewhere we obtain uniform estimates over these
systems.

\begin{lem}\label{aep}
Suppose that $s\in [2,\infty)$ is such that both $X$ and $X^*$ have cotype $s$. Let $t>(s\vee q)\vee p$ be a positive real number. Then
\begin{align*}
&\mathbf{E}_{\mathcal{D}^\star} \bigg|
\sum_{\substack{R\in\mathcal{D}'}} \langle g_{R}\rangle_{R_j} (A_1(R)+A_2(R)+A_{3}'(R)) \langle f_Q\rangle_{Q_i}
\bigg|\\&\qquad\qquad
\lesssim C(r,\eta)\epsilon^{1/t} \|T\|_{\mathcal{L}(L^p(\mu;X))}\|g\|_{L^q(\mu;X^*)}\|f\|_{L^p(\mu;X)}
\end{align*}
for every $f\in L^p(\R^N,\mu;X)$ and $q\in L^q(\R^N,\mu;X^*)$.
\end{lem}

\begin{proof}
First we focus on the sum involving the terms $A_1$; these are
defined in \eqref{alpha1}. Randomize and use H\"older's inequality for the estimate
\begin{equation}\label{ranes}
\begin{split}
& \bigg|
\sum_{\substack{R\in\mathcal{D}'}} \langle g_{R}\rangle_{R_j} \langle 1_{\Delta'_{R_j}}\psi_{R,j},T(1_{\Delta_{Q_i}^\mathcal{G}}\phi_{Q,i})\rangle\langle f_Q\rangle_{Q_i}\bigg|\\
&=\bigg|\int_\Omega \Big\langle \sum_{S\in\mathcal{D}'} \epsilon_S
1_{\Delta'_{S_j}}\psi_{S,j}\langle g_{S}\rangle_{S_j},T\Big(\sum_{R\in\mathcal{D}'}\epsilon_R1_{\Delta_{Q_i}^\mathcal{G}}\phi_{Q,i}\langle f_Q\rangle_{Q_i} \Big)
 \Big\rangle d\mathbf{P}(\epsilon)\bigg|\\
&\le \bigg\| \sum_{S\in\mathcal{D}'} \epsilon_S1_{\Delta'_{S_j}}\psi_{S,j}\langle g_{S}\rangle_{S_j}\bigg\|_{L^{q}(\mathbf{P}\otimes\mu;X^*)} \bigg\|T\Big(\sum_{R\in\mathcal{D}'}\epsilon_R1_{\Delta_{Q_i}^\mathcal{G}}\phi_{Q,i}\langle f_Q\rangle_{Q_i}\Big)\bigg\|_{L^p(\mathbf{P}\otimes \mu;X)}.
\end{split}
\end{equation}
First extract the operator norm from the second factor
and index the summation in terms of $\mathcal{D}$. Then, by
using the contraction principle and estimate \[
|1_{\Delta_{Q_i}^\mathcal{G}}\phi_{Q,i}| \le 1_{Q_i}\] which is valid $\mu$-almost everywhere, we see that
the second factor in the last line of \eqref{ranes} is bounded (up to a constant multiple) by
\begin{align*} 
\|T\|_{\mathcal{L}(L^p(\mu;X))}\bigg\|\sum_{k=-\infty}^\infty \epsilon_k f_{k}\bigg\|_{L^p(\mathbf{P}\otimes\mu;X)}\lesssim \|T\|_{\mathcal{L}(L^p(\mu;X))}\|f\|_{L^p(\mu;X)}.
\end{align*}

In order to estimate the first factor in the last line of \eqref{ranes} we let
$S\in\mathcal{D}'_{k}$, where $k\in\Z$.
Due to 
\eqref{rems} and \eqref{sl}, we have 
\[
\Delta_{S_j}'\subset G_\epsilon(S)=\bigcup_{G\in \mathcal{G}} \delta_G^\epsilon\subset 
\bigcup_{m=j(\eta)+k-r-1}^{j(\eta)+k-1} \bigcup_{G\in\mathcal{D}^\star_m}  \delta^\epsilon_G=:
\delta^\epsilon(k).
\] 
As a consequence,  we have $|1_{\Delta'_{S_j}}\psi_{S,j}|\le 1_{\delta^\epsilon(k)}1_{S_j}$ pointwise
$\mu$-almost everywhere.
Using also the contraction principle and the assumption that $t\ge q$, we get
\begin{align*}
&\mathbf{E}_{\mathcal{D}^\star}\bigg\| \sum_{S\in\mathcal{D}'} \epsilon_S1_{\Delta'_{S_j}}\psi_{S,j}\langle g_{S}\rangle_{S_j}\bigg\|_{L^{q}(\mathbf{P}\otimes\mu;X^*)}
\\&\quad\lesssim
\mathbf{E}_{\mathcal{D}^\star}\bigg\|\sum_{k\in\Z} \epsilon_k1_{\delta^\epsilon(k)}\sum_{S\in\mathcal{D}'_{k}} 1_{S_j}\langle g_{S}\rangle_{S_j}\bigg\|_{L^{q}(\mathbf{P}\otimes\mu;X^*)}\\
&\quad\le \bigg(\int_{\R^N} \bigg[\mathbf{E}_{\mathcal{D}^\star}\bigg\|\sum_{k\in\Z} \epsilon_k1_{\delta^\epsilon(k)}(x)\sum_{S\in\mathcal{D}'_{k}} 1_{S_j}(x)\langle g_{S}\rangle_{S_j}\bigg\|^t_{L^{q}(\mathbf{P};X^*)}\bigg]^{q/t}
d\mu(x)\bigg)^{1/q}.
\end{align*}
If $x\in\R^N$, the last integrand evaluated at $x$ is of
the form as in Proposition \ref{improved} with
\[
\xi_k = \sum_{S\in\mathcal{D}'_{k}} 1_{S_j}(x)\langle g_{S}\rangle_{S_j}\in X^*.
\]
The random variables \[\rho_k:=1_{\delta^\epsilon(k)}(x)\] as functions of
$\beta^\star \in \Omega^\star$, where $\Omega^\star$ is the probability
space supporting the distribution of the random dyadic system
$\mathcal{D}^\star$, belong to $L^t(\Omega^\star)$,
and they satisfy
\[
\sup_{k\in\Z}\|1_{\delta^\epsilon(k)}(x)\|_{L^t(\Omega^\star)} = \sup_{k\in\Z}\mathbf{P}_{\beta^\star} (1_{\delta^\epsilon(k)}(x)=1)^{1/t}\lesssim C(r,\eta)\epsilon^{1/t}.
\]
Hence, by Proposition \ref{improved}
\begin{align*}
&\mathbf{E}_{\mathcal{D}^\star}\bigg\| \sum_{S\in\mathcal{D}'} \epsilon_S1_{\Delta'_{S_j}}\psi_{S,j}\langle g_{S}\rangle_{S_j}\bigg\|_{L^{q}(\mathbf{P}\otimes\mu;X^*)}\\&\quad\lesssim C(r,\eta)
\epsilon^{1/t}\bigg\| \sum_{S\in\mathcal{D}'} \epsilon_S 1_{S_j}\langle g_{S}\rangle_{S_j} \bigg\|_{L^{q}(\mathbf{P}\otimes\mu;X^*)}\lesssim 
C(r,\eta)\epsilon^{1/t}\|g\|_{L^q(\mu;X^*)}.
\end{align*}
Combining the estimates above, we obtain the required estimate for
summation involving terms $A_1(R)$.

Estimate for the sum involving terms $A_2(R)$, see \eqref{alpha1}, is similar to the estimate above, involving terms $A_1(R)$. We omit the details.

It remains to estimate the following sum involving terms $A_{3}'(R)$, see Lemma \ref{talkaa},
\begin{align*}
&\mathbf{E}_{\mathcal{D}^\star} \bigg|
\sum_{\substack{R\in\mathcal{D}'}}  \sum_{\substack{G\in \mathcal{G}(R)\\\tilde G\subset \Delta_{Q_i}^\mathcal{G}\cap \Delta_{R_j}^\mathcal{G}}} \langle g_{R}\rangle_{R_j}\langle 1_{\tilde G}b_{R,G,j},T(1_{(1+\epsilon)\tilde G\setminus \tilde G}\phi_{Q,i})\rangle\langle f_Q\rangle_{Q_i}\bigg|.
\end{align*}
Observe that the inner summation
involves only finitely many
terms for every fixed $R$ -- in fact, the number of terms is bounded by a constant depending
on $\eta$ and $N$. Hence, by reindexing these cubes and using the triangle-inequality,
we are left with estimating quantities of the form
\begin{align*}
E:=&\mathbf{E}_{\mathcal{D}^\star} \bigg|
\sum_{\substack{R\in\mathcal{D}'}} \langle g_{R}\rangle_{R_j}  \langle 1_{\tilde G}b_{R,G,j},T(1_{(1+\epsilon)\tilde G\setminus \tilde G}\phi_{Q,i})\rangle  \langle f_Q\rangle_{Q_i}\bigg|,
\end{align*}
where $G=G(R)\in \mathcal{G}(R)$ inside the summation satisfies $\tilde G\subset \Delta_{Q_i}^\mathcal{G}\cap \Delta_{R_j}^\mathcal{G}$. 

At this stage we randomize, apply H\"older's inequality, and extract the operator norm in order to obtain the estimate
\begin{equation}\label{lls}
\begin{split}
E&\le \|T\|_{\mathcal{L}(L^p(\mu;X))}\mathbf{E}_{\mathcal{D}^\star}\bigg\{\bigg\| \sum_{S\in\mathcal{D}'} \epsilon_S 1_{\tilde G}b_{R,G,j}\langle g_{S}\rangle_{S_j}\bigg\|_{L^{q}(\mathbf{P}\otimes\mu;X^*)} \\&\qquad\qquad\qquad\qquad\qquad\cdot \bigg\|\sum_{R\in\mathcal{D}'}\epsilon_R1_{(1+\epsilon)\tilde G\setminus \tilde G}\phi_{Q,i}\langle f_Q\rangle_{Q_i}\bigg\|_{L^p(\mathbf{P}\otimes \mu;X)}\bigg\}.
\end{split}
\end{equation}
By lemma \ref{talkaa}, 
\[
|1_{\tilde G(S)}b_{R,G,j}|\lesssim 1_{\tilde G(S)} \le 1_{\Delta_{S_j}^\mathcal{G}}\le 1_{S_j}
\]
pointwise $\mu$-almost everywhere.
Also,
$(1+\epsilon)\tilde G(R)\setminus \tilde G(R)\subset 5 G(R)\subset Q_i$ and
\begin{align*}
&(1+\epsilon)\tilde G(R)\setminus \tilde G(R)
\subset G_\epsilon(R)
\subset \delta^\epsilon(k),\qquad Q=Q(R)\in\mathcal{D}_k.
\end{align*}
It follows that
$|1_{(1+\epsilon)\tilde G\setminus \tilde G}\phi_{Q,i}|\lesssim 1_{\delta^\epsilon(k)}1_{Q_i}$
$\mu$-almost everywhere if $Q\in\mathcal{D}_k$.
Hence, by indexing the second summation in the right hand side of \eqref{lls}
in terms of $\mathcal{D}$,
the argument proceeds as above. We omit the details.
\end{proof}

\subsection*{Estimate for $\eta$-boundary terms}
Here we focus on a summation involving the $\eta$-boundary terms
\[
M_2(R)+M_4(R)+\alpha_2(R)+\alpha_3(R),
\]
see \eqref{tenf} and \eqref{ekahaj}.
Observe that  although
both $\alpha_2(R)$ and $\alpha_3(R)$ depend 
on the random dyadic system $\mathcal{D}^\star$, the estimate below
are uniform over all such systems.

\begin{lem}\label{etalem}
Suppose that $s\in [2,\infty)$ is such that both $X$ and $X^*$ have cotype $s$. Let $t>(s\vee q)\vee p$ be a positive real number. Then
\begin{align*}
&\mathbf{E}_{\mathcal{D}}\mathbf{E}_{\mathcal{D}'}\bigg|
\sum_{\substack{R\in\mathcal{D}'}} \langle g_{R}\rangle_{R_j} (M_2(R)+M_4(R)+\alpha_2(R)+\alpha_3(R)) \langle f_Q\rangle_{Q_i}
\bigg|\\&\quad
\lesssim C(r)\eta^{1/t} \|T\|_{\mathcal{L}(L^p(\mu;X))}\|g\|_{L^q(\mu;X^*)}\|f\|_{L^p(\mu;X)}
\end{align*}
for every $f\in L^p(\R^N,\mu;X)$ and $g\in L^q(\R^N,\mu;X^*)$.
\end{lem}

\begin{proof}
By \eqref{tenf} and \eqref{ekahaj},
\begin{align*}
M_2(R)+\alpha_2(R)& =\langle 1_{R_{j,\partial}}\psi_{R,j},T(1_{Q_i}\phi_{Q,i})\rangle
-\langle 1_{\Delta_{R_j}^\partial}\psi_{R,j},T(1_{\Delta_{Q_i}^\mathcal{G}}\phi_{Q,i})\rangle;\\
M_4(R)+\alpha_3(R)&=\langle 1_{\Delta_{R_j}} \psi_{R,j},T(1_{Q_{i,\partial}}\phi_{Q,i})\rangle-\langle 1_{\Delta_{R_j}}\psi_{R,j},T(1_{\Delta_{Q_i}^\partial}\phi_{Q,i})\rangle.
\end{align*}
Observe that
\begin{equation}\label{ekat}
\begin{split}
&|1_{R_{j,\partial}}\psi_{R,j}|+|1_{\Delta_{R_j}^\partial}\psi_{R,j}| \lesssim 1_{R_{j,\partial}},\quad |1_{Q_i}\phi_{Q,i}|+|1_{\Delta_{Q_i}^\mathcal{G}}\phi_{Q,i}|\lesssim 1_{Q_i};\\
&|1_{\Delta_{R_j}}\psi_{R,j}|\lesssim 1_{R_j},\,\,\,\quad |1_{Q_{i,\partial}}\phi_{Q,i}|+|1_{\Delta_{Q_i}^\partial}\phi_{Q,i}|\lesssim 1_{Q_{i,\partial}}.
\end{split}
\end{equation}
pointwise $\mu$-almost everywhere.
By triangle inequality, it suffices
to estimate the following sums: one involving terms $m(R)\in\{M_2(R),\alpha_2(R)\}$,
and
the other involving terms in $\{M_4(R),\alpha_3(R)\}$. We
focus on the first sum; the second one is estimated
in an analogous manner.

Randomizing, using H\"older's inequality, extracting the operator norm of $T$,
and finally applying the contraction principle with \eqref{ekat} results in the estimate
\begin{equation}\label{eeka}
\begin{split}
\mathbf{E}_{\mathcal{D}}\bigg|
\sum_{\substack{R\in\mathcal{D}'}} \langle g_{R}\rangle_{R_j} m(R) \langle f_Q\rangle_{Q_i}
\bigg|\lesssim &\mathbf{E}_{\mathcal{D}}\bigg\| \sum_{S\in\mathcal{D}'} \epsilon_S1_{S_{j,\partial}} \langle g_{S}\rangle_{S_j}\bigg\|_{L^{q}(\mathbf{P}\otimes\mu;X^*)} \\&\cdot \|T\|_{\mathcal{L}(L^p(\mu;X))}\bigg\|\sum_{R\in\mathcal{D}'}\epsilon_R1_{Q_i}\langle f_Q\rangle_{Q_i}\bigg\|_{L^p(\mathbf{P}\otimes \mu;X)}.
\end{split}
\end{equation}
Indexing the summation in terms of $\mathcal{D}$ and using the contraction
principle, we see that the last factor is bounded $\|f\|_{L^p(\mu;X)}.$
For the first factor in the right hand side of \eqref{eeka}, we write
\[
\delta^\eta(k)=\bigcup_{m=k-r-1}^{k-1} \bigcup _{Q\in\mathcal{D}_m} \delta^\eta_Q.
\]
By \eqref{rems}, we have
\[
1_{S_{j,\partial}}\le 1_{S_j}1_{\delta_{Q_i}^\eta}\le 1_{S_j}1_{\delta^\eta(k)},\quad \text{ if }Q=Q(S),\,S\in\mathcal{D}'_{k}.
\]
Fix $x\in\R^N$. The random variables $\rho_k:=1_{\delta^\eta(k)}(x)$ as functions of
$\beta\in (\{0,1\}^N)^\Z$, $\mathcal{D}=\mathcal{D}(\beta)$, belong to $L^t((\{0,1\}^N)^\Z)$
and they satisfy
\[
\sup_{k\in\Z} \|1_{\delta^\eta(k)}(x)\|_{L^t((\{0,1\}^N)^\Z)} = \sup_{k\in\Z}\mathbf{P}_{\beta} (1_{\delta^\eta(k)}(x)=1)^{1/t}\lesssim C(r)\eta^{1/t}.
\]
Hence, proceeding as in the proof of Lemma \ref{aep}, we find that
\begin{align*}
\mathbf{E}_{\mathcal{D}}\bigg\| \sum_{S\in\mathcal{D}'} \epsilon_S1_{S_{j,\partial}}\langle g_{S}\rangle_{S_j}\bigg\|_{L^{q}(\mathbf{P}\otimes\mu;X^*)}\lesssim 
C(r)\eta^{1/t}\bigg\|  \sum_{R\in\mathcal{D}'} \epsilon_R1_{R_j}\langle g_{R}\rangle_{R_j} \bigg\|_{L^{q}(\mathbf{P}\otimes\mu;X^*)}.
\end{align*}
Noticing that the last term is bounded by a constant multiple of
$C(r)\eta^{1/t}\|g\|_{L^q(\mu;X^*)}$ finishes the proof.
\end{proof}

%

\section{Synthesis}\label{synthesis}

The proof of Theorem \ref{mainth} will be completed.  
This involves
choosing appropriate values for the auxiliary parameters $r,\eta,\epsilon$.
Hence, any dependence on these numbers will be indicated explicitly.

\begin{proof}[Proof of Theorem \ref{mainth}]
Let us fix $f\in L^p(\mu;X)$ and $g\in L^q(\mu;X^*)$ such
that 
\[
\|T\|_{\mathcal{L}(L^p(\mu;X))}\le 2|\langle g, Tf\rangle|,\qquad \|f\|_p = 1=\|g\|_q.
\]
Taking expectations over estimate \eqref{firstes} gives us
\[
|\langle g,Tf\rangle|\lesssim\|g\|_q\|f\|_p + \mathbf{E}_{\mathcal{D}}\mathbf{E}_{\mathcal{D}'}\bigg|\sum_{Q\in\mathcal{D},\,R\in\mathcal{D}'} \langle D_R^{a,2} g,T(D_Q^{a,1} f)\rangle\bigg|.
\]
Because $X$ is a $\mathrm{UMD}$ function lattice,
its dual $X^*$ is also a $\mathrm{UMD}$ function lattice. Hence, by symmetry, it suffices to consider the summation
over dyadic cubes $Q$ and $R$ for which $\ell(Q)\le \ell(R)$.

We decompose this series further as follows:
\begin{equation}\label{gbdec}
\sum_{R\in\mathcal{D}'}\sum_{\substack{Q\in\mathcal{D}\\ \ell(Q)\le \ell(R)}}
=
\sum_{R\in\mathcal{D'}}\sum_{\substack{Q\in\mathcal{D}_{R\text{-good}}\\\ell(Q)\le \ell(R)}}+
\sum_{R\in\mathcal{D'}}\sum_{\substack{Q\in\mathcal{D}_{R\text{-bad}}\\ \ell(Q)\le \ell(R)}}.
\end{equation}
Observe that this decomposition to good and bad parts depends
on  $\mathcal{D}'=\mathcal{D}(\beta')$. 

Let us first focus on  the good summation in the right hand side
of \eqref{gbdec}. 
We  denote
$Q\sim R$ if these cubes satisfy \eqref{rems}, that is, if they are close to each other both in position and size. Then
we  have the decomposition
\begin{equation}\label{hajo2}
\begin{split}
\sum_{R\in\mathcal{D}'} \sum_{\substack{Q\in\mathcal{D}_{R\text{-good}}\\\ell(Q)\le \ell(R)}} 
=&\sum_{R\in\mathcal{D}'} \sum_{\substack{Q\in\mathcal{D}_{R\text{-good}}\\ Q\sim R}} + 
\sum_{R\in\mathcal{D}'} \sum_{\substack{Q\in\mathcal{D}_{R\text{-good}}\\   Q\subset R \\ \ell(Q)<2^{-r}\ell(R) }}\\
&+\sum_{R\in\mathcal{D}'} \sum_{\substack{Q\in\mathcal{D}_{R\text{-good}}\\ Q\not\subset R \\ \ell(Q)<2^{-r}\ell(R)}} + 
\sum_{R\in\mathcal{D}'} \sum_{\substack{Q\in\mathcal{D}_{R\text{-good}} \\ 2^{-r}\ell(R)\le \ell(Q)\le  \ell(R) \\ \ell(Q)\le \dist(Q,R)}}.
\end{split}
\end{equation}

Let us consider the third double series on the right hand side further. Assume 
that $R\in\mathcal{D}'$ and $Q\in\mathcal{D}_{R\text{-good}}$ satisfies $Q\not\subset R$
and $\ell(Q)<2^{-r}\ell(R)$. Remark \ref{new_hyvmeas} implies that
$\dist(Q,R)=\dist(Q,\partial R) > \ell(Q)^\gamma\ell(R)^{1-\gamma}\ge \ell(Q)$.
As a consequence, we can write
\[
\sum_{R\in\mathcal{D}'} \sum_{\substack{Q\in\mathcal{D}_{R\text{-good}}\\Q\not\subset R \\ \ell(Q)<2^{-r}\ell(R)}} 
=
\sum_{R\in\mathcal{D}'} \sum_{\substack{Q\in\mathcal{D}_{R\text{-good}} \\ \ell(Q)<2^{-r}\ell(R) \\ \ell(Q)\le \dist(Q,R)}}.
\]
Hence, by combining 3rd and 4th term on the right hand side of \eqref{hajo2}, we obtain the 
identity
\begin{equation*}
\begin{split}
\sum_{R\in\mathcal{D}'} \sum_{\substack{Q\in\mathcal{D}_{R\text{-good}}\\\ell(Q)\le \ell(R)}} 
=&\sum_{R\in\mathcal{D}'} \sum_{\substack{Q\in\mathcal{D}_{R\text{-good}}\\ Q\sim R}} + 
\sum_{R\in\mathcal{D}'} \sum_{\substack{Q\in\mathcal{D}_{R\text{-good}}\\   Q\subset R \\ \ell(Q)<2^{-r}\ell(R) }}
+ 
\sum_{R\in\mathcal{D}'} \sum_{\substack{Q\in\mathcal{D}_{R\text{-good}}  \\ \ell(Q)\le \ell(R)\wedge \dist(Q,R)}}.
\end{split}
\end{equation*}
Invoking Propositions
\ref{ktest}, \ref{nestedlem}, and \ref{complem} we are able to estimate 
all of the summands above, and we reach the estimate
\begin{align*}
&
\mathbf{E}_{\mathcal{D}}\mathbf{E}_{\mathcal{D}'}
\bigg|\sum_{R\in\mathcal{D}'} \sum_{\substack{Q\in\mathcal{D}_{R\text{-good}}\\\ell(Q)\le \ell(R)}} \langle D_R^{a,2} g,T(D_Q^{a,1} f)
\rangle\bigg|\\&
\le
C(r,\eta,\epsilon)
+(C(r,\eta)\epsilon^{1/t} + C(r)\eta^{1/t})\|T\|_{\mathcal{L}(L^p(\mu;X))}.
\end{align*}

Then we concentrate on the remaining bad summation in the right hand side of \eqref{gbdec}.
By randomizing, using H\"older's inequality, and using Theorem \ref{nests}
with the identity $\|g\|_{L^q(X^*)}=1$, we get
\begin{equation}\label{tabad}
\begin{split}
&\mathbf{E}_{\mathcal{D}}\mathbf{E}_{\mathcal{D}'}\bigg|\sum_{R\in\mathcal{D}'}
\sum_{\substack{Q\in\mathcal{D}_{R\text{-bad}}\\ \ell(Q)\le \ell(R)}} \langle D_R^{a,2} g,T(D_Q^{a,1} f)
\rangle\bigg|\\
&= 
\mathbf{E}_{\mathcal{D}}\mathbf{E}_{\mathcal{D}'}\bigg|\sum_{k=0}^\infty\sum_{j\in\Z}
\sum_{R\in\mathcal{D}_j'} \sum_{Q\in\mathcal{D}_{j-k}^{(j-(j-k)-1)\text{-bad}}}
 \langle D_R^{a,2} g,T(D_Q^{a,1} f)\rangle\bigg|
\\
&=\mathbf{E}_{\mathcal{D}}\mathbf{E}_{\mathcal{D}'}\bigg|\sum_{k=0}^\infty
\int_{\Omega}
\sum_{j\in\Z}\sum_{i\in\Z}Ê\epsilon_j\epsilon_i
\sum_{R\in\mathcal{D}_j'} \sum_{Q\in\mathcal{D}_{i-k}^{(k-1)\text{-bad}}}
 \langle D_R^{a,2} g,T(D_Q^{a,1} f)\rangle d\mathbf{P}(\epsilon)\bigg|
\\
&\lesssim \sum_{k=0}^\infty\mathbf{E}_{\mathcal{D}}\mathbf{E}_{\mathcal{D}'}
\bigg\|T\bigg(\sum_{i\in\Z} \epsilon_i \sum_{Q\in\mathcal{D}_{i-k}^{(k-1)\text{-bad}}}D_Q^{a,1} f\bigg)\bigg\|_{L^p(\mathbf{P}\otimes\mu;X)}.
\end{split}
\end{equation}
In order to estimate this series, we fix $k\ge 0$.
Extracting the operator norm, we see that
the $k$'th summand is bounded by
\begin{align*}
\|T\|_{\mathcal{L}(L^p(X))}\cdot \mathbf{E}_{\mathcal{D}}\mathbf{E}_{\mathcal{D}'}
\bigg\|\sum_{i\in\Z} \epsilon_i \lambda_{\mathrm{bad},i}^kD_{i-k}^{a,1} f\bigg\|_{L^p(\mathbf{P}\otimes\mu;X)},
\end{align*}
where we have denoted
\[
\lambda_{\mathrm{bad},i}^k:=\sum_{Q\in\mathcal{D}_{i-k}^{(k-1)\text{-bad}}} 1_Q\in L^1(\R^N,\mu;\R).\]

Fix $t> (s\vee p)\vee q$, where $s$ is  such that both $X$ and $X^*$ have cotype $s\in [2,\infty)$.
Using Proposition \ref{improved}, we get the estimate
\begin{align*}
&\mathbf{E}_{\mathcal{D}}\mathbf{E}_{\mathcal{D}'}
\bigg\|\sum_{i\in\Z} \epsilon_i \lambda_{\mathrm{bad},i}^k D_{i-k}^{a,1} f\bigg\|_{L^p(\mathbf{P}\otimes\mu,X)}\\
&\le 
\mathbf{E}_{\mathcal{D}}\bigg(\int_{\R^N} \bigg[ \mathbf{E}_{\mathcal{D}'}\bigg\|\sum_{i\in\Z} \epsilon_i \lambda_{\mathrm{bad},i}^k(x)D_{i-k}^{a,1} f(x)\bigg\|^t_{L^{p}(\mathbf{P};X)}\bigg]^{p/t}
d\mu(x)\bigg)^{1/p}\\
&\lesssim \sup_{i,x} \|\lambda_{\mathrm{bad},i}^k(x)\|_{L^t(\mathbf{P}_{\beta'};\R)}
\bigg\|\sum_{i\in\Z}\epsilon_i D_{i-k}^{a,1} f\bigg\|_{L^p(\mathbf{P}\otimes\mu;X)}.
\end{align*}
Note that, by using Theorem \ref{nests}, we have the estimate
\[
\bigg\|\sum_{i\in\Z}\epsilon_i D_{i-k}^{a,1} f\bigg\|_{L^p(\mathbf{P}\otimes\mu;X)}
\lesssim 1.
\]
On the other hand, if $x\in\R^N$, then by Lemma \ref{nmeas} we have
\begin{align*}
&\sup_{i} \|\lambda_{\mathrm{bad},i}^k (x)\|_{L^t(\mathbf{P}_{\beta'};\R)}
\\&= \sup_{i}\big\{\mathbf{P}_{\beta'}[x\in Q\in\mathcal{D}_{i-k}\text{ and }Q\text{ is }(k-1)\text{-bad}(\gamma,r)]^{1/t}\big\}\lesssim 2^{-(r\vee (k-1))\gamma/t},
\end{align*}

All in all, we have established the following estimate
\begin{align*}
&\mathbf{E}_{\mathcal{D}}\mathbf{E}_{\mathcal{D}'}\bigg|\sum_{R\in\mathcal{D}'}
\sum_{\substack{Q\in\mathcal{D}_{R\text{-bad}}\\ \ell(Q)\le \ell(R)}} \langle D_R^{a,2} g,T(D_Q^{a,1} f)
\rangle\bigg|
\lesssim \|T\|_{\mathcal{L}(L^p(\mu;X))}\sum_{k=0}^\infty 2^{-(r\vee (k-1))\gamma/t}
\\&\lesssim r 2^{-r\gamma/t} \|T\|_{\mathcal{L}(L^p(\mu;X))}=\delta(r)\|T\|_{\mathcal{L}(L^p(\mu;X))}.
\end{align*}
Here $\delta(r)\to 0$ as $r\to \infty$.

Collecting the estimates above, we find that
\begin{equation}
\|T\|_{\mathcal{L}(L^p(\mu;X))} \le C(r,\eta,\epsilon)+(C\delta(r)+C(r)\eta^{1/t}+C(r,\eta)\epsilon^{1/t})\|T\|_{\mathcal{L}(L^p(\mu;X))}.
\end{equation}
Next we choose $r$ so large hat $C\delta(r)<1/4$. Then we choose
$\eta>0$ so small that $C(r)\eta^{1/t}<1/4$. Lastly we choose
$\epsilon>0$ so small that $C(r,\eta)\epsilon^{1/t}<1/4$. This results in the desired estimate
\[
\|T\|_{\mathcal{L}(L^p(\mu;X))} \le C(r,\eta,\epsilon)+\frac{3}{4}\|T\|_{\mathcal{L}(L^p(\mu;X))}.
\]
Indeed, it follows that $\|T\|_{\mathcal{L}(L^p(\mu;X))} \le 4C(r,\eta,\epsilon)$.
\end{proof}

\section{Operator-valued kernels}\label{oper_kernels}

In this section we explain the proof of Theorem \ref{mainth_operator}. This
proof is a straightforward modification of the proof of Theorem \ref{mainth}.

We define a {\em $d$-dimensional Rademacher--Calder\'on--Zygmund kernel}
as a function $K(x,y)$ of variables
$x,y\in \R^N$ with $x\not=y$ and taking values in $\mathcal{L}(X)$, which
satisfies 
\begin{equation}\label{kernel_R_bounds}
\begin{split}
&\mathcal{R}\big(\{|x-y|^d K(x,y)\,:\,x,y\in\R^N,\,x\not=y\}\big)\le 1;\\
&\mathcal{R}\bigg(\bigg\{\frac{|x-y|^{d+\alpha}}{|x-x'|^\alpha} [K(x,y)-K(x',y)],
\frac{|x-y|^{d+\alpha}}{|x-x'|^\alpha} [K(y,x)-K(y,x')]\\
&\qquad\qquad\qquad\qquad :\,x,x',y\in\R^N,\,0<|x-x'|\le |x-y|/2\bigg\}\bigg)\le 1
\end{split}
\end{equation}
for some $\alpha>0$. 
Recall that $\mathcal{R}(\mathcal{T})$ designates the Rademacher-bound 
of an operator family $\mathcal{T}\subset \mathcal{L}(X)$, as defined after \eqref{r_bounded}.

Let $T:f\mapsto Tf$ be a linear operator acting on some
functions $f:\R^N\to X$ or
$f:\R^N\to \C$, producing new functions
$Tf:\R^N\to X$ in the former case
and $Tf:\R^N\to \mathcal{L}(X)$ in the latter. If $\xi\in X$ and
$F:\R^N\to \C$ or
$F:\R^N\to \mathcal{L}(X)$, define
the function
$F\otimes \xi$ by $(F\otimes \xi)(x):=F(x)\xi$, where the
last expression is the product of a scalar and a vector, or the action of
an operator on a vector, respectively. With this notation, suppose that
$T(\phi\otimes \xi)=(T\phi)\otimes \xi$ for $\phi:\R^N\to \C$
and $\xi\in X$. The adjoint $T^*$ is defined via duality
$\langle g,f\rangle = \int_{\R^N} \langle g(x),f(x)\rangle\,d\mu(x)$
between functions $f:\R^N\to X$ and $g:\R^N\to X^*$: for
$\phi,\psi:\R^N\to \C$, $\xi\in X$ and
$\xi^*\in X^*$,
\[
\xi^*\big(\langle \psi,T\phi\rangle \xi\big) =\langle \psi\otimes \xi^*,T(\phi\otimes \xi)\rangle
:= \langle T^*(\psi\otimes \xi^*),\phi\otimes \xi\rangle =:\big(\langle T^*\psi,\phi\rangle \xi^*\big)(\xi),
\]
and hence $\langle T^*\psi,\phi\rangle = (\langle \psi,T\phi))^*\in\mathcal{L}(X^*)$
for scalar-valued functions $\phi,\psi$.

Such a $T$ is called an {\em $\mathcal{L}(X)$-valued Rademacher--Calder\'on--Zygmund operator}
with kernel $K$ if
\[
Tf(x) = \int_{\R^N} K(x,y) f(y)\,d\mu(y)
\] for points $x\in\R^N$ outside the support of $f$.
%

We are ready to explain the modifications in the proof of Theorem \ref{mainth}.
These occur in sections \ref{dec_cald}--\ref{synthesis} and,
roughly speaking, they 
are as follows: one repeats
the proof, and the assumed $\mathcal{R}$-boundedness conditions ensure
that whenever one ``pulled out'' bounded scalar coefficients from the randomized
series, which persist throughout the arguments, the same can be done with
the operator coefficients by the very definition \eqref{r_bounded} of $\mathcal{R}$-boundedness.
Some technicalities arise when treating the paraproducts in
Section \ref{paraproducts}.
Our goal is to provide a comprehensive treatment and, at the same time, avoid repeating arguments. To accomplish this task, we 
have chosen to explain the modifications in sections \ref{decoul}, \ref{separated} and \ref{paraproducts}.

\subsection*{Operator-valued decoupling estimates}
Let us begin with Section \ref{decoul}. 
Instead of scalars satisfying \eqref{ades}, we now consider the following
$\mathcal{R}$-bounded 
families of operators:
\begin{equation}\label{ades_operator}
\mathcal{R}\big(\{\lambda_{RQ}T_{RQ} \in \mathcal{L}(X)\,:\,R\in\mathcal{D}',\,Q\in\mathcal{D}_{R\rm{-good}},\,\ell(Q)\le \ell(R)\}\big)\lesssim 1,
\end{equation}
where 
$T_{RQ}\in\mathcal{L}(X)$ and
the scalar coefficients are
\[
\lambda_{RQ} := \frac{D(Q,R)^{d+\alpha}}{\mu(R_u)\mu(Q_v)\ell(Q)^{\alpha/2}\ell(R)^{\alpha/2}}.
\]
These $\mathcal{R}$-bounded families occur in the
following counterpart of Lemma \ref{tut}.

\begin{lem}\label{tut_operator}
Assume that $E_{k-1}f_k=f_k$ and $E_{k-1}g_k=g_k$, where
$f_k$ and $g_k$, $k\in\Z$, are as in Lemma \ref{tut}. Assume also
that the estimate \eqref{ades_operator} holds. Then
\begin{equation}\label{treq_operator}
\begin{split}
&\bigg|\sum_{R\in\mathcal{D}'} \sum_{\substack{Q\in\mathcal{D}_{R\textrm{-good}} \\ \ell(Q)\le \ell(R)}} \langle g_R\rangle_{R_u}T_{RQ}\langle f_Q\rangle_{Q_v}\bigg|\\
&\qquad\qquad\qquad\lesssim \bigg\|\sum_{k=-\infty}^\infty \epsilon_k g_{k}\bigg\|_{L^{q}(\mathbf{P}\otimes\mu;X^*)} 
\bigg\|\sum_{k=-\infty}^\infty \epsilon_k f_{k}\bigg\|_{L^p(\mathbf{P}\otimes \mu;X)}.
\end{split}
\end{equation}
\end{lem}

The proof of this lemma proceeds as the proof of Lemma \ref{tut} with appropriate modifications.
The key fact is that the operators 
\[
\tilde t_{RQ} = 2^{(n+j)\alpha/4} \mu(S) \frac{T_{RQ}}{\mu(R_u)\mu(Q_v)}\in \mathcal{L}(X),
\]
where the parameters are clear from the context, belong to an $\mathcal{R}$-bounded
family.
This follows from the normalizations and the condition \eqref{ades_operator}.

Then we can proceed to Section \ref{separated}, where the goal
is to prove a counterpart of Proposition \ref{ktest} under the
assumptions of Theorem \ref{mainth_operator}. For this purpose,
we need the following counterpart of Lemma \ref{tokaa}.

\begin{lem}\label{tokaa_operator}
Suppose
that for every pair of cubes $Q\in\mathcal{D}$ and $R\in\mathcal{D}'$, satisfying
$\ell(Q)\le \ell(R)\wedge \dist(Q,R)$, we are given functions
$\phi_Q,\psi_R\in L^1(\R^N,\mu;\C)$ such that 
$\mathrm{supp}(\phi_Q)\subset Q$, $\mathrm{supp}(\psi_R)\subset R$, and
\[
\int \phi_Q\,d\mu=0.
\]
Then 
\[
\mathcal{R}\big(\{\sigma_{RQ}\langle \psi_R,T\phi_Q\rangle\in\mathcal{L}(X)\,:\,\ell(Q)\le \ell(R)\wedge \dist(Q,R)\}\big)\le 1,
\]
where the normalizing factors are given by
\[
\sigma_{RQ}:=\frac{\dist(Q,R)^{d+\alpha}}{\ell(Q)^\alpha\|\phi_Q\|_{L^1(\mu)}\|\psi_R\|_{L^1(\mu)}}.
\]
\end{lem}

\begin{proof}
Suppose that $Q\in\mathcal{D}$ and $R\in\mathcal{D}'$ satisfy $\ell(Q)\le \ell(R)\wedge \dist(Q,R)$.
Let $y_Q$ be the center of the cube $Q$. Denoting
\[
F(x,y):=\frac{|y-y_Q|^\alpha}{|x-y_Q|^{d+\alpha}}\phi_Q(y)\psi_R(x)\sigma_{RQ},
\]
we obtain
\[
\int_{\R^N} \int_{\R^N}Ê|F(x,y)|\,d\mu(y)\,d\mu(x)\le 1.
\]
 Hence, by denoting
\[
\mathcal{T}=\bigg\{
\frac{|x-y|^{d+\alpha}}{|y-y'|^\alpha} [K(x,y)-K(x,y')]\,:\,y,y',x\in\R^N,\,|y-y'|\le |x-y|/2\bigg\},
\]
we obtain
\begin{align*}
\sigma_{RQ}\langle \psi_R,T\phi_Q\rangle
&=\int_{\R^N}\int_{\R^N} K(x,y)\phi_Q(y)\psi_R(x)\sigma_{RQ}\,d\mu(y)\,d\mu(x)\\
&=\int_{\R^N}\int_{\R^N} [K(x,y)-K(x,y_Q)]\phi_Q(y)\psi_R(x)\sigma_{RQ}\,d\mu(y)\,d\mu(x)\\
&= \int_{\R^N}\int_{\R^N} \frac{|x-y_Q|^{d+\alpha}}{|y-y_Q|^\alpha} 
[K(x,y)-K(x,y_Q)]\,F(x,y)\,d\mu(y)\,d\mu(x)\\
&\in \overline{\mathrm{abs\, conv}}\,(\mathcal{T}).
\end{align*}
Here the closure is taken in the strong operator topology and
the absolute convex hull, denoted by ${\mathrm{abs\, conv}\,(\mathcal{T})}$, is
the set of all vectors of the form $\sum_{j=1}^k \lambda_j x_j$ with
$\sum_{j=1}^k |\lambda_j|\le 1$ and
$x_j\in\mathcal{T}$ for $j=1,2,\ldots,k$.
Since,
\[\mathcal{R}(\overline{\mathrm{abs\, conv}}\,(\mathcal{T})) = \mathcal{R}(\mathcal{T}),\]
it remains to use the second $\mathcal{R}$-boundedness estimate
in \eqref{kernel_R_bounds}
\end{proof}

Proceeding as in the proof of Proposition \ref{ktest}, and using
Lemma \ref{tokaa_operator} instead of Lemma \ref{tokaa}, we find that the $\mathcal{R}$-boundedness
estimate \eqref{ades_operator} holds for the family of operators 
in $\mathcal{L}(X)$ defined by the equation
\eqref{matrix_scalar}. Hence, after applying
Lemma \ref{tut_operator} instead of Lemma \ref{tut}, the proof of Proposition \ref{ktest} continues
as before.


\subsection*{Operator-valued paraproducts}
We proceed to Section \ref{paraproducts}.
Let us first indicate the modifications in the proof of the estimate \eqref{estpar}, the boundedness of the paraproduct.
The first one  comes in the proof of Lemma \ref{boundest}:
Theorem \ref{haa} and Proposition \ref{ylcarl} are used with $\mathrm{UMD}$ function lattice $Z$
instead of $\C$.

The step from 
\eqref{difficult_start} to
\eqref{Ies} is now established by the following lemma
and assumption $\mathcal{R}(\bar B_Z)\lesssim 1$.

\begin{lem} Suppose that $t>q\vee s$, where $X^*$ has cotype $s$. Then
\begin{equation}\label{vect_quantity}
\begin{split}
&\bigg\| \sum_{j\in \Z}Ê\epsilon_j^\star d_j E_j g\bigg\|_{L^q(\Omega^\star\times\R^N\times \Omega;X^*)}
\\&\lesssim \mathcal{R}(\bar B_Z)\cdot \|\{||d_j(\cdot)||_{L^t(\Omega;Z)}\}_{j\in\Z}\|_{\mathrm{Car}^t(\mathcal{D}')}\cdot \|g\|_{L^q(\R^N;X^*)},
\end{split}
\end{equation}
where $\epsilon^\star=\{\epsilon_j^\star\,:\,j\in\Z\}\in \Omega^\star$ are Rademacher random variables and
\[
d_j:\R^N\to L^t(\Omega;Z):x\mapsto \bigg(\epsilon \mapsto \sum_{R\in\mathcal{D}_j'} \sum_{\substack{Q\in\mathcal{D}\\S(Q)=R}} \epsilon_Q\pi_{Q,R^a}(x)\bigg)
\]
for functions $\pi_{Q,R^a}:\R^N\to Z$ that are determined by \eqref{pi_choice}.
\end{lem}

\begin{proof}
Note first that LHS\eqref{vect_quantity} can be written as
\begin{align*}
\bigg(\iiint_{\Omega^\star\times\R^N\times \Omega} \bigg|\sum_{j\in\Z} \epsilon_j^\star
\frac{d_j(x,\epsilon)}{|d_j(x,\epsilon)|_Z} |d_j(x,\epsilon)|_Z E_j g(x)\bigg|_{X^*}^q\,d\mathbf{P}(\epsilon)\,d\mu(x)\,d\mathbf{P}(\epsilon^\star)\bigg)^{1/q}.
\end{align*}
Using Fubini's theorem and the
fact that the closed unit ball of $Z$ is $\mathcal{R}$-bounded, we see that
 LHS\eqref{vect_quantity} can be bounded by a constant multiple of
\begin{align*}
\mathcal{R}(\bar B_Z)
\bigg(\iiint_{\Omega\times\R^N\times \Omega^\star} \bigg|\sum_{j\in\Z} \epsilon_j^\star
|d_j(x,\epsilon)|_Z E_j g(x)\bigg|_{X^*}^q\,d\mathbf{P}(\epsilon^\star)\,d\mu(x)\,d\mathbf{P}(\epsilon)\bigg)^{1/q}.
\end{align*}
Recall that $t>q$.
Using Fubini's theorem, followed by the H\"older's inequality,
we find that LHS\eqref{vect_quantity} is bounded by a constant multiple of
\begin{align*}
&\mathcal{R}(\bar B_Z)\bigg(\iint_{\Omega^{\star}\times\R^N} \bigg\|\sum_{j\in\Z} \epsilon_j^\star
|d_j(x)|_Z E_j g(x)\bigg\|_{L^t(\Omega;X^*)}^q\,d\mu(x)\,d\mathbf{P}(\epsilon^\star)\bigg)^{1/q}\\
&=\mathcal{R}(\bar B_Z)\bigg\| \sum_{j\in\Z} \epsilon_j^\star |d_j(\cdot)|_Z E_j g\bigg\|_{L^q(\Omega^\star\timesÊ\R^N;L^t(\Omega;X^*))}.
\end{align*}
Let us denote $\tilde d_j(x,\epsilon):=|d_j(x,\epsilon)|_Z$. Then, for
a fixed $x\in\R^N$,
\begin{align*}
\|\tilde d_j(x)\|_{L^t(\Omega;\C)} = \bigg(\int_\Omega |\tilde d_j(x,\epsilon)|^t\,d\mathbf{P}(\epsilon)
\bigg)^{1/t} = \|d_j(x)\|_{L^t(\Omega;Z)}.
\end{align*}
Hence, by using Lemma \ref{genemb}, we can conclude
that the estimate \eqref{vect_quantity} holds.
\end{proof}

In order to estimate
the right hand side of \eqref{cotype_cont}, we use the fact that $L^2(\Omega,Z)$ has cotype $2$ since
$Z$ has it. 
The described modifications suffice for obtaining estimate \eqref{estpar} in the context
of Theorem \ref{mainth_operator}.
Finally, in the proof of estimate \eqref{estpar2} we use
Theorem \ref{haa}, with $\mathrm{UMD}$ function lattice $Z$, and the fact that
the family $\{E_k\}_{k\in\Z}$ of operators
in $L^q(\mu;Z)$
is $\mathcal{R}$-bounded  by the $\mathrm{UMD}$-valued 
Stein's inequality \cite{bourgain}.

This concludes the description of modifications in Section \ref{paraproducts}.

\bibliographystyle{plain}
\bibliography{nonhomog}

%
%
%
%
%
%
%
%

\end{document}